\documentclass[12pt]{article}
\usepackage[utf8]{inputenc}
\usepackage{amsmath,amscd}
\usepackage{mathtools}
\usepackage[normalem]{ulem}

\usepackage{thm-restate}

\usepackage{times}

\usepackage{amsthm}
\usepackage[margin=1in]{geometry}
\usepackage{color}
\usepackage{url}
\usepackage{framed}

\usepackage{amssymb}
\usepackage{tikz-cd} 
\usepackage{amsfonts}

\usepackage{thmtools} 
\usepackage[colorlinks]{hyperref}
\hypersetup{colorlinks,linkcolor={red},citecolor={blue},urlcolor={red}}
\usepackage{cleveref}
\definecolor{darkblue}{rgb}{0.0, 0.0, 0.8}
\definecolor{darkred}{rgb}{0.8, 0.0, 0.0}
\definecolor{darkgreen}{rgb}{0.0, 0.8, 0.0}

\newcommand{\eps}{\varepsilon}

\newcommand{\diam}{\mathrm{diam}}

\newcommand{\R}{\mathbb{R}}

\newcommand{\dis}{\mathrm{dis}}

\newcommand{\supp}{\mathrm{supp}}

\newcommand{\dgh}{d_\mathcal{GH}}

\newcommand{\dgws}[1]{d_{\mathcal{GW},{#1}}^\mathrm{S}}

\newcommand{\dgw}[1]{d_{\mathcal{GW},{#1}}}

\newcommand{\dW}[1]{d_{\mathcal{W},{#1}}}

\newcommand{\dH}{d_\mathcal{H}}

\newcommand{\lc}{\left(}
\newcommand{\rc}{\right)}
\newcommand{\ls}{\left|}
\newcommand{\rs}{\right|}

\newcommand{\ms}{\mathcal{M}}

\newcommand{\ws}{\mathcal{M}^w}

\newcommand{\X}{\mathcal{X}}
\newcommand{\Y}{\mathcal{Y}}

\newtheorem{theorem}{Theorem}[section]

\newtheorem{proposition}[theorem]{Proposition}

\newtheorem{coro}[theorem]{Corollary}

\newtheorem{lemma}[theorem]{Lemma}
\newtheorem{remark}[theorem]{Remark}
\newtheorem{example}[theorem]{Example}
\newtheorem{claim}[theorem]{Claim}

\newtheorem{definition}[theorem]{Definition}

\newtheorem{conjecture}{Conjecture}

\providecommand{\keywords}[1]
{
  \small	
  \textbf{\textit{Keywords---}} #1
}

\begin{document}
\title{Characterization of Gromov-type geodesics}

\author{Facundo M\'emoli\thanks{Dept. of Mathematics and Dept. of Computer Science and Engineering, The Ohio State University, Columbus, OH, USA. \url{memoli@math.osu.edu}.} 
\and Zhengchao Wan\thanks{Dept. of Mathematics, The Ohio State University, Columbus, OH, USA. \url{wan.252@osu.edu}}}
\date{}
\maketitle

\begin{abstract}
    It is well known that, when endowed with the Gromov-Hausdorff distance $d_\mathcal{GH}$, the collection $\mathcal{M}$ of all isometry classes of compact metric spaces  is a {complete and separable} space. It is also known that $(\mathcal{M},d_\mathcal{GH})$ is a geodesic metric space, but there is no known {structural} characterization of geodesics in $\mathcal{M}$. 
    
    In this paper we provide two characterizations of geodesics in $\mathcal{M}$. We call a Gromov-Hausdorff geodesic $\gamma:[0,1]\rightarrow\mathcal{M}$ \emph{Hausdorff-realizable} if there exists a compact metric space $Z$ containing isometric copies of $\gamma\left(t\right)$ for each $t\in[0,1]$ such that the Hausdorff distance satisfies $d_\mathcal{H}^Z\left(\gamma\left(s\right),\gamma\left(t\right)\right)=d_\mathcal{GH}\left(\gamma\left(s\right),\gamma\left(t\right)\right)$ for all $s,t\in[0,1]$. In this way, $\gamma$ is actually a geodesic in the Hausdorff hyperspace of $Z$, and we call it a Hausdorff geodesic. We prove that in fact \emph{every} Gromov-Hausdorff geodesic is Hausdorff-realizable. Inspired by this characterization, we further elucidate a structural connection between Hausdorff geodesics and Wasserstein geodesics: we show that every Hausdorff geodesic is equivalent to a so-called \emph{Hausdorff displacement interpolation}. This equivalence allows us to establish that every Gromov-Hausdorff geodesic is \emph{dynamic}, a notion which we develop in analogy with dynamic optimal couplings in the theory of optimal transport. 
    
    Besides geodesics in $\mathcal{M}$, we also study geodesics on the collection $\mathcal{M}^w$ of isomorphism classes of compact metric measure spaces. Sturm constructed a family of Gromov-type distances on $\mathcal{M}^w$, which we denote $d_{\mathcal{GW},{p}}^\mathrm{S}$ (for $p\in[1,\infty)$), as an analogue of $d_\mathcal{GH}$, and proved that $\left(\mathcal{M}^w,d_{\mathcal{GW},{p}}^\mathrm{S}\right)$ is also a geodesic space. We define a notion of \emph{Wasserstein-realizable} $d_{\mathcal{GW},{p}}^\mathrm{S}$ geodesics in a sense similar to Hausdorff-realizable geodesics and show that the set of all \emph{Wasserstein-realizable} $d_{\mathcal{GW},{p}}^\mathrm{S}$ geodesics is dense in the set of all $d_{\mathcal{GW},{p}}^\mathrm{S}$ geodesics. We further identify a rich class of $d_{\mathcal{GW},{p}}^\mathrm{S}$ geodesics which are Wasserstein-realizable. 
\end{abstract}


\newpage

\tableofcontents

\newpage
\section{Introduction}
The Gromov-Hausdorff distance $\dgh$ was independently introduced by Edwards \cite{edwards1975structure} and Gromov \cite{gromov1981groups} in order to quantify the difference between two given metric spaces. Let $(X,d_X)$ and $(Y,d_Y)$ be two compact metric spaces. Then, the Gromov-Hausdorff distance between them is defined as follows:
\begin{equation}\label{eq:dgh}
    \dgh\left(X,Y\right)\coloneqq\inf_{Z}\dH^Z\left(\varphi_X\left(X\right),\varphi_Y\left(Y\right)\right),
\end{equation}
where $\dH^Z$ denotes the Hausdorff distance between nonempty closed subsets of $Z$ (cf. \Cref{def:dH}) and the infimum is taken over all metric spaces $Z$ and isometric embeddings $\varphi_X:X\hookrightarrow Z$ and $\varphi_Y:Y\hookrightarrow Z$.
Let $\ms$ denote the collection of isometry classes of compact metric spaces. In this regard, it has been established that when endowed with $\dgh$, $\left(\ms,\dgh\right)$ is a complete and separable metric space \cite{burago2001course,petersen2006riemannian}. The geometry of $\left(\ms,\dgh\right)$ has been studied extensively recently \cite{ivanov2017local,ivanov2019isometry,klibus2018convexity}.

A metric measure space $\mathcal{X}=(X,d_X,\mu_X)$ is a metric space $(X,d_X)$ endowed with a Borel probability measure $\mu_X$. Denote by $\ws$ the collection of isomorphism classes (cf. \Cref{def:isomorphism}) of compact metric measure spaces with full support. By replacing the Hausdorff distance in \Cref{eq:dgh} with the $\ell^p$-Wasserstein distance (cf. \Cref{def:p-w-dist}) between probability measures, Sturm introduced the \emph{$L^p$-transportation distance} on $\ws$ as a counterpart to $\dgh$ \cite{sturm2006geometry,sturm2012space} and in this paper, we denote by $\dgws{p}$ (for each $p\in[1,\infty]$) his $L^p$-transportation distance. In \cite{memoli2007use,memoli2011gromov} the first author introduced the Gromov-Wasserstein distance $\dgw{p}$ (for each $p\in[1,\infty]$) on $\ws$ based on an alternative representation of $\dgh$ (cf. \Cref{thm:dgh-dual}). $\dgws{p}$ and $\dgw{p}$ are both legitimate metrics on $\ws$ which generate the same topology, yet they do not coincide in general.

In recent years, the Gromov-Hausdorff and the Gromov-Wasserstein distances have found applications in shape analysis \cite{memoli2004comparing,memoli2005theoretical,memoli2007use,bronstein2008numerical,bronstein2010gromov}, machine learning \cite{peyre2019computational,alvarez2018gromov,titouan2019optimal,bunne2019learning,chowdhury2020generalize,xu2019scalable,vayer2019sliced,le2019fast}, biology \cite{liebscher2018new,demetci2020gromov, cao2020manifold}, and network analysis \cite{hendrikson2016using,chowdhury2019gromov, chowdhury2019metric}. It is worth noting that the three distances $\dgh$, $\dgws{p}$ and $\dgw{p}$ and their variants have been applied in topological data analysis \cite{chazal2009gromov,chazal2014persistence,blumberg2014robust,memoli2019quantitative,chowdhury2019gromov,blumberg2020stability,rolle2020stable} to establish stability results of invariants. As such, clarifying the properties of these distances, especially with respect  to the structure of their associated geodesics, may help develop suitable concomitant statistical methods  (see \cite{sturm2012space,chowdhury2020gromov}).

\subsection{Geodesic properties of $(\ms,\dgh)$,  $\left(\ws,\dgws{p}\right)$  and  $\left(\ws,\dgw{p}\right)$}
The fact that $\left(\ws,\dgw{p}\right)$ is a geodesic space for each $p\in[1,\infty)$ was proved by K.T. Sturm in \cite[Theorem 3.1]{sturm2012space}.\footnote{Sturm proved the geodesic property for the collection $\ws_p$ larger than $\ws$ for each $p\in[1,\infty]$ which contains complete metric measure spaces with finite $\ell^p$-size, a generalization of the notion of diameter. However, the same technique applies to the case of $\ws$ without any changes.} His proof is constructive in that a special type of geodesics, which in this paper we call \emph{straight-line $\dgw p$ geodesics}, was identified for any two given metric measure spaces (cf. \cite[Theorem 3.1]{sturm2012space}). Furthermore, Sturm provided a complete characterization of geodesics in $(\ws,\dgw p)$ for each $p\in(1,\infty)$ by proving that \emph{every} geodesic in $\left(\ws,\dgw{p}\right)$ is a straight-line $\dgw p$ geodesic.

The geodesic property of $(\ms,\dgh)$ was first proved in \cite{ivanov2016gromov} via the so called \emph{mid-point criterion} (cf. \Cref{thm:mid-pt-geo}) which did not yield explicit geodesics. Later in \cite{chowdhury2018explicit}, the fact that $(\ms,\dgh)$ is geodesic was reproved by identifying the \emph{straight-line $\dgh$ geodesics} (cf. \Cref{thm:str-line-geo}) which are analogous to the straight-line $\dgw p$ geodesics constructed in \cite{sturm2012space}. In {\cite{sturm2020email} Sturm proved that $\left(\ws,\dgws p\right)$ (for each $p\in[1,\infty)$) is geodesic by constructing what we call the \emph{straight-line $\dgws p$ geodesic} between any two metric measure spaces, which is a slight variant of the straight-line $\dgw p$ geodesic with respect to $\dgw p$ (cf. \Cref{thm:straight-line-gw-geo}).} 

Unlike the case of $\left(\ws,\dgw{p}\right)$ where the only geodesics are the straight-line $\dgw p$ geodesics, neither straight-line $\dgh$ geodesics nor straight-line $\dgws p$ geodesics completely characterize geodesics in $(\ms,\dgh)$ and $\left(\ws,\dgws p\right)$, respectively. Indeed, the authors of \cite{chowdhury2018explicit} discovered \emph{deviant} geodesics in $(\ms,\dgh)$, i.e., geodesics which are not straight-line $\dgh$ geodesics. Following a strategy similar to the one used in \cite{chowdhury2018explicit}, we discover non-straight-line $\dgws p$ geodesics in $\left(\ws,\dgws p\right)$ and present our construction in \Cref{app:d-b-geodesic} of this paper. This inspires us to obtain better understanding of geodesics in $(\ms,\dgh)$ and $\left(\ws,\dgws{p}\right)$,

\subsection{Our results}

In this paper we elucidate characterization results for geodesics in $(\ms,\dgh)$ and $\left(\ws,\dgws{p}\right)$. These characterizations accommodate both straight-line $\dgh$ / $\dgws p$ geodesics and the above mentioned deviant geodesics.

\paragraph{Hausdorff-realizable Gromov-Hausdorff geodesics.} {In this paper, the terms ``Gromov-Hausdorff geodesics'' and ``$\dgh$ geodesics'' are used interchangeably when referring to geodesics in $(\ms,\dgh)$.}

Given a metric space $X$, its Hausdorff hyperspace $\mathcal{H}\left(X\right)$ is the set of all nonempty bounded closed subsets of $X$ endowed with the Hausdorff distance $\dH^X$ (cf. \Cref{def:dH}). Blaschke's compactness theorem (see for example \cite[Theorem 7.3.8]{burago2001course}) states that $\mathcal{H}\left(X\right)$ is compact whenever $X$ is compact. Furthermore, it is known that if $X$ is a compact geodesic space then so is $\mathcal{H}\left(X\right)$ \cite{bryant1970convexity,serra1998hausdorff}. We call a geodesic in the Hausdorff hyperspace of some metric space a \emph{Hausdorff geodesic}. Hausdorff geodesics are easier to study than Gromov-Hausdorff geodesics since the definition of $\dgh$ relies on finding the infimum of Hausdorff distances over certain metric embeddings (cf. \Cref{def:dGH}). This inspires us to relate Gromov-Hausdorff geodesics with Hausdorff geodesics.

It was first observed and proved in \cite{ivanov2019hausdorff} that every straight-line $\dgh$ geodesic can be realized as a Hausdorff geodesic with respect to some ambient metric space (cf. \Cref{prop:hausdorff-geodesic-straight-line}). We further show that under mild conditions, such a metric space construction is actually compact and thus the corresponding straight-line $\dgh$ geodesic can be realized as a geodesic in the Hausdorff hyperspace of certain compact metric space (cf. \Cref{prop:strline-hausdorff}). We call any such a Gromov-Hausdorff geodesic \emph{Hausdorff-realizable}. It turns out that Hausdorff-realizability is a universal phenomenon:

\begin{restatable}{thm}{thmhreal}\label{thm:main-h-realizable}
Every Gromov-Hausdorff geodesic is Hausdorff-realizable.
\end{restatable}

\paragraph{Wasserstein-realizable Gromov-Wasserstein geodesics.} Given a metric space $X$, the $\ell^p$-\emph{Wasserstein hyperspace} $\mathcal{W}_p\left(X\right)$ is the set of all Borel probability measures on $X$ with finite moments and endowed with the \emph{$\ell^p$-Wasserstein distance} $d_{\mathcal{W},p}$ (cf. \Cref{def:p-w-dist}). Sturm's $L^p$-transportation distance $\dgws{p}$ is defined by replacing the Hausdorff distance term in the definition of $\dgh$ with an $\ell^p$-Wasserstein distance term (cf. \Cref{def:dGW}). For the sake of symmetry of nomenclature, we call the $L^p$-transportation distance $\dgws{p}$ the \emph{Sturm's $\ell^p$-Gromov-Wasserstein distance}. Since we only focus on $\dgws{p}$ instead of $\dgw{p}$ in this paper, we also simply call $\dgws{p}$ the $\ell^p$-Gromov-Wasserstein distance without causing any confusion. See also \Cref{fig:gh vs gw} for an illustration. {Then, in this paper, the terms ``$\ell^p$-Gromov-Wasserstein geodesics'' and ``$\dgws{p}$ geodesics'' are used interchangeably when referring to geodesics in $\lc\ws,\dgws p\rc$.}

\begin{figure}[htb]
	\centering		\includegraphics[width=0.2\textwidth]{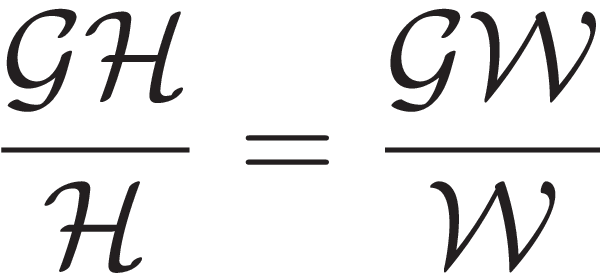}
	\caption{{\textbf{Nomenclature.} The Gromov-Hausdorff distance between two metric spaces is defined via infimizing Hausdorff distances on certain ambient spaces (cf. \Cref{def:dGH}). This procedure of infimizing some quantities over certain ambient spaces is called ``gromovization" in \cite{memoli2011gromov}, e.g., the Gromov-Hausdorff distance is the gromovization of the Hausdorff distance. In this way, Sturm's $L^p$-transportation distance is the gromovization of the Wasserstein distance and hence we call it the $\ell^p$-Gromov-Wasserstein distance in our paper. The figure illustrates the respective gromovization processes for $\dgh$ and $\dgws p$.}} \label{fig:gh vs gw}
\end{figure}

Inspired by \Cref{thm:main-h-realizable}, we analogously define the so-called $\ell^p$-Wasserstein-realizable geodesics in $\lc \ws,\dgws p\rc$ (cf. \Cref{def:w-real-geo}). {For $p\in[1,\infty)$, denote $\Gamma^p$ the collection of all $\dgws{p}$ geodesics. Let $d_{\infty,p}$ be the uniform metric on $\Gamma^p$, i.e., 
$$d_{\infty,p}\left(\gamma_1,\gamma_2\right)\coloneqq\sup_{t\in[0,1]}\dgws{p}\left(\gamma_1\left(t\right),\gamma_2\left(t\right)\right)$$
for any $\gamma_1,\gamma_2\in\Gamma^p$. Let $\Gamma_\mathcal{W}^p$ denote the subset of $\Gamma^p$ consisting of all Wasserstein-realizable geodesics. Then, by applying  techniques and strategies similar to those used in proving \Cref{thm:main-h-realizable}, we obtain that $\Gamma^p_\mathcal{W}$ is a dense subset of $\Gamma^p$ (cf. \Cref{prop:main-w-real-dense}).
}

Though we conjecture that every $\dgws{p}$ geodesic is Wasserstein-realizable, we have not been able to establish the closedness of $\Gamma_\mathcal{W}^p$ and the conjecture still remains open. We then turn to consider geodesics satisfying certain constraints and eventually, in \Cref{thm:bdd-geo-w-real}, we are able to identify a certain type of $\dgws{p}$ geodesics called \emph{Hausdorff-bounded} (cf. \Cref{def:hausdorff-bdd}) which turn out to be Wasserstein-realizable. For any two metric measure spaces $\gamma(s)=(X_s,d_s,\mu_s)$ and $\gamma(t)=(X_t,d_t,\mu_t)$ along a given Hausdorff-bounded geodesic $\gamma$, the Gromov-Hausdorff distance $\dgh(X_s,X_t)$ between their underlying metric spaces is bounded above by the Gromov-Wasserstein distance $f\lc\dgws{p}(\gamma(s),\gamma(t))\rc$ through some suitable function $f$. Such control over Gromov-Hausdorff distances allows us to exploit the techniques used for proving \Cref{thm:main-h-realizable} in order to also establish Wasserstein-realizability.

\begin{restatable}{thm}{thmbddgeo}\label{thm:bdd-geo-w-real}
Given $p\in[1,\infty)$, every Hausdorff-bounded $\ell^p$-Gromov-Wasserstein geodesic is $\ell^p$-Wasserstein-realizable.
\end{restatable}

It turns out that both the straight-line $\dgws p$ geodesics introduced in \cite{sturm2020email} and the deviant geodesics constructed in \Cref{app:d-b-geodesic} are Hausdorff-bounded and thus Wasserstein-realizable (cf. \Cref{prop:strline-gw-h-bdd} and \Cref{rmk:deviant and branching GW}). Then, \Cref{thm:bdd-geo-w-real} indeed characterizes a large class of Gromov-Wasserstein geodesics as Wasserstein geodesics.

\paragraph{Dynamic (Gromov-)Hausdorff geodesics.} Since every Gromov-Hausdorff geodesic is a Hausdorff geodesic, we turn to study properties of Hausdorff geodesics. In the last part of the paper, we devote to draw connection between Hausdorff geodesics and Wasserstein geodesics. In particular, we establish a counterpart to the theory of displacement interpolation for Hausdorff geodesics.

For $p>1$, a geodesic $\gamma:[0,1]\rightarrow \mathcal{W}_p(X)$ of probability measures in $\mathcal{W}_p\left(X\right)$ is also called a \emph{displacement interpolation} since it is characterized by a probability measure on the set of all geodesics in $X$ itself (see \cite[Chapter 7]{villani2008optimal} and \Cref{thm:dis-int} for a precise statement). More precisely, let $\Gamma([0,1],X)$ denote the set of all geodesics in $X$, then there exists a probability measure $\Pi\in\mathcal{P}(\Gamma([0,1],X))$ such that $\gamma(t)=(e_t)_\#(\Pi)$ for each $t\in[0,1]$, where $e_t:\Gamma([0,1],X)\rightarrow X$ is the evaluation map at $t$ sending any function $\gamma:[0,1]\rightarrow X$ to $\gamma(t)$.

Now, we state an analogous characterization for Hausdorff geodesics. Let $X\in\ms$ and $A,B\subseteq X$ be two closed subsets. Let $\rho\coloneqq \dH^X\left(A,B\right)>0$. We define
$$\mathfrak{L}\left(A,B\right)\coloneqq\left\{\gamma:[0,1]\rightarrow X:\,\gamma(0)\in A,\,\gamma(1)\in B\text{ and }\forall s,t\in[0,1],\,d_X\left(\gamma(s),\gamma(t)\right)\leq|s-t|\,\rho\right\}. $$
{In other words, $\mathfrak{L}\left(A,B\right)$ is the set of all $\rho$-Lipschitz curves (cf. \Cref{sec:geo}) in $X$ starting in $A$ and ending in $B$.}
We call a closed subset $\mathfrak{D}\subseteq\mathfrak{L}\left(A,B\right)$ a \emph{Hausdorff displacement interpolation} between $A$ and $B$ if $e_0\left(\mathfrak{D}\right)=A$ and $e_1\left(\mathfrak{D}\right)=B$, where $e_0$ and $e_1$ are evaluation maps at $t=0$ and $t=1$, respectively. Then, we have the following characterization of Hausdorff geodesics via the Hausdorff displacement interpolation:

\begin{restatable}{thm}{thmdynhaus}\label{thm:main-dyn-hausdorff}
Given a compact metric space $X$, let $\gamma:[0,1]\rightarrow \mathcal{H}\left(X\right)$ be any map. We assume that $\rho\coloneqq\dH^X\left(\gamma\left(0\right),\gamma\left(1\right)\right)>0$. Then, the following two statements are equivalent:
\begin{enumerate}
    \item $\gamma$ is a Hausdorff geodesic;
    \item there exists a \emph{nonempty} closed subset $\mathfrak{D}\subseteq\mathfrak{L}\left(\gamma\left(0\right),\gamma\left(1\right)\right) $ such that $\gamma\left(t\right)=e_t\left(\mathfrak{D}\right)$ for all $t\in[0,1]$.
\end{enumerate}
\end{restatable}
As in the theory of optimal transport where people accept that ``a geodesic in the space of laws is the law of a geodesic'' \cite[page 126]{villani2008optimal}, \emph{a geodesic in the space of closed subsets is the closed subset of a certain set of Lipschitz curves}. {It is tempting to ask whether one can require $\mathfrak{D}$ to contain only geodesics instead of Lipschitz curves. There are, however, counterexamples to this; see \Cref{rmk:couterex-haus-geo}.} As an application of \Cref{thm:main-dyn-hausdorff}, in \Cref{thm:exist infinite geod} we prove the existence of \emph{infinitely} many distinct Gromov-Hausdorff geodesics connecting any two compact metric spaces.

We further extend our characterization of Hausdorff geodesics via displacement interpolation to Gromov-Hausdorff geodesics. Though defined via Hausdorff distances, there is a dual formula for $\dgh$ which involves \emph{correspondences} between sets (cf. \Cref{sec:gh-detail}). This notion is akin to the notion of coupling between probability measures which are inherent to the Wasserstein distance (see \Cref{rmk:relation-dyn-coup-corr} for a more detailed comparison). We extend the notion of correspondence to the so-called \emph{dynamic optimal correspondence} (cf. \Cref{def:dyn-cor}), which is a concept analogous to \emph{dynamic optimal couplings} (cf. \Cref{def:dyn-coup}) in the theory of measure displacement interpolation. We say that a geodesic $\gamma$ in $\ms$ is \emph{dynamic} if it possesses a dynamic optimal correspondence.

{The above facts are put together as follows: via \Cref{thm:main-h-realizable} we identify any given Gromov-Hausdorff geodesic with a Hausdorff geodesic; then we invoke \Cref{thm:main-dyn-hausdorff} to generate a Hausdorff displacement interpolation, which gives rise to a dynamic optimal correspondence. In the end, we obtain the following characterization of Gromov-Hausdorff geodesics:}

\begin{restatable}{thm}{thmdyngh}\label{thm:main-dyn-gh}
Every Gromov-Hausdorff geodesic $\gamma:[0,1]\rightarrow \ms$ is dynamic.
\end{restatable}

\paragraph{Miscellaneous results.} In the course of proving the main results described above, we established various supporting results and provided novel proofs of some known results both of which are of independent interest. Here we list some of such results:
\begin{enumerate}
    \item In \Cref{thm:hyper-geod-iff} we prove that a compact metric space is geodesic \textit{if and only if} its Hausdorff hyperspace is geodesic.
    \item In \Cref{ex:geod-sphere} we reconstruct the Gromov-Hausdorff geodesic defined in \cite{chowdhury2018explicit} between $\mathbb S^0$ and $\mathbb S^n$ via a Hausdorff geodesic construction.
    \item In \Cref{thm:W-1-lip} we prove that the Wasserstein extensor $\mathcal{W}_p$ for each $p\in[1,\infty]$ is 1-Lipschitz.
    \item In page \pageref{succinct proof} we provide a {novel} succinct proof of the fact that $(\ms,\dgh)$ is geodesic.
    \item In \Cref{app:d-b-geodesic} we construct examples of deviant and branching $\dgws p$ geodesics.
\end{enumerate}

\subsection{Organization of the paper} 
In \Cref{sec:background}, we collect necessary background material regarding real functions, geodesics, the Gromov-Hausdorff distance and the Gromov-Wasserstein distance. In \Cref{sec:metric-ext}, we further examine properties of metric extensions including Hausdorff hyperspaces, Wasserstein hyperspaces and the Urysohn universal metric space. In \Cref{sec:real-geo}, we study Hausdorff-realizable and Wasserstein-realizable geodesics and prove \Cref{thm:main-h-realizable} and \Cref{thm:bdd-geo-w-real}. In \Cref{sec:dyn-geo}, we study the Hausdorff displacement interpolation and dynamic Gromov-Hausdorff geodesics and we prove \Cref{thm:main-dyn-hausdorff} and \Cref{thm:main-dyn-gh}. In \Cref{sec:discussion}, we discuss some open problems and in the Appendix, we provide extra proofs and constructions.

\section{Background material}\label{sec:background}
In this section, we first collect some notions and basic results for real functions. Then, we provide necessary background materials regarding continuous curves and geodesics in metric spaces. We also introduce definitions and certain results of both the Gromov-Hausdorff distance and the Gromov-Wasserstein distance.
\subsection{Elementary properties of real functions}\label{sec:real function}

{In this section we collect some notions and basic results of real functions which we will use in \Cref{sec:W-reall-geo}.}

\begin{definition}[Proper functions]\label{def:proper-function}
For any function $f:I\rightarrow J$ where $I$ and $J$ are intervals in $\overline{\mathbb
R}\coloneqq[0,\infty]$ containing 0, we say $f$ is \emph{proper} if $f(0)=0$ and $f$ is continuous at $0$.
\end{definition}

For any increasing real function $f:[0,\infty)\rightarrow [0,\infty)$, we define its \emph{inverse} $f^{-1}:[0,\infty)\rightarrow[0,\infty]$ as follows:
$$f^{-1}(y)\coloneqq\inf\{x\geq 0:\,f(x)\geq y\},\quad\forall y\in [0,\infty), $$
where we adopt the convention that $\inf\emptyset=\infty$. Then, we have the following elementary properties of inverse of increasing functions:

\begin{proposition}[Some properties of inverse of increasing functions]\label{prop:increasing-prop}
Fix an increasing function $f:[0,\infty)\rightarrow[0,\infty)$. Then, we have the following:
\begin{enumerate}
    \item $f^{-1}:[0,\infty)\rightarrow[0,\infty]$ is still an increasing function; 
    \item if $f$ is unbounded, i.e., $\lim_{x\rightarrow\infty}f(x)=\infty$, then $f^{-1}(y)<\infty$ for any $y\in [0,\infty)$;
    \item for all $x,y\in[0,\infty)$, $x<f^{-1}(y)$ implies that $f(x)\leq y$;
    \item if $f$ is \emph{strictly} increasing, then for $x,y\in[0,\infty)$, $f(x)\leq y$ implies that $x\leq f^{-1}(y)$;
    \item if $f$ is proper, then for any $y>0$, $f^{-1}(y)>0$ and $f^{-1}(0)=0$. If $f$ is moreover \emph{strictly} increasing, then $f^{-1}$ is continuous at $0$ and in particular $f^{-1}$ is proper.
\end{enumerate}
\end{proposition}

\subsection{Curves and geodesics in metric spaces}\label{sec:geo}

A \emph{metric space} $(X,d_X)$ is a pair such that $X$ is a set and $d_X:X\times X\rightarrow\mathbb R_{\geq 0}$ is a function satisfying the following conditions:
\begin{enumerate}
    \item for any $x,x'\in X$, $d_X(x,x')\geq 0$ and the equality holds if and only if $x=x'$;
    \item for any $x,x'\in X$, $d_X(x,x')=d_X(x',x)$;
    \item for any $x,x',x''\in X$, $d_X(x,x')\leq d_X(x,x'')+d_X(x'',x')$.
\end{enumerate}
We often abbreviate $(X,d_X)$ to $X$ to represent a metric space.

A map $\varphi:X\rightarrow Y$ between metric spaces is called an \textit{isometric embedding}, usually denoted by $\varphi:X\hookrightarrow Y$, if for each $x,x'\in X$
$$d_Y(\varphi(x),\varphi(x'))=d_X(x,x'). $$
We say a metric space $X$ is \emph{isometric} to another metric space $Y$ if there exists a surjective isometric embedding $\varphi:X\hookrightarrow Y$. When $X$ is isometric to $Y$, we write $X\cong Y$.

Given a metric space $X$, a \textit{curve} in $X$ is any continuous map $\gamma:[0,1]\rightarrow X$. For $C>0$, a $C$-Lipschitz curve is any curve $\gamma$ such that $d_X(\gamma(s),\gamma(t))\leq C\cdot |s-t|$ for $s,t\in[0,1]$. One important result that we use in the sequel is the following variant of the Arzel\`a-Ascoli theorem (compare with the version given in \cite[Theorem 2.5.14]{burago2001course}). We omit the proof here since it is essentially the same as the one for \cite[Theorem 2.5.14]{burago2001course} and it is also a direct consequence of a more general statement which we prove later (cf. \Cref{thm:general-AA}).

\begin{theorem}[Arzel\`a-Ascoli theorem]\label{thm:AA}
Let $X$ be a compact metric space and let $\{\gamma_i:[0,1]\rightarrow X\}_{i=0}^\infty$ be a sequence of $C$-Lipschitz curves for a fixed $C>0$, i.e., $d_X\left(\gamma_i\left(s\right),\gamma_i\left(t\right)\right)\leq C\cdot|s-t|$ for any $s,t\in[0,1]$ and $i=0,1,\ldots$. Then, there is a uniformly convergent subsequence of $\{\gamma_i\}_{i=0}^\infty$ with a $C$-Lipschitz limit $\gamma:[0,1]\rightarrow X$.
\end{theorem}

There is one special type of Lipschitz curves called geodesics:

\begin{definition}[Geodesics]
A curve $\gamma:[0,1]\rightarrow X$ is called a \emph{geodesic} if for any $s,t\in[0,1]$ one has
$d_X\left(\gamma\left(s\right),\gamma\left(t\right)\right)\leq|t-s|\cdot d_X\left(\gamma\left(0\right),\gamma\left(1\right)\right),$ i.e., $\gamma$ is $d_X(\gamma(0),\gamma(1))$-Lipschitz.
\end{definition}
By the triangle inequality, it is clear that $d_X\left(\gamma\left(s\right),\gamma\left(t\right)\right)=|t-s|\cdot d_X\left(\gamma\left(0\right),\gamma\left(1\right)\right)$ for any $s,t\in[0,1]$. 

If for any $x_0,x_1\in X$, there exists a geodesic $\gamma:[0,1]\rightarrow X$ such that $\gamma(0)=x_0$ and $\gamma(1)=x_1$, we call $X$ a \emph{geodesic space}. The following is a useful criterion for checking whether a metric space is geodesic or not.

\begin{theorem}[Mid-point criterion,  {\cite[Theorem 2.4.16]{burago2001course}}]\label{thm:mid-pt-geo}
A \emph{complete} metric space $X$ is geodesic if and only if for any $x,y\in X$, there exists $z\in X$ (which we call a mid-point between $x$ and $y$) such that 
$$d_X(x,z)=d_X(y,z)=\frac{1}{2}d_X(x,y). $$
\end{theorem}

Concatenation is one way of constructing new geodesics from existing ones.

\begin{proposition}[Geodesic concatenation]\label{prop:geo-concatenate}
Let $X$ be a metric space and for $i=1,\ldots,n$ let $\gamma_i:[0,1]\rightarrow X$ be a geodesic. For $i=1,\ldots,n$, let $\rho_i\coloneqq  d_X(\gamma_i(0),\gamma_i(1))$. Assume that $\gamma_i(1)=\gamma_{i+1}(0)$ for $i=1,\ldots,n-1$ and $\rho_i>0$ for all $i=1,\ldots,n$. Let $\rho\coloneqq\sum_{i=1}^n\rho_i. $
Then, the curve $\gamma:[0,1]\rightarrow X$ defined as follows is a $\rho$-Lipschitz curve:
$$\gamma(t)\coloneqq\begin{cases}
\gamma_1\lc\frac{\rho}{\rho_1}t\rc, & t\in\left[0,\frac{\rho_1}{\rho}\right]\\
\gamma_2\lc\frac{\rho}{\rho_2}\lc t-\frac{\rho_1}{\rho}\rc\rc, & t\in \left(\frac{\rho_1}{\rho},\frac{\rho_1+\rho_2}{\rho}\right]\\
\cdots,&\cdots\\
\gamma_n\lc\frac{\rho}{\rho_n}\lc t-\sum_{i=1}^{n-1}\frac{\rho_i}{\rho}\rc\rc, & t\in \left(\sum_{i=1}^{n-1}\frac{\rho_i}{\rho},1\right]
\end{cases}$$
In particular, if $\rho=d_X(\gamma_1(0),\gamma_n(1))$, then $\gamma$ is a geodesic.
\end{proposition}

\begin{proof}
Given any $s,t\in[0,1]$, there are two cases to consider.
\begin{enumerate}
    \item There exists $k$ such that $s,t\in\left[\sum_{i=1}^{k-1}\frac{\rho_i}{\rho},\sum_{i=1}^{k}\frac{\rho_i}{\rho}\right]$. Then, 
    \begin{align*}
        d_X(\gamma(s),\gamma(t))&=d_X\lc\gamma_k\lc\frac{\rho}{\rho_k}\lc s-\sum_{i=1}^{k-1}\frac{\rho_i}{\rho}\rc\rc,\gamma_k\lc\frac{\rho}{\rho_k}\lc t-\sum_{i=1}^{k-1}\frac{\rho_i}{\rho}\rc\rc\rc\\
        &=\left|\frac{\rho}{\rho_k}\lc s-\sum_{i=1}^{k-1}\frac{\rho_i}{\rho}\rc-\frac{\rho}{\rho_k}\lc t-\sum_{i=1}^{k-1}\frac{\rho_i}{\rho}\rc\right| \rho_k\\
        &=|s-t|\rho.
    \end{align*}
            
    \item There exist $k,l>0$ such that $s\in\left[\sum_{i=1}^{k-1}\frac{\rho_i}{\rho},\sum_{i=1}^{k}\frac{\rho_i}{\rho}\right]$ and $t\in\left[\sum_{i=1}^{k+l-1}\frac{\rho_i}{\rho},\sum_{i=1}^{k+l}\frac{\rho_i}{\rho}\right]$. Then, by case 1, we have that
    $$d_X\lc\gamma(s),\gamma\lc\sum_{i=1}^{k}\frac{\rho_i}{\rho}\rc\rc=\left|s-\sum_{i=1}^{k}\frac{\rho_i}{\rho}\right|\rho, $$
    and
    $$d_X\lc\gamma(t),\gamma\lc\sum_{i=1}^{k+l-1}\frac{\rho_i}{\rho}\rc\rc=\left|t-\sum_{i=1}^{k+l-1}\frac{\rho_i}{\rho}\right|\rho, $$
    \begin{align*}
        d_X(\gamma(s),\gamma(t))&\leq d_X\lc\gamma(s),\gamma\lc\sum_{i=1}^{k}\frac{\rho_i}{\rho}\rc\rc+ d_X\lc\gamma(t),\gamma\lc\sum_{i=1}^{k+l-1}\frac{\rho_i}{\rho}\rc\rc\\
        &+\sum_{j=0}^{l-2} d_X\lc\gamma\lc\sum_{i=1}^{k+j}\frac{\rho_i}{\rho}\rc,\gamma\lc\sum_{i=1}^{k+j+1}\frac{\rho_i}{\rho}\rc\rc\\
        &=\lc\left|s-\sum_{i=1}^{k}\frac{\rho_i}{\rho}\right|+\left|t-\sum_{i=1}^{k+l-1}\frac{\rho_i}{\rho}\right|+ \sum_{j=0}^{l-2}\frac{\rho_{k+j+1}}{\rho}\rc \rho\\
        &=|t-s|\rho.
    \end{align*}
\end{enumerate}
Therefore, $\gamma$ is a $\rho$-Lipschitz curve.
\end{proof}


\subsection{Gromov-Hausdorff distance}\label{sec:gh-detail}

Given a metric space $X$, there is a well-known notion of distance between closed subsets of $X$: the \emph{Hausdorff distance}.

\begin{definition}[Hausdorff distance]\label{def:dH}
For nonempty closed subsets $A,B\subseteq X$, the Hausdorff distance $\dH^X$ between them is defined by
$$\dH^X\left(A,B\right)\coloneqq\inf\{r:\,A\subseteq B^r,\,B\subseteq A^r\}, $$
where $A^r\coloneqq\{x\in X:\,d_X\left(x,A\right)\leq r\}$ is called the $r$-thickening of $A$.
\end{definition}

\begin{remark}\label{rmk:alternative dH}
It is easy to see that $\dH^X\left(A,B\right)=\max\lc\sup_{x\in A}\inf_{y\in B} d_X(x,y),\sup_{y\in B}\inf_{x\in A} d_X(x,y)\rc$ (cf. \cite[Exercise 7.3.2]{burago2001course}). This formula is also sometimes given as the definition of the Hausdorff distance.
\end{remark}

\begin{lemma}[Hausdorff distance under isometric embedding]\label{lm:dH under embedding}
Let $\varphi:X\hookrightarrow Y$ be an isometric embedding of two compact metric spaces. Let $A,B$ be nonempty closed subsets of $X$. Then, we have that
\[\dH^X(A,B)=\dH^Y(\varphi(A),\varphi(B)).\]
\end{lemma}
\begin{proof}
Since $\varphi$ is an isometric embedding, by \Cref{rmk:alternative dH} we have that 
\begin{align*}
 \dH^X\left(A,B\right)&=\max\lc\sup_{x\in A}\inf_{y\in B} d_X(x,y),\sup_{y\in B}\inf_{x\in A} d_X(x,y)\rc   \\
 &=\max\lc\sup_{x\in A}\inf_{y\in B} d_Y(\varphi(x),\varphi(y)),\sup_{y\in B}\inf_{x\in A} d_Y(\varphi(x),\varphi(y))\rc\\
 &=\max\lc\sup_{x\in \varphi(A)}\inf_{y\in \varphi(B)} d_Y(x,y),\sup_{y\in \varphi(B)}\inf_{x\in\varphi(A)} d_Y(x,y)\rc\\
 &=\dH^Y(\varphi(A),\varphi(B)).
\end{align*}
\end{proof}

Now we recall the definition of the Gromov-Hausdorff distance defined in \Cref{eq:dgh} as follows:

\begin{definition}[Gromov-Hausdorff distance]\label{def:dGH}
Given two metric spaces $X$ and $Y$, the Gromov-Hausdorff distance between them is defined by
$$\dgh\left(X,Y\right)\coloneqq\inf_{Z}\dH^Z\left(\varphi_X\left(X\right),\varphi_Y\left(Y\right)\right), $$
where the infimum is taken over all metric spaces $Z$ and isometric embeddings $\varphi_X:X\hookrightarrow Z$ and $\varphi_Y:Y\hookrightarrow Z$.
\end{definition}

If $X$ and $Y$ are compact, then the infimum can be restricted to only compact metric spaces $Z$. 

For two compact metric spaces $X$ and $Y$, $\dgh(X,Y)=0$ if and only if $X\cong Y$. Recall that $\ms$ denote the set of all isometry classes of compact metric spaces. Then, $\left(\ms,\dgh\right)$ is a metric space. Moreover, $\left(\ms,\dgh\right)$ is a Polish\footnote{A metric space is Polish if it is complete and separable.} space; see \cite[Theorem 7.3.30]{burago2001course} and \cite[Proposition 42 and 43]{petersen2006riemannian} for more details.

\medskip
\noindent\textbf{Note}: although $\ms$ is the set of isometry classes, by a slight abuse of notation, we write $X\in \ms$ to refer to an individual compact metric space $X$ instead of its isometry class.
\medskip

One important description of $\dgh$ is the following duality formula (cf. \Cref{thm:dgh-dual}) via correspondences between sets \cite[Chapter 7]{burago2001course}. Given $\left(X,d_X\right),\left(Y,d_Y\right)\in\ms$, define $\mathcal{R}\left(X,Y\right)$ as the set of all $R\subseteq X\times Y$ such that $\pi_X(R)=X$ and $\pi_Y(R)=Y$ where $\pi_X:X\times Y\rightarrow X$ and $\pi_Y:X\times Y\rightarrow Y$ are canonical projections. We call each $R\in\mathcal{R}\left(X,Y\right)$ a \emph{correspondence} between $X$ and $Y$. For a correspondence $R$, we define its \emph{distortion} with respect to $d_X$ and $d_Y$ by
$$\dis\left(R\right)\coloneqq\inf_{\left(x,y\right),\left(x',y'\right)\in R}|d_X\left(x,x'\right)-d_Y\left(y,y'\right)|. $$
\begin{theorem}\label{thm:dgh-dual}
For any $X,Y\in\ms$, we have
$$\dgh\left(X,Y\right)=\frac{1}{2}\inf_{R\in\mathcal{R}\left(X,Y\right)}\dis\left(R\right). $$
\end{theorem}

We let  $\mathcal{R}^\mathrm{opt}\left(X,Y\right)$ denote the set of all correspondences such that the equality in \Cref{thm:dgh-dual} holds. It is proved in \cite[Proposition 1.1]{chowdhury2018explicit} that $\mathcal{R}^\mathrm{opt}\left(X,Y\right)\neq \emptyset$ and there exists an $R\in\mathcal{R}^\mathrm{opt}\left(X,Y\right)$ which is a \emph{compact} subset of $\left(X\times Y,\max\left(d_X,d_Y\right)\right)$. A direct consequence of this fact is the following result:

\begin{lemma}\label{lm:dgh_hausdorff-realizable}
If $X$ and $Y$ are compact, then there exists a compact metric space $Z$ and isometric embeddings $\varphi_X:X\hookrightarrow Z$ and $\varphi_Y:Y\hookrightarrow Z$ such that
$$\dgh\left(X,Y\right)= \dH^Z\left(\varphi_X\left(X\right),\varphi_Y\left(Y\right)\right). $$
\end{lemma}
\begin{proof}
Let $R\in\mathcal{R}^\mathrm{opt}\left(X,Y\right)$ and $Z\coloneqq X\cup Y$. Let $d_Z:Z\times Z\rightarrow\mathbb{R}$ be such that $d_Z|_{X\times X}=d_X$, $d_Z|_{Y\times Y}=d_Y$ and for $x\in X$ and $y\in Y$
$$d_Z\left(x,y\right)\coloneqq\inf_{\left(x',y'\right)\in R}\left(d_X\left(x,x'\right)+d_Y\left(y,y'\right)+\frac{1}{2}\dis\left(R\right)\right).$$
It is proved in \cite[Lemma 2.8]{memoli2018sketching} that $\left(Z,d_Z\right)$ is a metric space and $\dH^Z\left(X,Y\right)=\dgh\left(X,Y\right).$ $\left(Z,d_Z\right)$ is obviously compact since any sequence $\{z_i\}_{i=0}^\infty\subseteq Z$ must contain either a convergent subsequence in $X$ or a convergent subsequence in $Y$. 
\end{proof}

\paragraph{Geodesics.} The following result is proved in \cite[Theorem 1]{ivanov2016gromov} using the mid-point criterion (cf. \Cref{thm:mid-pt-geo}):

\begin{restatable}{theorem}{thmgeoGH}\label{thm:GH-geo}
$\left(\ms,\dgh \right)$ is a geodesic metric space.

\end{restatable}

In {\cite[Theorem 1.2]{chowdhury2018explicit}}, the authors proved the existence of optimal correspondences which they used to give an  explicit construction of Gromov-Hausdorff geodesics and, as a consequence provided, an alternative proof of \Cref{thm:GH-geo}:

\begin{theorem}[Straight-line $\dgh$ geodesic \cite{chowdhury2018explicit}]\label{thm:str-line-geo}
For $X,Y\in\ms$ and any $R\in\mathcal{R}^\mathrm{opt}\left(X,Y\right)$, the curve $\gamma_R:[0,1]\rightarrow\ms$ defined as follows is a geodesic:
$$\gamma_R\left(0\right)=\left(X,d_X\right),\gamma_R\left(1\right)=\left(Y,d_Y\right)\text{ and }\gamma_R\left(t\right)=\left(R,d_{R_t}\right)\text{ for }t\in\left(0,1\right), $$
where $d_{R_t}\left(\left(x,y\right),\left(x',y'\right)\right)\coloneqq\left(1-t\right)\,d_X\left(x,x'\right)+t\,d_Y\left(y,y'\right).$

\end{theorem}

We will henceforth use the notation: $R_t\coloneqq \gamma_R(t)$ for $t\in[0,1]$.

\paragraph{Convergence.} Recall that for $\eps>0$ and $X\in \ms$, the covering number $\mathrm{cov}_\eps\left(X\right)$ is the least number of $\eps$-balls\footnote{An $\eps$-ball is a closed ball in $X$ with radius $\eps$.} required to cover the whole space $X$.

\begin{definition}[Uniformly totally bounded class]\label{def:CND}
We say a class $\mathcal{K}$ of compact metric spaces is \emph{uniformly totally bounded}, if there exist a bounded function $Q:\left(0,\infty\right)\rightarrow\mathbb{N}$ and $D>0$ such that each $X\in\mathcal{K}$ satisfies the following: 
\begin{enumerate}
    \item $\diam\left(X\right)\leq D$,
    \item for any $\eps>0$, $\mathrm{cov}_\eps\left(X\right)\leq Q\left(\eps\right)$.
\end{enumerate}
We denote by $\mathcal{K}\left(Q,D\right)$ the uniformly totally bounded class consisting of all $X\in\ms$ satisfying the conditions above.
\end{definition}

\begin{theorem}[Gromov's pre-compactness theorem]\label{thm:pre-compact}
For any given bounded function $Q:\left(0,\infty\right)\rightarrow\mathbb{N}$ and $D>0$, the class $\mathcal{K}\left(Q,D\right)$ is pre-compact in $\left(\ms,\dgh\right)$, i.e., any sequence in $\mathcal{K}\left(Q,D\right)$ has a convergent subsequence.
\end{theorem}

Interested readers are referred to \cite[Section 7.4.2]{burago2001course} for a proof.

\subsection{Sturm's Gromov-Wasserstein distance}\label{sec:gw-detail}
\paragraph{Wasserstein distance.}Given a metric space $X$ and any $p\in[1,\infty]$, there exists a natural distance $d_{\mathcal{W},p}^X$, the \emph{$\ell^p$-Wasserstein distance}, comparing certain Borel probability measures on $X$.

\begin{definition}[$\ell^p$-Wasserstein distance]\label{def:p-w-dist}
For a metric space $X$ (not necessarily compact) and $p\in[1,\infty)$, let $\mathcal{P}_p\left(X\right)$ denote the collection of all Borel probability measures $\alpha$ on $X$ such that 
$$\forall x_0\in X,\quad\int_{X}d_X^p\left(x,x_0\right)d\alpha\left(x\right)<\infty.$$ 
For $\alpha,\beta\in\mathcal{P}_p\left(X\right)$, the $\ell^p$-Wasserstein distance between $\alpha$ and $\beta$ is defined as follows:
$$d_{\mathcal{W},p}^X\left(\alpha,\beta\right)\coloneqq\inf_{\mu\in\mathcal{C}\left(\alpha,\beta\right)}\left(\int_{X\times X}d_X^p\left(x_1,x_2\right)\,d\mu\left(x_1,x_2\right)\right)^\frac{1}{p}, $$
where $\mathcal{C}\left(\alpha,\beta\right)$ denotes the set of measure couplings between $\alpha$ and $\beta$.

For $p=\infty$, let $\mathcal{P}_\infty\left(X\right)$ denote the collection of all Borel probability measures on $X$ with bounded support. We define the $\ell^\infty$-Wasserstein distance between $\alpha,\beta\in\mathcal{P}_\infty\left(X\right)$ by
$$d_{\mathcal{W},\infty}^X\left(\alpha,\beta\right)\coloneqq\inf_{\mu\in\mathcal{C}\left(\alpha,\beta\right)}\sup_{\left(x_1,x_2\right)\in\supp\left(\mu\right)}d_X\left(x_1,x_2\right). $$
\end{definition}

\begin{lemma}[{\cite[Theorem 4.1]{villani2008optimal}}]
Fix $p\in[1,\infty)$. For a compact metric space $X$ and $\alpha,\beta\in\mathcal{P}_p(X)$, there exists $\mu\in\mathcal{C}(\alpha,\beta)$ such that 
$$d_{\mathcal{W},p}^X\left(\alpha,\beta\right)=\left(\int_{X\times X}d_X^p\left(x_1,x_2\right)\,d\mu\left(x_1,x_2\right)\right)^\frac{1}{p}. $$
We call such $\mu$ an \emph{optimal transference plan} between $\alpha$ and $\beta$ (with respect to $d_{\mathcal{W},p}^X$) and denote by $\mathcal{C}^\mathrm{opt}_p(\alpha,\beta)$ the collection of all optimal transference plans.
\end{lemma}

\paragraph{Sturm's Gromov-Wasserstein distance.} A metric measure space is a triple $\mathcal{X}=(X,d_X,\mu_X)$ where $(X,d_X)$ is a metric space and $\mu_X$ is a Borel probability measure on $(X,d_X)$. We use script letters such as $\mathcal{X}$ to denote a metric measure space $\mathcal{X}=(X,d_X,\mu_X)$.

\begin{definition}[Isomorphism of metric measure spaces]\label{def:isomorphism}
Given two metric measure spaces $\X$ and $\Y$, we say that they are \emph{isomorphic}, if there exists an isometry $\varphi:X\rightarrow Y$ such that $\mu_Y=\varphi_\#\mu_X$, where $\varphi_\#$ denotes the pushforward map under $\varphi$. Whenever $\mathcal{X}$ is isomorphic to $\mathcal{Y}$, we write $\mathcal{X}\cong_w\mathcal{Y}$.
\end{definition}

Now, we provide the definition of the Gromov-Wasserstein distance given by Sturm in \cite{sturm2006geometry,sturm2012space}.
\begin{definition}[Gromov-Wasserstein distance]\label{def:dGW}
Let $p\in[1,\infty]$ and let $\mathcal{X}=(X,d_X,\mu_X)$ and $\mathcal{Y}=(Y,d_Y,\mu_Y)$ be two compact metric measure spaces with full support. The $\ell^p$-Gromov-Wasserstein distance $\dgws{p}$ between $\mathcal{X}$ and $\mathcal{Y}$ is defined by
$$\dgws{p}\left(\mathcal{X},\mathcal{Y}\right)\coloneqq\inf_{Z}\dW{p}^Z\left((\varphi_X)_\#\mu_X,(\varphi_Y)_\#\mu_Y\right), $$
where the infimum is taken over all metric spaces $Z$ and isometric embeddings $\varphi_X:X\hookrightarrow Z$ and $\varphi_Y:Y\hookrightarrow Z$. 
\end{definition}

For notational simplicity, we sometimes identify $(\varphi_X)_\#\mu_X$ with $\mu_X$ and simply write $\dW{p}^Z\left(\mu_X,\mu_Y\right)$ to avoid carrying heavy notations of pushforward maps from isometric embeddings.

Let $p\in[1,\infty]$ and let $\mathcal{X}=(X,d_X,\mu_X)$ and $\mathcal{Y}=(Y,d_Y,\mu_Y)$ be two compact metric measure spaces with full support. Then, $\dgws{p}(\mathcal{X},\mathcal{Y})=0$ if and only if $\mathcal{X}$ and $\mathcal{Y}$ are isomorphic to each other. {The case when $p\in[1,\infty)$ was mentioned in \cite[Proposition 2.4]{sturm2012space} whereas the case $p=\infty$ can be obviously derived from \cite[Theorem 5.1 (a) and (g)]{memoli2011gromov}.} Let $\ws$ denote the collection of all isomorphism classes of compact metric measure spaces with full support. Then, for each $p\in[1,\infty]$, $\left(\ws,\dgws p\right)$ is a metric space.

\noindent\textbf{Note}: although $\ws$ is a set of isomorphism classes, by a slight abuse of notation, we write $\X\in \ws$ to refer to an individual metric measure space $\X$ instead of its isomorphism class.

The following is a useful alternative formulation of the Gromov-Wasserstein distance:

{\begin{remark}[Metric coupling formulation]\label{rmk:metric coupling}
Let $\mathcal{D}(d_X,d_Y)$ denote the set of all metrics $d:X\sqcup Y\times X\sqcup Y\rightarrow\mathbb R_{\geq 0}$ such that $d|_{X\times X}=d_X$ and $d|_{Y\times Y}=d_Y$. We call each element in $\mathcal{D}(d_X,d_Y)$ a \emph{metric coupling} between $d_X$ and $d_Y$. Then, it is easy to check that 
$$\dgws{p}\left(\mathcal{X},\mathcal{Y}\right)=\inf_{d\in\mathcal{D}(d_X,d_Y)}\dW{p}^{(X\sqcup Y,d)}\left(\mu_X,\mu_Y\right). $$
\end{remark}}

{The following result is analogous to \Cref{lm:dgh_hausdorff-realizable} for the Gromov-Hausdorff distance.}

\begin{lemma}\label{lm:dgw_w-realizable}
Let $\mathcal{X}=(X,\mu_X),\mathcal{Y}=(Y,\mu_Y)\in\ws$ and $p\in[1,\infty)$. Then, there exists a compact metric space $Z$ and isometric embeddings $\varphi_X:X\hookrightarrow Z$ and $\varphi_Y:Y\hookrightarrow Z$ such that
$$\dgws{p}\left(\mathcal{X},\mathcal{Y}\right)= \dW{p}^Z\left((\varphi_X)_\#\mu_X,(\varphi_Y)_\#\mu_Y\right). $$
\end{lemma}
\begin{proof}
It is proved in \cite[Proposition 2.4]{sturm2012space} that there exists a metric space $\hat{Z}$ and isometric embeddings ${\varphi}_X:X\hookrightarrow \hat{Z}$ and ${\varphi}_Y:Y\hookrightarrow \hat{Z}$ such that $\dgws{p}\left(\mathcal{X},\mathcal{Y}\right)= \dW{p}^{\hat{Z}}\left(({\varphi}_X)_\#\mu_X,({\varphi}_Y)_\#\mu_Y\right). $ Now let $Z\coloneqq{\varphi}_X(X)\cup{\varphi}_Y(Y)$. Then, $Z$ is compact. Since $\mathrm{im}(\varphi_X),\mathrm{im}(\varphi_Y)\subseteq Z$, both $\varphi_X$ and $\varphi_Y$ are actually isometric embeddings ${\varphi}_X:X\hookrightarrow Z$ and ${\varphi}_Y:Y\hookrightarrow Z$, respectively. Then, it is easy to see that 
$$\dgws{p}\left(\mathcal{X},\mathcal{Y}\right)=\dW{p}^{\hat{Z}}\left((\varphi_X)_\#\mu_X,(\varphi_Y)_\#\mu_Y\right)= \dW{p}^Z\left((\varphi_X)_\#\mu_X,(\varphi_Y)_\#\mu_Y\right). $$
\end{proof}


\section{Metric extensions}\label{sec:metric-ext}
For any two metric spaces $X$ and $Y$, if there exists an isometric embedding $X\hookrightarrow Y$, then we call $Y$ a \emph{metric extension} of $X$. A \emph{metric extensor} is any map $\mathcal{F}$ taking a compact metric space $X$ to another metric space $\mathcal{F}\left(X\right)$ such that $\mathcal{F}(X)$ is a metric extension of $X$. In this section, we examine three standard models of metric extensions, namely, the Hausdorff hyperspace, the Wasserstein hyperspace and the Urysohn universal metric space. Properties of these metric extensions and their corresponding metric extensors are essential for proving our main results.
\subsection{Hausdorff hyperspaces}
Given a metric space $X$, the \emph{Hausdorff hyperspace} $\mathcal{H}\left(X\right)$ of $X$ is composed of all nonempty bounded closed subsets of $X$ and is endowed with the Hausdorff distance $\dH^X$ as its metric.
\begin{theorem}\label{thm:hyper-complete}
If $X$ is a complete metric space, then $\left(\mathcal{H}\left(X\right),\dH^X\right)$ is also complete.
\end{theorem}
\begin{theorem}[Blaschke's theorem]\label{thm:blaschke}
If $X$ is a compact metric space, then $\left(\mathcal{H}\left(X\right),\dH^X\right)$ is also compact.
\end{theorem}
See \cite[Section 7.3]{burago2001course} for proofs of the above two results. Note that $\mathcal{H}$ mapping $X$ to $\mathcal{H}(X)$ is then a map from $\ms$ to $\ms$. The map sending $x\in X$ to the singleton $\{x\}\in \mathcal{H}\left(X\right)$ for each $x\in X$ is an isometric embedding from $X$ to $\mathcal{H}\left(X\right)$. This implies that $\mathcal{H}:\ms\rightarrow\ms$ is a metric extensor, which we call the \emph{Hausdorff extensor}. One interesting aspect of $\mathcal{H}$ as a map is the stability. In fact, it is proved in \cite{mikhailov2018hausdorff} that $\mathcal{H}$ is a 1-Lispchitz map:

\begin{theorem}[{\cite[Theorem 2]{mikhailov2018hausdorff}}]\label{thm:H-stb}
For any $X,Y\in\ms$, we have
$$\dgh\left(\mathcal{H}\left(X\right),\mathcal{H}\left(Y\right)\right)\leq \dgh\left(X,Y\right). $$
\end{theorem}

Given an isometric embedding $\varphi:X\hookrightarrow Z$, for any closed subset $A\subseteq X$, the image $\varphi(A)$ is a closed subset of $Z$. This induces an isometric embedding $\varphi_*:(\mathcal{H}(X),\dH^X)\hookrightarrow (\mathcal{H}(Z),\dH^Z)$ mapping $A\in \mathcal{H}(X)$ to $\varphi(A)\in \mathcal{H}(Z)$. Then, \Cref{thm:H-stb} is a direct consequence of the following interesting result: 
\begin{theorem}[{\cite[Theorem 1]{mikhailov2018hausdorff}}]\label{thm:H-equal}
Given two compact metric spaces $\left(X,d_X\right),\left(Y,d_Y\right)$, suppose there exist a metric space $Z$ (not necessarily compact) and isometric embeddings $\varphi_X:X\hookrightarrow Z$ and $\varphi_Y:Y\hookrightarrow Z$. Then, we have that 
$$\dH^{\mathcal{H}(Z)}\big((\varphi_X)_*(\mathcal{H}(X)),(\varphi_Y)_*(\mathcal{H}(Y))\big)=\dH^Z(\varphi_X(X),\varphi_Y(Y)). $$
\end{theorem}

\paragraph{Geodesics in Hausdorff hyperspaces.} One interesting fact about $\mathcal{H}\left(X\right)$ is that it preserves the geodesic property of $X$:

\begin{theorem}\label{thm:hgeo}
Given $X\in \ms$, if $X$ is geodesic, then so is $\mathcal{H}\left(X\right)$.
\end{theorem}

The above theorem was first proved in \cite{bryant1970convexity} using the mid-point criterion (cf. \Cref{thm:mid-pt-geo}) and was later reproved in \cite{serra1998hausdorff} via the following explicit construction:
\begin{theorem}\label{thm:haus-geo-cons}
Let $X\in\ms$ be a geodesic space. Let $A,B$ be two closed subsets of $X$. Let $\rho\coloneqq\dH^X\left(A,B\right)$. Then, $\gamma:[0,1]\rightarrow \mathcal{H}\left(X\right)$ defined by $\gamma\left(t\right)\coloneqq A^{t\,\rho}\cap B^{\left(1-t\right)\,\rho}$ is a Hausdorff geodesic connecting $A$ and $B$, where for any $r\geq 0$, $A^r\coloneqq\{x\in X:\,\exists a\in A \text{ such that }d_X\left(a,x\right)\leq r\}$.
\end{theorem}

Though the construction is correct, the proof of \Cref{thm:haus-geo-cons} given in \cite{serra1998hausdorff} is based on the following false claim:

\begin{claim}[False claim in the proof of {\cite[Theorem 1]{serra1998hausdorff}}]\label{claim:false claim}
Given $X\in\ms$ and a map $\gamma:[0,1]\rightarrow X$, if $d_X\left(\gamma\left(0\right),\gamma\left(t\right)\right)=t\cdot d_X(\gamma(0),\gamma(1))$ and $d_X\left(\gamma\left(t\right),\gamma\left(1\right)\right)=(1-t)\cdot d_X(\gamma(0),\gamma(1))$ hold for all $t\in[0,1]$, then $\gamma$ is a geodesic.
\end{claim}

A simple counterexample goes as follows: let $Y\coloneqq[0,3]\times [0,1]\subseteq\mathbb R^2$ endowed with the usual Euclidean metric and let $X$ be the quotient space of $Y$ obtained by collapsing both $\{0\}\times[0,1]$ and $\{3\}\times[0,1]$ to points, respectively. Then, any ``reasonable'' curve connecting these two points will satisfy the condition in the claim while not necessary being a geodesic. See \Cref{app:geo-hyp} for details and a correct proof of \Cref{thm:haus-geo-cons} which still follows the main idea in \cite{serra1998hausdorff}.

{It is worth noting that based on a new technique which we introduce later, i.e., the Hausdorff displacement interpolation, we are able to provide efficient alternative proofs for both \Cref{thm:hgeo} and \Cref{thm:haus-geo-cons} in \Cref{sec:hausdorff displacement interpolation}. In particular, regarding  \Cref{thm:hgeo} we prove a stronger result which provides both necessary and sufficient conditions instead of just a one way implication.}

The following is a simple observation which we will use heavily in the sequel for transforming a Hausdorff geodesic to a Gromov-Hausdorff geodesic.

\begin{lemma}\label{lm:hgeo-to-dghgeo}
Let $X,Y,Z\in\ms$. Suppose that $Z$ is geodesic and there exist $\varphi_X:X\hookrightarrow Z$ and $\varphi_Y:Y\hookrightarrow Z$ such that $\dH^Z\left(\varphi_X\left(X\right),\varphi_Y\left(Y\right)\right)=\dgh\left(X,Y\right)$. Then, any Hausdorff geodesic $\gamma:[0,1]\rightarrow \mathcal{H}\left(Z\right)$ such that $\gamma\left(0\right)=X$ and $\gamma\left(1\right)=Y$ is actually a Gromov-Hausdorff geodesic, i.e., for any $s,t\in[0,1]$
\[\dgh\left(\gamma\left(s\right),\gamma\left(t\right)\right)=|s-t|\,\dgh\left(X,Y\right). \]
\end{lemma}

\begin{proof}
We only need to show that $\dH^Z\left(\gamma\left(s\right),\gamma\left(t\right)\right)=\dgh\left(\gamma\left(s\right),\gamma\left(t\right)\right)$ for any $s,t\in[0,1]$. By \Cref{def:dGH}, we have $\dH^Z\left(\gamma\left(s\right),\gamma\left(t\right)\right)\geq \dgh\left(\gamma\left(s\right),\gamma\left(t\right)\right)$. Without loss of generality, we assume that $s\leq t$. Since $\gamma$ is a Hausdorff geodesic, we have
\begin{align*}
    \dgh\left(X,Y\right)&=\dH^Z\left(\gamma\left(0\right),\gamma\left(1\right)\right)\\
    &=\dH^Z\left(\gamma\left(0\right),\gamma\left(s\right)\right)+\dH^Z\left(\gamma\left(s\right),\gamma\left(t\right)\right)+\dH^Z\left(\gamma\left(t\right),\gamma\left(1\right)\right)\\
    &\geq \dgh\left(\gamma\left(0\right),\gamma\left(s\right)\right)+\dgh\left(\gamma\left(s\right),\gamma\left(t\right)\right)+\dgh\left(\gamma\left(t\right),\gamma\left(1\right)\right)\\
    &\geq \dgh\left(\gamma\left(0\right),\gamma\left(1\right)\right)\\
    &=\dgh\left(X,Y\right).
\end{align*}
Therefore, every equality holds. In particular, $\dH^Z\left(\gamma\left(s\right),\gamma\left(t\right)\right)=\dgh\left(\gamma\left(s\right),\gamma\left(t\right)\right)$.
\end{proof}

\begin{example}[$\dgh$ geodesic connecting $\mathbb S^0$ and $\mathbb S^n$]\label{ex:geod-sphere}
In \cite{chowdhury2018explicit} the authors constructed  explicit Gromov-Hausdorff geodesics between the spheres $\mathbb S^0$ and $\mathbb S^n$ with the canonical geodesic distance,  for each $n\in\mathbb N$. We recover their construction via the techniques introduced in this section as follows. Note that if we identify $\mathbb S^0$ with any pair of antipodal points, e.g., the north and south poles, in $\mathbb S^n$, then $\dH^{\mathbb S^n}(\mathbb S^0,\mathbb S^n)=\frac{\pi}{2}$. By \cite[Proposition 1.2]{chowdhury2018explicit}, $\dgh(\mathbb S^0,\mathbb S^n)=\frac{\pi}{2}=\dH^{\mathbb S^n}(\mathbb S^0,\mathbb S^n).$ Then, by \Cref{thm:haus-geo-cons} and \Cref{lm:hgeo-to-dghgeo}, the Hausdorff geodesic $\gamma:[0,1]\rightarrow \mathcal{H}(\mathbb S^n)$ defined by $t\mapsto (\mathbb S^0)^{t\cdot\frac{\pi}{2}}\cap(\mathbb S^n)^{(1-t)\cdot\frac{\pi}{2}}=(\mathbb S^0)^{t\cdot\frac{\pi}{2}}$ for $t\in[0,1]$ is a Gromov Hausdorff geodesic connecting $\mathbb S^0$ and $\mathbb S^n$. Note that $\gamma$ is exactly the same geodesic connecting $\mathbb S^0$ and $\mathbb S^n$ constructed in \cite[Proposition 1.3]{chowdhury2018explicit}. See also \Cref{fig:geod-circle} for an illustrative representation of $\gamma$.
\end{example}

\begin{figure}[htb]
	\centering		\includegraphics[width=0.3\textwidth]{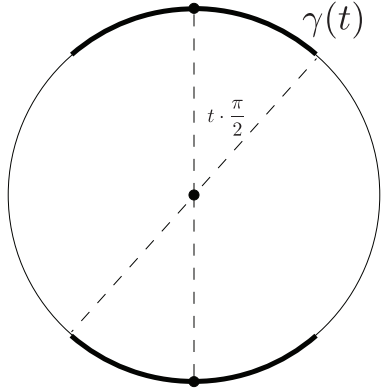}
	\caption{\textbf{Illustration of \Cref{ex:geod-sphere}.} In the figure we identify $\mathbb S^0$ with the north and south poles of $\mathbb S^1$ and illustrate $\gamma(t)$ for some $t\in(0,1)$ as the thickened subset of $\mathbb{S}^1$.} \label{fig:geod-circle}
\end{figure}

\subsection{Wasserstein hyperspaces}\label{sec:W-geo}
Given a metric space $X$, let $\mathcal{W}_p\left(X\right)\coloneqq\left(\mathcal{P}_p\left(X\right),d_{\mathcal{W},p}\right)$ (cf. \Cref{def:p-w-dist}). We call $\mathcal{W}_p\left(X\right)$ the \emph{$\ell^p$-Wasserstein hyperspace} of $X$. Note that when $X$ is compact, $\mathcal{P}_p\left(X\right)=\mathcal{P}\left(X\right)$ for any $p\in[1,\infty]$, where $\mathcal{P}(X)$ denotes the collection of all Borel probability measures on $X$. The following two theorems are standard results about Wasserstein hyperspaces and see for example \cite[Section 6]{villani2008optimal} for proofs.

\begin{theorem}\label{thm:complete-w}
For $p\in[1,\infty)$, if $X$ is Polish, then $\mathcal{W}_p\left(X\right)$ is also Polish.
\end{theorem}
\begin{theorem}\label{thm:compact-w}
For $p\in[1,\infty)$, if $X$ is compact, then $\mathcal{W}_p\left(X\right)$ is also compact.
\end{theorem}

Note that $\mathcal{W}_p$ sending $X\in\ms$ to $\mathcal{W}_p(X)\in\ms$ then defines a map from $\ms$ to $\ms$ analogously to the case of $\mathcal{H}$. Moreover, the map sending $x$ to the Dirac measure $\delta_x\in\mathcal{P}(X)$ is an isometric embedding from $X$ into $\mathcal{P}(X)$. Therefore, $\mathcal{W}_p:\ms\rightarrow\ms$ is a metric extensor, which we call the ($\ell^p$-)\emph{Wasserstein extensor} in the sequel.

Inspired by \Cref{thm:H-equal}, we establish the following result:

\begin{theorem}\label{thm:W-equal}
Given two compact metric spaces $\left(X,d_X\right)$ and $\left(Y,d_Y\right)$, suppose there exist a (not necessarily compact) metric space $\left(Z,d_Z\right)$ and isometric embeddings $\varphi_X:X\hookrightarrow Z$ and $\varphi_Y:Y\hookrightarrow Z$. Then, for $p\in[1,\infty]$, we have that both $(\varphi_X)_\#:\mathcal{W}_p(X)\rightarrow\mathcal{W}_p(Z)$ and $(\varphi_Y)_\#:\mathcal{W}_p(Y)\rightarrow\mathcal{W}_p(Z)$ are isometric embeddings. Moreover, we have that
\[\dH^{\mathcal{W}_p\left(Z\right)}\big((\varphi_X)_\#\lc \mathcal{W}_p\left(X\right)\rc,(\varphi_Y)_\#\lc \mathcal{W}_p\left(Y\right)\rc\big)= \dH^Z\left(\varphi_X(X),\varphi_Y(Y)\right).\]
\end{theorem}

See \Cref{app:ot} for a proof. As a direct yet unexpected consequence, we obtain the following 1-Lipschitz property of the Wasserstein extensor $\mathcal{W}_p:\ms\rightarrow\ms$:

\begin{theorem}\label{thm:W-1-lip}
Given $X,Y\in\ms$ and any $p\in[1,\infty]$, we have
$$\dgh\left(\mathcal{W}_p\left(X\right),\mathcal{W}_p\left(Y\right)\right)\leq\dgh\left(X,Y\right). $$
\end{theorem}

\begin{remark}[Comparison with related results]
In the literature, there are studies about the stability of $\mathcal{W}_2$ on $\ms$. For example, {one can derive from \cite[Proposition 4.1]{lott2009ricci} that 
$$\dgh\left(\mathcal{W}_2\left(X\right),\mathcal{W}_2\left(Y\right)\right)\leq f_{XY}\lc\dgh\left(X,Y\right)\rc$$
for all $X,Y\in\ms$ and for some function $f_{XY}$ depending on the diameters of $X$ and $Y$. \Cref{thm:W-1-lip} is novel in that we not only proved stability of the Wasserstein extensor $\mathcal{W}_p$ for \emph{all} $p\in[1,\infty]$, but we also obtained the stronger result that $\mathcal{W}_p$ is 1-Lipschitz.}
\end{remark}

\paragraph{Geodesics in $\mathcal{W}_p\left(X\right)$.} We now state some results regarding geodesics in Wasserstein hyperspaces.

\begin{theorem}[{\cite[Corollary 7.22]{villani2008optimal}}]\label{thm:wp-geo}
If $X$ is a geodesic metric space, then $\mathcal{W}_p(X)$ is a geodesic metric space for all $p\in[1,\infty)$.
\end{theorem}

The case when $p=1$ is special in that $\mathcal{W}_1(X)$ is geodesic regardless of whether $X$ is itself geodesic:

\begin{theorem}\label{thm:W-geodesic}
For \emph{any} Polish metric space $X$, $\mathcal{W}_1\left(X\right)$ is geodesic.
\end{theorem}

{In fact, this theorem follows directly from the following explicit construction:
\begin{lemma}[{\cite[Theorem 5.1]{bottou2018geometrical}}]\label{lm:W-geodesic}
Let $X$ be a Polish metric space and let $\alpha,\beta\in\mathcal{P}_1\left(X\right)$. We define $\gamma:[0,1]\rightarrow\mathcal{P}(X)$ as follows: for each $t\in[0,1]$, let $\gamma(t)\coloneqq \left(1-t\right)\alpha+t\beta$. Then, for each $t\in[0,1]$, $\gamma(t)\in\mathcal{P}_1(X)$ and $\gamma$ is an $\ell^1$-Wasserstein geodesic connecting $\alpha$ and $\beta$. We call $\gamma$ the \emph{linear interpolation} between $\alpha$ and $\beta$.
\end{lemma}}

{\Cref{lm:W-geodesic} was first mentioned and proved in \cite{bottou2018geometrical} via Kantorovich duality. In \Cref{app:ot}, we provide an alternative proof which proceeds by calculating $\dW 1^X(\gamma(s),\gamma(t))$ for all $s,t\in[0,1]$ via an explicit construction of an optimal coupling between $\gamma(s)$ and $\gamma(t)$.}

\begin{remark}
For $p\neq 1$, a statement similar to that of \Cref{lm:W-geodesic} for $\mathcal{W}_p\left(X\right)$ is not necessarily true. For example, consider the {two point space $X=\{0,1\}$ with interpoint distance $1$}. Then, for $p\in[1,\infty)$, one can easily verify that $\mathcal{W}_p\left(X\right)\cong\left([0,1],d^\frac{1}{p}\right)$ where $d$ denotes the Euclidean metric on $[0,1]$. In this case, $\mathcal{W}_p\left(X\right)$ is geodesic if and only if $p=1$.
\end{remark}

The next lemma is a counterpart to \Cref{lm:hgeo-to-dghgeo} in the setting of metric measure spaces.

\begin{lemma}\label{lm:wgeo-to-dgwgeo}
Let $\mathcal{X}=(X,d_X,\mu_X),\mathcal{Y}=(Y,d_Y,\mu_Y)\in\ws$ and let $Z\in\ms$. Fix $p\in[1,\infty)$. Suppose there exist $\varphi_X:X\hookrightarrow Z$ and $\varphi_Y:Y\hookrightarrow Z$ such that 
$$\dW{p}^Z\left((\varphi_X)_\#\mu_X,(\varphi_Y)_\#\mu_Y\right)=\dgws{p}\left(\mathcal{X},\mathcal{Y}\right).$$ 
Then, any $\ell^p$-Wasserstein geodesic $\gamma:[0,1]\rightarrow \mathcal{W}_p\left(Z\right)$ such that $\gamma\left(0\right)=(\varphi_X)_\#\mu_X$ and $\gamma\left(1\right)=(\varphi_Y)_\#\mu_Y$ is actually an $\ell^p$-Gromov-Wasserstein geodesic. More precisely, if for each $t\in[0,1]$ we let $X_t\coloneqq\supp(\gamma(t))\subseteq Z$ and denote by $\tilde{\gamma}(t)$ the metric measure space $\left(X_t,d_Z|_{X_t\times X_t},\gamma(t)\right)$, then $\tilde{\gamma}:[0,1]\rightarrow\ws$ is a geodesic, i.e.,
$$\dgws{p}\left(\tilde{\gamma}\left(s\right),\tilde{\gamma}\left(t\right)\right)=|s-t|\cdot\dgws{p}\left(\mathcal{X},\mathcal{Y}\right) \quad\forall s,t\in[0,1].$$
\end{lemma}

The proof is essentially the same as the one for \Cref{lm:hgeo-to-dghgeo} and we omit details.

\paragraph{A new proof of \Cref{thm:GH-geo}.}
As mentioned in the introduction, $\left(\ms,\dgh\right)$ is a geodesic metric space. {This was proved in \cite{ivanov2016gromov,chowdhury2018explicit} respectively via the mid-point criterion and via an explicit construction of geodesics}. Now, we end this section by presenting a novel and succinct proof of this fact based on geodesic properties of $\mathcal{W}_1$ and $\mathcal{H}$.

\thmgeoGH*
\begin{proof}\label{succinct proof}
Given two $X,Y\in\ms$, let $\eta\coloneqq\dgh \left(X,Y\right).$ Then, there exists $Z\in\ms$ and isometric embeddings $\varphi_X:X\hookrightarrow Z$ and $\varphi_Y:Y\hookrightarrow Z$ such that $\dH^Z\left(\varphi_X(X),\varphi_Y(Y)\right)=\eta$ (cf. \Cref{lm:dgh_hausdorff-realizable}). Without loss of generality, we assume that $Z$ is geodesic (otherwise we replace $Z$ with one of its extensions $\mathcal{W}_1\left(Z\right)$, which is geodesic by \Cref{thm:W-geodesic}). Then, $\mathcal{H}\left(Z\right)$ is geodesic by \Cref{thm:hgeo}. Consequently, there exists a Hausdorff geodesic $\gamma:[0,1]\rightarrow \mathcal{H}\left(Z\right)$ such that $\gamma\left(0\right)=X$ and $\gamma\left(1\right)=Y$ (we regard $X$ and $Y$ as subsets of $Z$). Then, by \Cref{lm:hgeo-to-dghgeo}, one concludes that $\gamma$ is a geodesic in $\ms$ connecting $X$ and $Y$. Therefore, $\ms$ is itself geodesic.
\end{proof}

\subsection{Urysohn universal metric space}\label{sec:urysohn}
In this section we introduce the Urysohn universal metric space which is a remarkable construction by Urysohn \cite{urysohn1927espace}. This metric space can be regarded as an ambient space inside which one can isometrically embed every Polish metric space. It has some natural connections with the Gromov-Hausdorff distance and the Gromov-Wasserstein distance which we will elucidate. These connections will serve an important role in the proof of our main theorems.

\begin{theorem}[Urysohn universal metric space]\label{thm:urysohn}
There exists a unique (up to isometry) Polish space $\left(\mathbb{U},d_{\mathbb{U}}\right)$, {which we call the \emph{Urysohn universal metric space}, satisfying the following two properties:}
\begin{enumerate}
    \item (Universality) For any separable metric space $X$, there exists an isometric embedding $X\hookrightarrow \mathbb{U}$.
    \item (Homogeneity) For any isometry $\varphi$ between two finite subsets $A,A'\subseteq\mathbb{U}$, there exists an isometry $\tilde{\varphi}:\mathbb{U}\rightarrow\mathbb{U}$ such that $\tilde{\varphi}|_A=\varphi$.
\end{enumerate}
\end{theorem}
The universality property of $\mathbb{U}$ makes the constant map taking each compact metric space $X$ to $\mathbb{U}$ a metric extensor.

The homogeneity property in \Cref{thm:urysohn} above can be generalized to compact metric spaces:

\begin{theorem}[\cite{huhunaivsvili1955property}]\label{thm:u-compact-homo}
Given two \emph{compact} subsets $A,A'\subseteq\mathbb{U}$ and an isometry $\varphi:A\rightarrow A'$, there exists an isometry $\tilde{\varphi}:\mathbb{U}\rightarrow\mathbb{U}$ such that $\tilde{\varphi}|_A=\varphi$.
\end{theorem}

\paragraph{Connection with the Gromov-Hausdorff distance.} For any $X,Y\in\ms$, due to universality, $\mathbb{U}$ is a common ambient space. This allows us to consider the Hausdorff distance $\dH^\mathbb U$ between $X$ and $Y$ which implies a connection between $\mathbb U$ and $\dgh$. The following relation between the Gromov-Hausdorff distance and the Urysohn universal metric space was pointed out in \cite{berestovskii1992manifolds,antonyan2020gromov}:

\begin{proposition}
For any $X,Y\in\ms$, we have 
$$\dgh\left(X,Y\right)=\inf_{\varphi_X,\varphi_Y} \dH^\mathbb{U}\left(\varphi_X\left(X\right),\varphi_Y\left(Y\right)\right), $$
where the infimum is taken over all isometric embeddings $\varphi_X:X\hookrightarrow \mathbb{U}$ and $\varphi_Y:Y\hookrightarrow \mathbb{U}$.
\end{proposition}

We further improve this proposition through the following lemma:

\begin{lemma}\label{lm:u-dgh}
For any $X,Y\in\ms$, there exist isometric embeddings $\varphi_X:X\hookrightarrow \mathbb{U}$ and $\varphi_Y:Y\hookrightarrow \mathbb{U}$ such that 
$$\dgh\left(X,Y\right)=\dH^\mathbb{U}\left(\varphi_X\left(X\right),\varphi_Y\left(Y\right)\right). $$
\end{lemma}

\begin{proof}
By \Cref{lm:dgh_hausdorff-realizable}, there exist $Z\in\ms$ and isometric embeddings $\psi_X:X\hookrightarrow Z$ and $\psi_Y:Y\hookrightarrow Z$ such that $\dH^Z\left(\psi_X\left(X\right),\psi_Y\left(Y\right)\right)=\dgh\left(X,Y\right)$. Then, since $Z$ is compact, there exists an isometric embedding $\varphi:Z\hookrightarrow\mathbb{U}$ (cf. \Cref{thm:urysohn}). Let $\varphi_X=\varphi|_{\psi_X(X)}\circ\psi_X$ and $\varphi_Y=\varphi|_{\psi_Y(Y)}\circ\psi_Y$. Then, by \Cref{lm:dH under embedding} we have that
$$\dgh\left(X,Y\right)= \dH^Z\left(\psi_X\left(X\right),\psi_Y\left(Y\right)\right)=\dH^{\varphi\left(Z\right)}\left(\varphi_X\left(X\right),\varphi_Y\left(Y\right)\right)=\dH^\mathbb{U}\left(\varphi_X\left(X\right),\varphi_Y\left(Y\right)\right).$$
\end{proof}

Then, combining with \Cref{thm:u-compact-homo}, we derive the following lemma:
\begin{lemma}\label{lm:u-dgh-any}
For any $X,Y\in\ms$, let $\varphi_X:X\hookrightarrow\mathbb{U}$ be an isometric embedding. Then, there exists an isometric embedding $\varphi_Y:Y\hookrightarrow \mathbb{U}$ such that 
$$\dgh\left(X,Y\right)=\dH^\mathbb{U}\left(\varphi_X\left(X\right),\varphi_Y\left(Y\right)\right). $$
\end{lemma}

\begin{proof}
By \Cref{lm:u-dgh}, there exist isometric embeddings $\psi_X:X\hookrightarrow\mathbb{U}$ and $\psi_Y:Y\hookrightarrow\mathbb{U}$ such that $\dgh\left(X,Y\right)=\dH^\mathbb{U}\left(\psi_X\left(X\right),\psi_Y\left(Y\right)\right).$ Now, both $\varphi_X\left(X\right)$ and $\psi_X\left(X\right)$ are compact subsets of $\mathbb{U}$ and $\tau\coloneqq\varphi_X\circ\psi_X^{-1}:\psi_X\left(X\right)\rightarrow\varphi_X\left(X\right)$ is an isometry. By \Cref{thm:u-compact-homo}, there exists an isometry $\tilde{\tau}:\mathbb{U}\rightarrow\mathbb{U}$ such that $\tilde{\tau}|_{\psi_X\left(X\right)}=\tau$. Let $\varphi_Y\coloneqq\tilde{\tau}|_{\psi_Y(Y)}\circ\psi_Y:Y\rightarrow\mathbb{U}$. It is clear that $\varphi_Y$ is an isometric embedding and thus
$$\dH^\mathbb{U}\left(\varphi_X\left(X\right),\varphi_Y\left(Y\right)\right)= \dH^\mathbb{U}\left(\tilde{\tau}^{-1}\circ\varphi_X\left(X\right),\tilde{\tau}^{-1}\circ\varphi_Y\left(Y\right)\right)=\dH^\mathbb{U}\left(\psi_X\left(X\right),\psi_Y\left(Y\right)\right)=\dgh\left(X,Y\right).$$
\end{proof}

See \Cref{fig:u-dgh} for an illustration of the proof for \Cref{lm:u-dgh-any}.

\begin{figure}[htb]
	\centering		\includegraphics[width=0.8\textwidth]{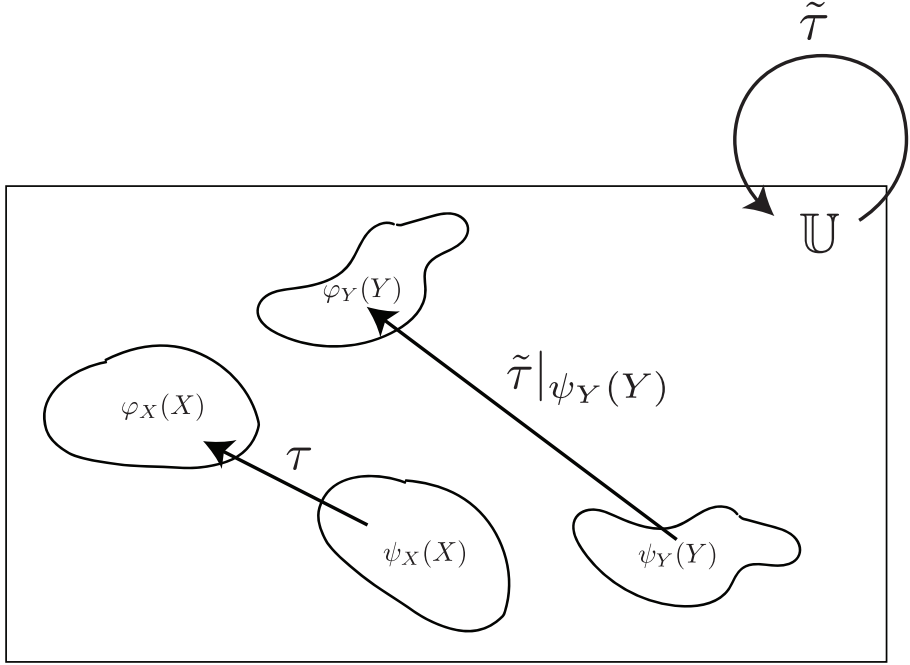}
	\caption{\textbf{Illustration of the proof of \Cref{lm:u-dgh-any}}} \label{fig:u-dgh}
\end{figure}

\Cref{lm:u-dgh-any} then leads us to the following key observation which is instrumental in proving \Cref{thm:main-h-realizable}.

\begin{lemma}
For any Gromov-Hausdorff geodesic $\gamma:[0,1]\rightarrow \mathcal{M}$, let $\rho\coloneqq \dgh\left(\gamma\left(0\right),\gamma\left(1\right)\right)$. Then, for any finite sequence $0\leq t_0<t_1<\ldots<t_n\leq 1$, there exist for all $i=0,\ldots,n$ isometric embeddings $\varphi_i:\gamma\left(t_i\right)\hookrightarrow \mathbb{U}$ such that 
$$\dH^\mathbb{U}\big(\varphi_i\left(\gamma\left(t_i\right)\right),\varphi_j\left(\gamma\left(t_j\right)\right)\big)=|t_i-t_j|\,\rho,\quad\forall0\leq i,j\leq n.$$
\end{lemma}
\begin{proof}
We prove the lemma by induction on $n$. When $n=0$, the statement holds true trivially. Now, suppose the statement holds for $n\geq 0$. Consider a sequence $0\leq t_0<t_1<\ldots<t_n<t_{n+1}\leq 1$. By the induction assumption, there exist $\varphi_i:\gamma\left(t_i\right)\hookrightarrow\mathbb{U}$ for $0\leq i\leq n$ such that 
\[\dH^\mathbb{U}\left(\varphi_i\left(\gamma\left(t_i\right)\right),\varphi_j\left(\gamma\left(t_j\right)\right)\right)=|t_i-t_j|\rho,\,\forall 0\leq i,j\leq n.\]
By \Cref{lm:u-dgh-any}, there exists an isometric embedding $\varphi_{n+1}:\gamma\left(t_{n+1}\right)\hookrightarrow\mathbb{U}$ such that \[\dH^\mathbb{U}\big(\varphi_n\left(\gamma\left(t_n\right)\right),\varphi_{n+1}\left(\gamma\left(t_{n+1}\right)\right)\big)=|t_n-t_{n+1}|\rho.\]
Then, for any $i<n$, we have
\begin{align*}
    \dH^\mathbb{U}\big(\varphi_i\left(\gamma\left(t_i\right)\right),\varphi_{n+1}\left(\gamma\left(t_{n+1}\right)\right)\big)&\leq \dH^\mathbb{U}\big(\varphi_i\left(\gamma\left(t_i\right)\right),\varphi_{n}\left(\gamma\left(t_{n}\right)\right)\big)+\dH^\mathbb{U}\big(\varphi_n\left(\gamma\left(t_n\right)\right),\varphi_{n+1}\left(\gamma\left(t_{n+1}\right)\right)\big)\\
    &\leq \left(t_n-t_i\right)\rho+\left(t_{n+1}-t_n\right)\rho\\
    &=\left(t_{n+1}-t_i\right)\rho.
\end{align*}
Since $\dH^\mathbb{U}\big(\varphi_i\left(\gamma\left(t_i\right)\right),\varphi_{n+1}\left(\gamma\left(t_{n+1}\right)\right)\big)\geq \dgh\big(\varphi_i\left(\gamma\left(t_i\right)\right),\varphi_{n+1}\left(\gamma\left(t_{n+1}\right)\right)\big)=\left(t_{n+1}-t_i\right)\rho$, we have that $\dH^\mathbb{U}\big(\varphi_i\left(\gamma\left(t_i\right)\right),\varphi_{n+1}\left(\gamma\left(t_{n+1}\right)\right)\big)=\left(t_{n+1}-t_i\right)\rho$ for all $i<n$. This concludes the induction step.
\end{proof}

\begin{coro}\label{coro:finite-seq-h-dgh}
Given any Gromov-Hausdorff geodesic $\gamma:[0,1]\rightarrow \mathcal{M}$, we let $\rho\coloneqq \dgh\left(\gamma\left(0\right),\gamma\left(1\right)\right)$. Then, for any finite sequence $0\leq t_0<t_1<\ldots<t_n\leq 1$, there exist $X\in \ms$ and isometric embeddings $\varphi_i:\gamma\left(t_i\right)\hookrightarrow X$ for $i=0,\ldots,n$ such that
$$\dH^X\big(\varphi_i\left(\gamma\left(t_i\right)\right),\varphi_j\left(\gamma\left(t_j\right)\right)\big)=|t_i-t_j|\,\rho,\quad\forall0\leq i,j\leq n.$$
\end{coro}

\begin{proof}
By the previous lemma, there exist isometric embeddings $\varphi_i:\gamma\left(t_i\right)\hookrightarrow \mathbb{U}$ such that 
$$\dH^\mathbb{U}\big(\varphi_i\left(\gamma\left(t_i\right)\right),\varphi_j\left(\gamma\left(t_j\right)\right)\big)=|t_i-t_j|\rho.$$
Let $X\coloneqq\cup_{i=0}^n\varphi_i\left(\gamma_i\left(t_i\right)\right)\subseteq\mathbb{U}$. Then, since each $\varphi_i\left(\gamma_i\left(t_i\right)\right)$ is compact, we have that $X$ is compact and thus
$$\dH^X\big(\varphi_i\left(\gamma\left(t_i\right)\right),\varphi_j\left(\gamma\left(t_j\right)\right)\big)=\dH^\mathbb{U}\big(\varphi_i\left(\gamma\left(t_i\right)\right),\varphi_j\left(\gamma\left(t_j\right)\right)\big)=|t_i-t_j|\rho,\,\forall 0\leq i,j\leq n.$$
\end{proof}

\paragraph{Gromov-Wasserstein counterparts.} All the previous results in this section have counterparts for the Gromov-Wasserstein distance. We will list the most useful ones for later use and we delay the proofs to the end of this section. {Recall from \Cref{sec:gw-detail} that we use script letters such as $\mathcal{X}$ to denote metric measure spaces $\mathcal{X}=(X,d_X,\mu_X)$.}

\begin{lemma}\label{lm:u-dgw}
For any $p\in[1,\infty)$ and any $\mathcal{X},\mathcal{Y}\in\ws$, there exist isometric embeddings $\varphi_X:X\hookrightarrow \mathbb{U}$ and $\varphi_Y:Y\hookrightarrow \mathbb{U}$ such that 
$$\dgws{p}\left(\mathcal{X},\mathcal{Y}\right)=\dW{p}^\mathbb{U}\left((\varphi_X)_\#\mu_X,(\varphi_Y)_\#\mu_Y\right). $$
\end{lemma}

\begin{lemma}\label{lm:u-dgw-any}
For any $p\in[1,\infty)$ and any $\mathcal{X},\mathcal{Y}\in\ws$, let $\varphi_X:X\hookrightarrow\mathbb{U}$ be an isometric embedding. Then, there exists an isometric embedding $\varphi_Y:Y\hookrightarrow \mathbb{U}$ such that 
$$\dgws{p}\left(\mathcal{X},\mathcal{Y}\right)=\dW{p}^\mathbb{U}\left((\varphi_X)_\#\mu_X,(\varphi_Y)_\#\mu_Y\right). $$
\end{lemma}

\begin{lemma}\label{lm:pgw-geo-finite-embedding-u}
For any $p\in[1,\infty)$, let $\gamma:[0,1]\rightarrow \ws$ be an $\ell^p$-Gromov-Wasserstein geodesic. Let $\rho\coloneqq \dgws{p}\left(\gamma\left(0\right),\gamma\left(1\right)\right)$ and write $\gamma(t)\coloneqq(X_t,d_t,\mu_t)$ for each $t\in[0,1]$. Then, for any finite sequence $0\leq t_0<t_1<\ldots<t_n\leq 1$, there exist isometric embeddings $\varphi_i:X_{t_i}\hookrightarrow \mathbb{U}$ such that
$$\dW{p}^\mathbb{U}\left((\varphi_i)_\#\left(\mu_{t_i}\right),(\varphi_j)_\#\left(\mu_{t_j}\right)\right)=|t_i-t_j|\,\rho,\quad\forall0\leq i,j\leq n.$$
\end{lemma}

\begin{coro}\label{coro:finite-seq-w-dgw}
Assume the same conditions as in \Cref{lm:pgw-geo-finite-embedding-u}. Then, for any finite sequence $0\leq t_0<t_1<\ldots<t_n\leq 1$, there exist $X\in \ms$ and isometric embeddings $\varphi_i:X_{t_i}\hookrightarrow X$ for $i=0,\ldots,n$ such that 
$$\dW{p}^X\left((\varphi_i)_\#\mu_{t_i},(\varphi_j)_\#\mu_{t_j}\right)=|t_i-t_j|\,\rho,\quad\forall0\leq i,j\leq n.$$
\end{coro}

\paragraph{Relegated proofs.}
\begin{proof}[Proof of \Cref{lm:u-dgw}]
By \Cref{lm:dgw_w-realizable}, there exist $Z\in\ms$ and isometric embeddings $\psi_X:X\hookrightarrow Z$ and $\psi_Y:Y\hookrightarrow Z$ such that $\dW{p}^Z\left((\psi_X)_\#\mu_X,(\psi_Y)_\#\mu_Y\right)=\dgws{p}\left(\mathcal{X},\mathcal{Y}\right)$. Then, since $Z$ is compact, there exists an isometric embedding $\varphi:Z\hookrightarrow\mathbb{U}$ (cf. \Cref{thm:urysohn}). Let $\varphi_X=\varphi|_{\psi_X(X)}\circ\psi_X$ and $\varphi_Y=\varphi|_{\psi_Y(Y)}\circ\psi_Y$. Then,
\begin{align*}
    \dgws{p}\left(\mathcal{X},\mathcal{Y}\right)&= \dW{p}^Z\left((\psi_X)_\#\mu_X,(\psi_Y)_\#\mu_Y\right)\\
    &=\dW{p}^{\varphi\left(Z\right)}\left((\varphi_X)_\#\mu_X,(\varphi_Y)_\#\mu_Y\right)\\
    &=\dW{p}^\mathbb{U}\left((\varphi_X)_\#\mu_X,(\varphi_Y)_\#\mu_Y\right).
\end{align*}
\end{proof}

\begin{proof}[Proof of \Cref{lm:u-dgw-any}]
By \Cref{lm:u-dgw}, there exist isometric embeddings $\psi_X:X\hookrightarrow\mathbb{U}$ and $\psi_Y:Y\hookrightarrow\mathbb{U}$ such that $\dgws{p}\left(\mathcal{X},\mathcal{Y}\right)=\dW{p}^\mathbb{U}\left((\psi_X)_\#\mu_X,(\psi_Y)_\#\mu_Y\right).$ Now, both $\varphi_X\left(X\right)$ and $\psi_X\left(X\right)$ are compact subsets of $\mathbb{U}$ and $\tau\coloneqq\varphi_X\circ\psi_X^{-1}:\psi_X\left(X\right)\rightarrow\varphi_X\left(X\right)$ is an isometry. By \Cref{thm:u-compact-homo}, there exists an isometry $\tilde{\tau}:\mathbb{U}\rightarrow\mathbb{U}$ such that $\tilde{\tau}|_{\psi_X\left(X\right)}=\tau$. Let $\varphi_Y\coloneqq\tilde{\tau}\circ\psi_Y:Y\rightarrow\mathbb{U}$. It is clear that $\varphi_Y$ is an isometric embedding and thus
\begin{align*}
    \dW{p}^\mathbb{U}\left((\varphi_X)_\#\mu_X,(\varphi_Y)_\#\mu_Y\right)&= \dW{p}^\mathbb{U}\left((\tilde{\tau}^{-1})_\#\circ(\varphi_X)_\#\mu_X,(\tilde{\tau}^{-1})_\#\circ(\varphi_Y)_\#\mu_Y\right)\\
    &=\dW{p}^\mathbb{U}\left((\psi_X)_\#\mu_X,(\psi_Y)_\#\mu_Y\right)\\
    &=\dgws{p}\left(X,Y\right).
\end{align*}
\end{proof}

\begin{proof}[Proof of \Cref{lm:pgw-geo-finite-embedding-u}]
We prove the lemma by induction on $n$. When $n=0$, the statement holds true trivially. Now, suppose the statement holds for $n\geq 0$. Consider a sequence $0\leq t_0<t_1<\ldots<t_n<t_{n+1}\leq 1$. By the induction assumption, there exist isometric embeddings $\varphi_i:X_{t_i}\hookrightarrow\mathbb{U}$ for $0\leq i\leq n$ such that $\dW{p}^\mathbb{U}\left((\varphi_i)_\#\mu_{t_i},(\varphi_j)_\#\mu_{t_j}\right)=|t_i-t_j|\rho,\,\forall 0\leq i,j\leq n$. By \Cref{lm:u-dgw-any}, there exists an isometric embedding $\varphi_{n+1}:X_{t_{n+1}}\hookrightarrow\mathbb{U}$ such that $\dW{p}^\mathbb{U}\left((\varphi_n)_\#\mu_{t_n},(\varphi_{n+1})_\#\mu_{t_{n+1}}\right)=|t_n-t_{n+1}|\rho.$ Then, for any $i<n$, we have
\begin{align*}
    &\dW{p}^\mathbb{U}\left((\varphi_i)_\#\mu_{t_i},(\varphi_{n+1})_\#\mu_{t_{n+1}}\right)\\
    \leq &\dW{p}^\mathbb{U}\left((\varphi_i)_\#\mu_{t_i},(\varphi_{n})_\#\mu_{t_{n}}\right)+\dW{p}^\mathbb{U}\left((\varphi_n)_\#\mu_{t_n},(\varphi_{n+1})_\#\mu_{t_{n+1}}\right)\\
    \leq &\left(t_n-t_i\right)\rho+\left(t_{n+1}-t_n\right)\rho\\
    =&\left(t_{n+1}-t_i\right)\rho.
\end{align*}
Since $\dW{p}^\mathbb{U}\left((\varphi_i)_\#\mu_{t_i},(\varphi_{n+1})_\#\mu_{t_{n+1}}\right)\geq \dgws{p}\left((\varphi_i)_\#\mu_{t_i},(\varphi_{n+1})_\#\mu_{t_{n+1}}\right)=\left(t_{n+1}-t_i\right)\rho$, we have that $\dW{p}^\mathbb{U}\left((\varphi_i)_\#\mu_{t_i},(\varphi_{n+1})_\#\mu_{t_{n+1}}\right)=\left(t_{n+1}-t_i\right)\rho$ for all $i<n$. This concludes the induction step.
\end{proof}

\begin{proof}[Proof of \Cref{coro:finite-seq-w-dgw}]
By \Cref{lm:pgw-geo-finite-embedding-u}, there exist isometric embeddings $\varphi_i:X_{t_i}\hookrightarrow \mathbb{U}$ for $i=0,\ldots,n$ such that $\dW{p}^X\left((\varphi_i)_\#\mu_{t_i},(\varphi_j)_\#\mu_{t_j}\right)=|t_i-t_j|\rho.$ Let $X\coloneqq\cup_{i=0}^n\varphi_i\left(X_{t_i}\right)\subseteq\mathbb{U}$. Then, since each $\varphi_i\left(X_{t_i}\right)$ is compact, we have that $X$ is compact and thus for all $0\leq i,j\leq n$ we have
$$\dW{p}^X\left((\varphi_i)_\#\mu_{t_i},(\varphi_j)_\#\mu_{t_j}\right)=\dW{p}^\mathbb{U}\left((\varphi_i)_\#\mu_{t_i},(\varphi_j)_\#\mu_{t_j}\right)=|t_i-t_j|\rho.$$
\end{proof}

\section{Hausdorff and Wasserstein-realizable geodesics}\label{sec:real-geo}
In this section, we study both Hausdorff and Wasserstein-realizable geodesics and prove \Cref{thm:main-h-realizable} and \Cref{thm:bdd-geo-w-real}. Both proofs rely on convergence results of Lipschitz curves under certain metric extensors and we first study such convergence results in \Cref{sec:convergence}.
\subsection{Convergence}\label{sec:convergence}
Given a metric space $X$ and a Hausdorff convergent sequence\footnote{A Hausdorff convergent sequence in a metric space $X$ is a sequence of compact subsets of $X$ converging under the Hausdorff distance $\dH^X$.} of subsets $\{A_i\}_{i=0}^\infty$ with limit $A\subseteq X$, then $\lim_{i\rightarrow\infty}\dgh(A_i,A)=0$, since $\dgh(A_i,A)\leq \dH^X(A_i,A)$ for $i=0,\ldots$. Conversely, a Gromov-Hausdorff convergent sequence\footnote{A Gromov-Hausdorff convergent sequence is a sequence of compact metric spaces converging under the Gromov-Hausdorff distance $\dgh$.} of compact metric spaces $\{X_i\}_{i=0}^\infty$ with limit $X\in\ms$ can be realized as a Hausdorff convergent sequence in some ambient space. A similar statement was mentioned in \cite[Chapter 10]{petersen2006riemannian} whose proof operates by passing to a subsequence. We provide a proof for our statement which involves the Urysohn universal metric space (cf. \Cref{sec:urysohn}).

\begin{lemma}\label{lm:ghconv=hconv}
Let $\{X_i\}_{i=0}^\infty$ be a convergent sequence in $(\ms,\dgh)$ with limit $X\in\ms$. Then, there exist a \emph{Polish} metric space $Z$ and isometric embeddings $\varphi:X\hookrightarrow Z$ and $\varphi_i:X_i\hookrightarrow Z$ for $i=0,\ldots$ such that $\lim_{i\rightarrow\infty}\dH^Z(\varphi_i(X_i),\varphi(X))=0$.
\end{lemma}
\begin{proof}
Let $Z=\mathbb{U}$, the Urysohn universal metric space. Then, $Z$ is Polish. By universality (cf. \Cref{thm:urysohn}), there exists an isometric embedding $\varphi:X\hookrightarrow \mathbb{U}$. By \Cref{lm:u-dgh-any}, there exist isometric embeddings $\varphi_i:X_i\hookrightarrow \mathbb{U}$ such that $\dH^\mathbb{U}(\varphi(X),\varphi_i(X_i))=\dgh(X,X_i)$ for $i=0,\ldots$. Therefore, $\lim_{i\rightarrow\infty}\dH^Z(\varphi_i(X_i),\varphi(X))=0$.
\end{proof}

\paragraph{A generalized Arzel\`a-Ascoli theorem.} The version of Arzel\`a-Ascoli theorem in \Cref{thm:AA} requires a fixed range $X$ for all curves $\gamma_i:[0,1]\rightarrow X$. This can be generalized to curves $\gamma_i:[0,1]\rightarrow X_i$ with convergent ranges, i.e., $X_i$ converges (in a suitable sense) to $X\in\ms$ as $i$ approaches $\infty$. 

\begin{theorem}[Generalized Arzel\`a-Ascoli theorem]\label{thm:general-AA}
Let $(Z,d_Z)$ be a complete metric space and let $\{X_i\}_{i=0}^\infty$ be a Hausdorff convergent sequence of compact subsets of $Z$. Let $X\in \mathcal{H}(Z)$ be the limit of $\{X_i\}_{i=0}^\infty$ under $\dH^Z$. Let $\{\gamma_i:[0,1]\rightarrow X_i\}_{i=0}^\infty$ be a sequence of $C$-Lipschitz curves for some $C>0$ fixed. Then, there is a uniformly convergent (in the sense of $d_Z$) subsequence of $\{\gamma_i\}_{i=0}^\infty$ with a $C$-Lipschitz limit $\gamma:[0,1]\rightarrow X$.
\end{theorem}

\begin{proof}
Let $T\coloneqq\{t_n\}_{n=0}^\infty$ be a countable dense subset of $[0,1]$. Let $\rho_i\coloneqq\dH^Z(X_i,X)$ for $i=0,\ldots$. Then, for each $\gamma_i(t_0)\in X_i$, there exists $x_i^0\in X$ such that $d_Z(\gamma_i(t_0),x_i^0)\leq\rho_i$. Since $X$ is compact, there exists a subsequence of $\{x_i^0\}_{i=0}^\infty$, still denoted by $\{x_i^0\}_{i=0}^\infty$, converging to a point $x^0\in X$. Then, since $\rho_i\rightarrow0$ as $i\rightarrow\infty$, we have
$$0\leq\lim_{i\rightarrow\infty}d_Z(\gamma_i(t_0),x^0)\leq\lim_{i\rightarrow\infty}d_Z(\gamma_i(t_0),x_i^0)+ \lim_{i\rightarrow\infty}d_Z(x_i^0,x^0)=0,$$ 
and consequently, $\lim_{i\rightarrow\infty}d_Z(\gamma_i(t_0),x^0)=0$. We similarly consider $t_1, t_2$ and so on and construct $x^n\in X$ for $n=1,2,\ldots$ in a manner similar to the construction of $x^0$. Then, by a standard diagonal argument, there exist a subsequence of $\{X_i\}_{i=0}^\infty$ (still denoted by $\{X_i\}_{i=0}^\infty$) and points $x^n\in X$ for $n=0,\ldots$ such that $\lim_{i\rightarrow\infty}d_Z(\gamma_i(t_n),x^n)=0$ for $n=0,\ldots$. Then, for $m,n\in\mathbb{N}$, we have
\begin{equation}\label{eq:gamma-T-C-Lip}
    d_Z(x^m,x^n)=\lim_{i\rightarrow\infty}d_Z(\gamma_i(t_m),\gamma_i(t_n))\leq C\cdot |t_m-t_n|.
\end{equation}

Now, we define $\gamma:[0,1]\rightarrow X$ as follows:
$$\gamma(t) \coloneqq \begin{cases} x^n,&t=t_n\in T\\
\lim_{j\rightarrow\infty} x^{n_j},&t\in [0,1]\backslash T\text{ and }\{t_{n_j}\}_{j=0}^\infty\text{ is a subsequence of }T\text{ converging to }t
\end{cases}.$$
The existence of the limit $\lim_{j\rightarrow\infty} x^{n_j}$ is due to completeness of $Z$ and \Cref{eq:gamma-T-C-Lip}. It is obvious that $\gamma(t)$ is well-defined, i.e., its image is independent of the choice of $\{t_{n_j}\}_{j=0}^\infty$. It is easy to check that $\gamma$ is also $C$-Lipschitz. Now, it remains to prove that $\{\gamma_i\}_{i=0}^\infty$ uniformly converges to $\gamma$. For any $\eps>0$, pick a finite subsequence $T_N\coloneqq\{t_0,t_1,\ldots, t_N\}$ of $T$ (possibly relabeled) such that $T_N$ is an $\frac{\eps}{3C}$-net of $[0,1]$. By definition of $\gamma$ on $T$, there exists $M>0$ such that for all $i>M$ and for all $n=0,\ldots,N$, we have $d_Z(\gamma(t_n),\gamma_i(t_n))\leq\frac{\eps}{3}$. Then, for any $t\in [0,1]$, there exists $t_n\in T_N$ such that $|t-t_n|\leq \frac{\eps}{3C}$ and that for $i>M$
\begin{align*}
    d_Z(\gamma(t),\gamma_i(t))&\leq d_Z(\gamma(t),\gamma(t_n))+d_Z(\gamma(t_n),\gamma_i(t_n))+d_Z(\gamma_i(t_n),\gamma_i(t))\\
    &\leq C\cdot |t-t_n|+\frac{\eps}{3} + C\cdot |t-t_n|\\
    &\leq \eps.
\end{align*}
This implies that $\{\gamma_i\}_{i=0}^\infty$ converges to $\gamma$ uniformly. 
\end{proof}

\subsection{Hausdorff-realizable geodesics}\label{sec:haus-real-geo}
We first define Hausdorff-realizable geodesics as follows:
\begin{definition}[Hausdorff-realizable geodesic]
A geodesic $\gamma:[0,1]\rightarrow\ms$ is called \emph{Hausdorff-realizable}, if there exist $X\in\ms$ and for each $t\in[0,1]$ isometric embedding $\varphi_t:\gamma\left(t\right)\hookrightarrow X$ such that
$$\dH^X\big(\varphi_s\left(\gamma\left(s\right)\right),\varphi_t\left(\gamma\left(t\right)\right)\big)=\dgh\left(\gamma\left(s\right),\gamma\left(t\right)\right),\quad\forall s,t\in[0,1].$$ 
In this case, we say that $\gamma$ is \emph{$X$-Hausdorff-realizable}.
\end{definition}

\begin{remark}[Hausdorff-realizable geodesics are Hausdorff geodesics]
{Suppose that $\gamma:[0,1]\rightarrow\ms$ is a $X$-Hausdorff-realizable Gromov-Hausdorff geodesic via the family of isometric embeddings 
\[\{\varphi_t:\gamma(t)\hookrightarrow X\}_{t\in[0,1]}.\]
Then, obviously the curve defined by $t\mapsto\varphi_t(\gamma(t))$ for $t\in[0,1]$ is a geodesic in the Hausdorff hyperspace $\mathcal{H}(X)$ of $X$. This is the converse of \Cref{lm:hgeo-to-dghgeo}. In short, we emphasize that \emph{a Hausdorff-realizable Gromov-Hausdorff geodesic is the same as a Hausdorff geodesic with respect to some underlying ambient space}.}
\end{remark}

\begin{example}[Trivial Gromov-Hausdorff geodesics are Hausdorff-realizable]
Let $\gamma:[0,1]\rightarrow\ms$ be a ``trivial" Gromov-Hausdorff geodesic, i.e., there exists $X\in\ms$ such that $\gamma(t)\cong X$ for all $t\in[0,1]$. Then, it is obvious that $\gamma$ is $X$-Hausdorff-realizable.
\end{example}

In \cite{ivanov2019hausdorff}, the authors show that any straight-line $\dgh$ geodesic can be Hausdorff-realized in a metric space:

\begin{proposition}[{\cite[Corollary 3.1]{ivanov2019hausdorff}}]\label{prop:hausdorff-geodesic-straight-line}
Let $X,Y\in\ms$ and $R\in\mathcal{R}^\mathrm{opt}\left(X,Y\right)$ and let $\rho\coloneqq\dgh\left(X,Y\right)$. Let $\gamma_R$ be the straight-line $\dgh$ geodesic connecting $X$ and $Y$ based on $R$ (cf. \Cref{thm:str-line-geo}). Let $Z\coloneqq R\times[0,1]$ and define $d_Z:Z\times Z\rightarrow\mathbb{R}_+$ by
$$d_Z\left(\left(\left(x,y\right),t\right),\left(\left(x',y'\right),t'\right)\right)\coloneqq\inf_{\left(x'',y''\right)\in R}\left(d_{R_t}\left(\left(x,y\right),\left(x'',y''\right)\right)+d_{R_{t'}}\left(\left(x',y'\right),\left(x'',y''\right)\right)\right)+\rho\,|t-t'|,$$
for any $\left(x,y\right),\left(x',y'\right)\in R$ and $t,t'\in [0,1]$. Then, by canonically identifying $R_t$ with $R\times\{t\}\subseteq Z$, we have $\dH^Z\left(R_s,R_t\right)=\dgh\left(R_s,R_t\right)$.
\end{proposition}

\begin{remark}\label{rmk:quotient-str-line}
In fact, $Z$ is a pseudo-metric space since for any $\left(x,y\right),\left(x,y'\right)\in R$, we have at $t=0$ that $d_Z\big(\left(\left(x,y\right),0\right),\left(\left(x,y'\right),0\right)\big)=d_X\left(x,x\right)=0$. A similar result holds for $t=1$. By identifying points at zero distance, we obtain a new metric space $\tilde{Z}$. It is obvious that the result in \Cref{prop:hausdorff-geodesic-straight-line} still holds by replacing $Z$ with $\tilde{Z}$.
\end{remark}

Now, we show that the space constructed in \Cref{prop:hausdorff-geodesic-straight-line} (or more precisely the quotient space discussed in the remark above) is compact for compact correspondences and thus show that straight-line $\dgh$ geodesics corresponding to compact correspondences are Hausdorff-realizable.
\begin{proposition}\label{prop:strline-hausdorff}
Assuming the same notation as in \Cref{prop:hausdorff-geodesic-straight-line}, if $R$ is compact in the product space $(X\times Y,d_{X\times Y}\coloneqq\max(d_X,d_Y))$, then
$\gamma_R$ is Hausdorff-realizable.
\end{proposition}

\begin{proof}
We only need to show that the construction $Z$ in \Cref{prop:hausdorff-geodesic-straight-line} is sequentially compact. Then, the metric space $\tilde{Z}$ in \Cref{rmk:quotient-str-line} is also sequentially compact and thus compact.

For any sequence $\{\left(\left(x_i,y_i\right),t_i\right)\}_{i=1}^\infty$ in $Z$, by compactness of $[0,1]$ and $R$, there exists a subsequence, still denoted by $\{\left(\left(x_i,y_i\right),t_i\right)\}_{i=1}^\infty$, such that $\{t_i\}_{i=1}^\infty$ converges to some $t\in[0,1]$ and that $\{\left(x_i,y_i\right)\}_{i=1}^\infty$ converges to some $\left(x,y\right)\in R$ under $d_{X\times Y}\coloneqq\max\left(d_X,d_Y\right)$. 

Now, we show that $\lim_{i\rightarrow\infty}d_Z\big(\left(\left(x_i,y_i\right),t_i\right),\left(\left(x,y\right),t\right)\big)=0$. Indeed,
\begin{align*}
0\leq d_Z\big(\left(\left(x,y\right),t\right),\left(\left(x_i,y_i\right),t_i\right)\big)&=\inf_{\left(x',y'\right)\in R}\left(d_{R_t}\left(\left(x,y\right),\left(x',y'\right)\right)+d_{R_{t_i}}\left(\left(x',y'\right),\left(x_i,y_i\right)\right)\right)+\rho\,|t-t_i|\\
&\leq d_{R_t}\left(\left(x,y\right),\left(x_i,y_i\right)\right)+\rho\,|t-t_i|\\
&=\left(1-t\right)\,d_X\left(x,x_i\right)+t\,d_Y\left(y,y_i\right)+\rho\,|t-t_i|\\
&\leq d_{X\times Y}\left(\left(x,y\right),\left(x_i,y_i\right)\right)+\rho\,|t-t_i|.
\end{align*}
Then, by assumptions on the sequence $\{((x_i,y_i),t_i)\}_{i=1}^\infty$, 
$$\lim_{i\rightarrow\infty}d_Z\big(\left(\left(x_i,y_i\right),t_i\right),\left(\left(x,y\right),t\right)\big)=0.$$ 
As a result, $Z$ is sequentially compact and hence we conclude the proof.
\end{proof}

The following lemma provides an interesting description of Hausdorff-realizable geodesics: for any $X$-Hausdorff-realizable geodesic $\gamma$, there is a smallest closed subset $\mathcal{G}_X\subseteq X$ which Hausdorff-realizes $\gamma$.

\begin{lemma}\label{lm:union-geo-closed}
Given a compact metric space $X$ and a Hausdorff geodesic $\gamma:[0,1]\rightarrow \mathcal{H}\left(X\right)$, the union $\mathcal{G}_X\coloneqq\cup_{t\in[0,1]}\gamma\left(t\right)$ is a closed (and thus compact) subset of $X$.
\end{lemma}

\begin{proof}
Let $\rho\coloneqq\dH^X\left(\gamma\left(0\right),\gamma\left(1\right)\right)$. If $\rho=0$, then $\gamma(t)=\gamma(0)$ for all $t\in[0,1]$. Thus $\mathcal{G}_X=\gamma(0)$ is closed. 

Now, we assume that $\rho>0$. Fix an arbitrary $x\in X$. Define $f_x:[0,1]\rightarrow\mathbb{R}$ by taking $t\in[0,1]$ to $d_X\left(x,\gamma\left(t\right)\right)\coloneqq\inf\{d_X\left(x,x_t\right):\,x_t\in\gamma\left(t\right)\}$. We first show that $f_x$ is continuous. Fix $t_0\in[0,1]$. Since $\gamma\left(t_0\right)$ is compact, there exists $x_{t_0}\in\gamma\left(t_0\right)$ such that $f_x\left(t_0\right)=d_X\left(x,x_{t_0}\right)$. For each $\eps>0$, let $\delta=\frac{\eps}{\rho}>0$. For any $t\in[0,1]$ such that $|t-t_0|<\delta$, there exists $x_t\in \gamma\left(t\right)$ such that \[d_X\left(x_t,x_{t_0}\right)\leq \dH^X\left(\gamma\left(t\right),\gamma\left(t_0\right)\right)=|t-t_0|\rho<\delta\rho=\eps.\]
Then,
$$f_x\left(t\right)\leq d_X\left(x,x_t\right)\leq d_X\left(x,x_{t_0}\right)+d_X\left(x_{t_0},x_t\right)< f_x\left(t_0\right)+\eps.$$
Now, assume $x_t'\in\gamma\left(t\right)$ is such that $f_x\left(t\right)=d_X\left(x,x_t'\right)$. Let $x_{t_0}'\in\gamma\left(t_0\right)$ be such that
\[d_X\left(x_{t_0}',x_t'\right)\leq\dH^X\left(\gamma\left(t_0\right),\gamma\left(t\right)\right)<\eps.\]
Then, 
$$f_x\left(t\right)= d_X\left(x,x_t'\right)\geq d_X\left(x,x_{t_0}'\right)-d_X\left(x_{t_0}',x_t'\right)> f_x\left(t_0\right)-\eps.$$
Therefore, $|f_x\left(t\right)-f_x\left(t_0\right)|<\eps$ for any $|t-t_0|<\delta$. This implies the continuity of $f_x$.

Let $\{x_i\}_{i=0}^\infty$ be a convergent sequence in $X$ such that $x_i\in\gamma\left(t_i\right)$ for some $t_i\in[0,1]$ and  $i=0,\ldots$. Suppose $x\in X$ is its limit. Assume that $x\not\in\mathcal{G}_X$ and thus $f_x\left(t\right)>0$ for each $t\in[0,1]$. Then, by continuity of $f_x$, there exists a constant $c>0$, such that $f>c$ on $[0,1]$. But $f_x\left(t_i\right)\leq d_X\left(x_i,x\right)$ and the right hand side approaches 0 as $i\rightarrow\infty$, a contradiction. Hence, there exists $t\in[0,1]$ such that $f\left(t\right)=0$ and thus $x\in\gamma\left(t\right)\subseteq\mathcal{G}_X$. This proves that $\mathcal{G}_X$ is closed in $X$. 
\end{proof}

Let $\Gamma$ be the collection of all Gromov-Hausdorff geodesics $\gamma:[0,1]\rightarrow\ms$. Let $d_\infty$ be the uniform metric on $\Gamma$, i.e., $d_\infty\left(\gamma_1,\gamma_2\right)\coloneqq\sup_{t\in[0,1]}\dgh\left(\gamma_1\left(t\right),\gamma_2\left(t\right)\right)$ for any $\gamma_1,\gamma_2\in\Gamma$. Let $\Gamma_\mathcal{H}$ denote the subset of $\Gamma$ consisting of all Hausdorff-realizable geodesics in $\ms$. Then, \Cref{thm:main-h-realizable} is equivalent to saying that $\Gamma_\mathcal{H}=\Gamma$. Before proving \Cref{thm:main-h-realizable}, we apply the properties developed in \Cref{sec:W-geo} towards proving the following {preliminary} result:

\begin{proposition}\label{prop:main-h-dense}
$\Gamma_\mathcal{H}$ is a dense subset of $\Gamma$.
\end{proposition}

\begin{proof}
Fix any Gromov-Hausdorff geodesic $\gamma:[0,1]\rightarrow\ms$ with $\rho\coloneqq\dgh\left(\gamma\left(0\right),\gamma\left(1\right)\right)>0$ and $\eps>0$. Let $0=t_0<\ldots<t_n=1$ be a sequence such that $t_{i+1}-t_i<\frac{\eps}{2\rho}$ for $i=0,\ldots,n-1$. Then, by \Cref{coro:finite-seq-h-dgh}, there exist $X\in\ms$ and isometric embeddings $\varphi_i:\gamma\left(t_i\right)\hookrightarrow X$ such that $\dH^X\big(\varphi_i(\gamma\left(t_i\right)),\varphi_j\lc\gamma\left(t_j\right)\rc\big)=|t_i-t_j|\rho.$ Let $Z\coloneqq \mathcal{W}_1\left(X\right)$ and still denote by $\varphi_i$ the composition of $\varphi_i:\gamma(t_i)\hookrightarrow X$ and the canonical embedding $X\hookrightarrow\mathcal{W}_1(X)=Z$ for $i=0,\ldots,n$. Then, we still have $\dH^Z\big(\varphi_i(\gamma\left(t_i\right)),\varphi_j(\gamma\left(t_j\right))\big)=|t_i-t_j|\rho.$ Since $Z$ is compact and geodesic (cf. \Cref{thm:compact-w} and \Cref{thm:W-geodesic}), $\mathcal{H}(Z)$ is geodesic (cf. \Cref{thm:hgeo}). Hence, there exist Hausdorff geodesics $\gamma_i:[0,1]\rightarrow \mathcal{H}\left(Z\right)$ such that $\gamma_i\left(0\right)=\varphi_i(\gamma\left(t_i\right))$ and $\gamma_i\left(1\right)=\varphi_{i+1}(\gamma\left(t_{i+1}\right))$ for $i=0,\ldots,n-1$. Then, $\gamma_i(1)=\gamma_{i+1}(0)$ for $i=0,\ldots,n-1$ and
\begin{align*}
    \dH^Z(\gamma_0(0),\gamma_{n-1}(1))&=\dH^Z\big(\varphi_0(\gamma(t_0)),\varphi_n(\gamma(t_n))\big)\\
    &=\sum_{i=0}^{n-1}\dH^Z\big(\varphi_i(\gamma(t_i)),\varphi_{i+1}(\gamma(t_{i+1}))\big)\\
    &=\sum_{i=0}^{n-1}\dH^Z\big(\gamma_i(0),\gamma_{i}(1))\big).
\end{align*}

Therefore, we can concatenate (cf. \Cref{prop:geo-concatenate}) all the $\gamma_i$s to obtain a new geodesic $\tilde{\gamma}:[0,1]\rightarrow \mathcal{H}\left(Z\right)$ such that $\tilde{\gamma}(t_i)=\varphi_i(\gamma(t_i))$ for each $i=0,\ldots,n$. By \Cref{lm:hgeo-to-dghgeo}, $\tilde{\gamma}$ is a Gromov-Hausdorff geodesic and by construction, $\tilde{\gamma}\in\Gamma_\mathcal{H}$. Now, for any $t\in[0,1]$, suppose $t\in[t_i,t_{i+1}]$ for some $i\in\{0,\ldots,n-1\}$. Then, 
\begin{align*}
    \dgh\left(\gamma\left(t\right),\tilde{\gamma}\left(t\right)\right)&\leq\dgh\left(\gamma\left(t\right),{\gamma}\left(t_i\right)\right)+\dgh\left(\gamma\left(t_i\right),\tilde{\gamma}\left(t_i\right)\right)+\dgh\left(\tilde{\gamma}\left(t_i\right),\tilde{\gamma}\left(t\right)\right) \\
    &=|t-t_i|\,\rho+0+|t-t_i|\,\rho\leq 2\cdot \frac{\eps}{2\rho}\cdot\rho =\eps.
\end{align*}
So, $d_\infty(\gamma,\tilde{\gamma})\leq\eps$. Therefore, we conclude that $\Gamma_\mathcal{H}$ is dense in $\Gamma$.
\end{proof}

\paragraph{Proof of \Cref{thm:main-h-realizable}.} Now, we obtain a proof of \Cref{thm:main-h-realizable} by showing that $\Gamma_\mathcal{H}$ is closed in $\Gamma$. {The proof is an intricate application of the generalized Arzel\`a-Ascoli theorem (\Cref{thm:general-AA}). In order to meet the conditions in \Cref{thm:general-AA}, one needs to exploit the stability of the Hausdorff extensor (\Cref{thm:H-equal}) and carefully leverage Gromov's pre-compactness theorem (\Cref{thm:pre-compact}).}

\thmhreal*

\begin{proof}
By \Cref{prop:main-h-dense}, we only need to show that $\Gamma_\mathcal{H}$ is closed in $\Gamma$.
Let $\{\gamma_i:[0,1]\rightarrow\ms\}_{i=0}^\infty$ be a Cauchy sequence in $\Gamma_\mathcal{H}$ with a limit $\gamma:[0,1]\rightarrow \mathcal{M}$ in $\Gamma$, i.e., $\lim_{i\rightarrow\infty}d_\infty(\gamma_i,\gamma)=0$. Moreover, we let $\rho\coloneqq \dgh\left(\gamma\left(0\right),\gamma\left(1\right)\right)$. 
\begin{claim}\label{clm:key-in-main}
There exist $X_i\in\ms$ such that $\gamma_i$ is $X_i$-Hausdorff-realizable for $i=0,\ldots$ and moreover $\{X_i\}_{i=0}^\infty$ has a $\dgh$-convergent subsequence.
\end{claim}

Assuming the claim for now, suppose $X\in\ms$ is such that $\lim_{i\rightarrow\infty}\dgh(X_i,X)=0$ (after possibly passing to a subsequence), then by \Cref{lm:ghconv=hconv}, there exist a Polish metric space $Z$ and isometric embeddings $\varphi:X\hookrightarrow Z$ and $\varphi_i:X_i\hookrightarrow Z$ for $i=0,\ldots$ such that $\lim_{i\rightarrow\infty}\dH^Z(\varphi_i(X_i),\varphi(X))=0$. By stability of the Hausdorff extensor $\mathcal{H}$ (cf. \Cref{thm:H-equal}), 
$$\lim_{i\rightarrow\infty}\dH^{\mathcal{H}(Z)}\big((\varphi_i)_*(\mathcal{H}(X_i)),\varphi_*(\mathcal{H}(X))\big)=\lim_{i\rightarrow\infty}\dH^Z(\varphi_i(X_i),\varphi(X))=0.$$ 
For any $\eps>0$, there exists $K(\eps)>0$ such that $\gamma_i$ is $(\rho+\eps)$-Lipschitz for $i\geq K(\eps)$, since $\gamma_i$ is $\dgh(\gamma_i(0),\gamma_i(1))$-Lipschitz and $\dgh(\gamma_i(0),\gamma_i(1))\rightarrow\rho$ as $i\rightarrow\infty$. For each $i=0,\ldots$, since $\gamma_i$ is $X_i$-Hausdorff-realizable, we can view $\gamma_i:[0,1]\rightarrow\ms$ as a Hausdorff geodesic $\gamma_i:[0,1]\rightarrow \mathcal{H}(X_i)$. Moreover, by \Cref{thm:hyper-complete}, $\mathcal{H}(Z)$ is complete. Then, by the generalized Arzel\`a-Ascoli theorem (\Cref{thm:general-AA}), the sequence $\{\gamma_i:[0,1]\rightarrow \mathcal{H}(X_i)\}_{i=K(\eps)}^\infty$ uniformly converges to a $(\rho+\eps)$-Lipschitz curve $\tilde{\gamma}:[0,1]\rightarrow \mathcal{H}(X)$, where we identify $\mathcal{H}(X_i)$ with $\varphi_i(\mathcal{H}(X_i))$ and $\mathcal{H}(X)$ with $\varphi(\mathcal{H}(X))$. Obviously, for $0<\eps'<\eps$, the subsequence $\{\gamma_i\}_{i=K(\eps')}^\infty$ uniformly converges to the same curve $\tilde{\gamma}$ and thus $\tilde{\gamma}$ is $(\rho+\eps')$-Lipschitz. Since $\eps'$ is arbitrary, $\tilde{\gamma}$ is $\rho$-Lipschitz. By uniform convergence, we have that for each $t\in[0,1]$, $\lim_{i\rightarrow\infty}\dH^{Z}(\tilde{\gamma}(t),\gamma_i(t))=0$. We know that $\gamma(t)$ is the Gromov-Hausdorff limit of $\{\gamma_i(t)\}_{i=0}^\infty$, thus $\tilde{\gamma}(t)\cong\gamma(t)$ for all $t\in[0,1]$. Since $\tilde{\gamma}$ is $\rho$-Lipschitz, we have that for each $s,t\in[0,1]$
$$\dH^X(\tilde{\gamma}(s),\tilde{\gamma}(t))\leq |s-t|\rho=\dgh(\gamma(s),\gamma(t))\leq \dH^X(\tilde{\gamma}(s),\tilde{\gamma}(t)).$$
Therefore, $\dH^X(\tilde{\gamma}(s),\tilde{\gamma}(t))=\dgh(\gamma(s),\gamma(t)) $ for $s,t\in[0,1]$ and thus $\gamma\in\Gamma_\mathcal{H}$. The structure of the argument above is also captured in \Cref{fig:proof-thm1}.

Now, we finish by proving \Cref{clm:key-in-main}
\begin{proof}[Proof of \Cref{clm:key-in-main}]
Since $\gamma_i\in\Gamma_\mathcal{H}$, there exists $Y_i\in\ms$ such that $\gamma_i$ is $Y_i$-Hausdorff-realizable. Then, let $X_i\coloneqq\mathcal{G}_{Y_i}=\cup_{t\in[0,1]}\gamma_i(t)$ as in \Cref{lm:union-geo-closed}. Here $\gamma_i(t)$ also denotes the isometric copy of itself into $Y_i$ and thus one can view $\gamma_i(t)$ as an element of $\mathcal{H}(Y_i)$. It is obvious that $\gamma_i$ is also $X_i$-Hausdorff-realizable. Now, we prove that $\{{X_i}\}_{i=0}^\infty$ has a convergent subsequence via Gromov's pre-compactness theorem (cf. \Cref{thm:pre-compact}). Let $\rho_i\coloneqq\dgh(\gamma_i(0),\gamma_i(1))$ for $i=0,1,\ldots$.

\begin{enumerate}
    \item Fix $i\in\mathbb{N}$ and $t\in[0,1]$. Then,
    $$\dH^{Y_i}(\gamma_i(t),\gamma_i(0))=\dgh(\gamma_i(t),\gamma_i(0))=t\rho_i\leq\rho_i.$$
    Therefore, for any $x_t\in\gamma_i(t)$, there exists $x_0\in\gamma_i(0)$ such that $d_{Y_i}(x_t,x_0)\leq\rho_i$. This implies that $\gamma_i(t)\subseteq(\gamma_i(0))^{\rho_i}\subseteq Y_i$. Since $t$ is arbitrary, we have that $X_i=\mathcal{G}_{Y_i}=\cup_{t\in[0,1]}\gamma_i(t)\subseteq \left(\gamma_i(0)\right)^{\rho_i}$. Therefore, $\diam\left(X_i\right)\leq 2\rho_i+\diam\left(\gamma_i\left(0\right)\right)$ for any $i=0,\ldots$. Since $\{\diam(\gamma_i(0))\}_{i=0}^\infty$ approaches $\diam(\gamma(0))$ and $\{\rho_i\}_{i=0}^\infty$ approaches $\rho$ as $i\rightarrow\infty$, there exists $\delta>0$ such that $\rho_i\leq \delta$ and $\diam(\gamma_i(0))\leq\delta$ for all $i=0,\ldots$. Therefore, $\{X_i\}_{i=0}^\infty$ is uniformly bounded by $3\delta$.
    
    \item For any $\eps>0$, pick $0=t_0<t_1<\ldots<t_N=1$ such that $t_{n+1}-t_n<\frac{\eps}{2\delta}$ for $n=0,\ldots,N-1$. Let $S_n\coloneqq\{s_n(k):\,k=0,\ldots,k_n\}$ be an $\frac{\eps}{4}$-net of $\gamma(t_n)$ for $n=0,\ldots,N$. Let $M>0$ be a positive integer such that $d_\infty(\gamma_i,\gamma)\leq \frac{\eps}{8}$ for all $i>M$. For $n\in\{0,\ldots,N\}$ and $i>M$, let $R_n^i\in\mathcal{R}^\mathrm{opt}\lc\gamma_i(t_n),\gamma(t_n)\rc$ be an optimal correspondence. Then,
    \[\dis(R_n^i)=2\dgh(\gamma_i(t_n),\gamma(t_n))\leq 2d_\infty(\gamma_i,\gamma)\leq \frac{\eps}{4}.\]
    For each $s_n(k)\in S_n$, choose $s_n^i(k)\in\gamma_i(t_n)$ such that $(s_n^i(k),s_n(k))\in R_n^i$. Then, we have that $S_n^i\coloneqq\{s_n^i(k)\}_{k=0}^{k_n}$ is an $\frac{\eps}{2}$-net of $\gamma_i(t_n)$. Indeed, for any $x_n^i\in \gamma_i(t_n)$, there exists $x_n\in\gamma(t_n)$ such that $(x_n^i,x_n)\in R_n^i$. Let $s_n(k)\in S_n$ be such that $d_{\gamma(t_n)}(x_n,s_n(k))\leq\frac{\eps}{4}$. Then, 
    $$d_{X_i}\lc x_n^i, s_n^i(k)\rc=d_{\gamma_i(t_n)}\lc x_n^i, s_n^i(k)\rc\leq \dis(R_n^i)+d_{\gamma(t_n)}(x_n,s_n(k))\leq \frac{\eps}{4}+\frac{\eps}{4}= \frac{\eps}{2}. $$
    Furthermore, we prove that $\cup_{n=0}^N S_n^i$ is an $\eps$-net of $X_i$. For each $t\in[0,1]$, suppose $t\in[t_n,t_{n+1}]$ for some $n\in\{0,\ldots,N-1\}$. For any $x_t^i\in\gamma_i(t)$, since 
    \[\dH^{X_i}(\gamma_i(t),\gamma_i(t_n))\leq |t-t_n|\rho_i\leq \frac{\eps}{2\delta}\cdot \delta=\frac{\eps}{2},\] 
    there exists $x_n^i\in\gamma_i(t_n)$ such that $d_{X_i}\lc x_n^i,x_t^i\rc\leq \frac{\eps}{2}$. Then, there exists $s_n^i(k)\in S_n^i$ such that $d_{X_i}\lc x_n^i,s_n^i(k)\rc\leq \frac{\eps}{2}$ and thus 
    $$d_{X_i}\lc x_t^i,s_n^i(k)\rc\leq d_{X_i}\lc x_t^i,x_n^i\rc+d_{X_i}\lc x_n^i,s_n^i(k)\rc\leq \eps.$$ 
    
    Now, note that $\left|\cup_{n=0}^NS_n^i\right|\leq \sum_{n=0}^Nk_n$ for each $i>M$. Let
    \[Q\left(\eps\right)\coloneqq\max\left(\max\{\mathrm{cov}_\eps\left(X_i\right):\,i=1,\ldots,M\},\sum_{n=0}^Nk_n\right),\]
    then we have $\mathrm{cov}_\eps\left(X_i\right)\leq Q\left(\eps\right)$ for all $i=0,\ldots$.
\end{enumerate}
Therefore, $\{X_i\}_{i=0}^\infty\subseteq \mathcal{K}\left(Q,3\delta\right)$ (cf. \Cref{def:CND}). By Gromov's pre-compactness theorem (cf. \Cref{thm:pre-compact}), $\{X_i\}_{i=0}^\infty$ has a convergent subsequence.
\end{proof}
\end{proof}

\begin{figure}[htb]
	\centering		\includegraphics[width=0.8\textwidth]{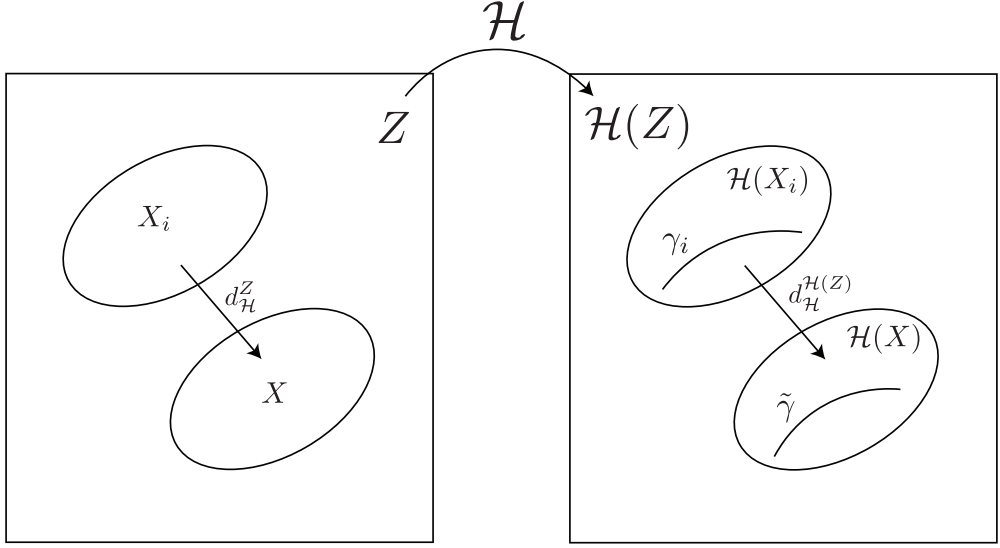}
	\caption{\textbf{Illustration of the proof of \Cref{thm:main-h-realizable}.} In this figure, we identify $X$ with $\varphi(X)$ and $\mathcal{H}(X)$ with $\varphi_*(\mathcal{H}(X))$ and similarly for $X_i$ and $\mathcal{H}(X_i)$. The figure illustrates our main strategy for proving \Cref{thm:main-h-realizable} as follows: we transform the $\dH^Z$ convergent sequence $\{X_i\}_{i=0}^\infty$ to a $\dH^{\mathcal{H}(Z)}$ convergent sequence $\{\mathcal{H}(X_i)\}_{i=0}^\infty$; then we use the generalized Arzel\`a-Ascoli theorem to establish a limit $\tilde{\gamma}:[0,1]\rightarrow\mathcal{H}(X)$ for the Lipschitz curves $\{\gamma_i:[0,1]\rightarrow\mathcal{H}(X_i)\}_{i=0}^\infty$; finally, we show that $\tilde{\gamma}$ coincides with $\gamma$.} \label{fig:proof-thm1}
\end{figure}

\subsection{Wasserstein-realizable geodesics}\label{sec:W-reall-geo}
We first specify the definition of Wasserstein-realizable geodesics as follows.
\begin{definition}[Wasserstein-realizable geodesic]\label{def:w-real-geo}
For $p\in[1,\infty)$, an $\ell^p$-Gromov-Wasserstein geodesic $\gamma:[0,1]\rightarrow\left(\ws,\dgws{p}\right)$ where $\gamma(t)\coloneqq(X_t,d_t,\mu_t)$ for $t\in[0,1]$ is called $\ell^p$-\emph{Wasserstein-realizable} (or simply Wasserstein-realizable), if there exist $X\in\ms$ and for each $t\in[0,1]$ isometric embedding $\varphi_t:X_t\hookrightarrow X$ such that
$$\dW{p}^X\left((\varphi_s)_\#\mu_s,(\varphi_t)_\#\mu_t\right)=\dgws{p}\left(\gamma\left(s\right),\gamma\left(t\right)\right),\quad\forall s,t\in[0,1].$$ 
In this case, we say that $\gamma$ is \emph{$X$-Wasserstein-realizable}.
\end{definition}

\begin{remark}[Wasserstein-realizable geodesics are Wasserstein geodesics]
{Suppose that an $\ell^p$-Gromov-Wasserstein geodesic $\gamma:[0,1]\rightarrow(\ws,\dgws{p})$ is $X$-Wasserstein-realizable via the following family of isometric embeddings 
\[\{\varphi_t:\gamma(t)\hookrightarrow X\}_{t\in[0,1]}.\]
Denote $\gamma(t)=(X_t,d_t,\mu_t)$ for each $t\in[0,1]$. Then, obviously $t\mapsto(\varphi_t)_\#\mu_t$ for $t\in[0,1]$ is a geodesic in the Wasserstein hyperspace $\mathcal{W}_p(X)$ of $X$. This is the converse of \Cref{lm:wgeo-to-dgwgeo}. In words, \emph{a Wasserstein-realizable Gromov-Wasserstein geodesic is a Wasserstein geodesic.}}
\end{remark}

\begin{example}[Trivial Gromov-Wasserstein geodesics are Wasserstein-realizable]
Let $\gamma:[0,1]\rightarrow\ws$ be a ``trivial" Gromov-Wasserstein geodesic, i.e., there exists $\mathcal{X}=(X,d_X,\mu_X)\in\ws$ such that $\gamma(t)\cong_w\mathcal{X}$ for all $t\in[0,1]$. Then, it is obvious that $\gamma$ is $X$-Wasserstein-realizable.
\end{example}

Now, we recall some notation. Let $\Gamma^p$ be the collection of all $\ell^p$-Gromov-Wasserstein geodesics. Let $d_{\infty,p}$ be the uniform metric on $\Gamma^p$, i.e., for any $\gamma_1,\gamma_2\in\Gamma^p$
$$d_{\infty,p}\left(\gamma_1,\gamma_2\right)\coloneqq\sup_{t\in[0,1]}\dgws{p}\left(\gamma_1\left(t\right),\gamma_2\left(t\right)\right).$$ 
Let $\Gamma_\mathcal{W}^p$ denote the subset of $\Gamma^p$ consisting of all Wasserstein-realizable geodesics in $\ws$.

\begin{restatable}{proposition}{propwrealdense}\label{prop:main-w-real-dense}
For $p\in[1,\infty)$, $\Gamma_\mathcal{W}^p$ is a dense subset of $\Gamma^p$.
\end{restatable}

\begin{proof}
Fix any $\ell^p$-Gromov-Wasserstein geodesic $\gamma:[0,1]\rightarrow\ws$. Let $\rho\coloneqq\dgws{p}\left(\gamma\left(0\right),\gamma\left(1\right)\right)$ and assume that $\rho>0$. For any $\eps>0$, let $0=t_0<t_1<\ldots<t_n=1$ be a sequence such that $t_{i+1}-t_i<\frac{\eps}{2\rho}$ for $i=0,\ldots,n-1$. Then, by \Cref{coro:finite-seq-w-dgw}, there exist $X\in\ms$ and isometric embeddings $\varphi_i:X_{t_i}\hookrightarrow X$ such that $\dW{p}^X\left((\varphi_i)_\#\mu_{t_i},\lc\varphi_j\rc_\#\mu_{t_j}\right)=|t_i-t_j|\rho.$ Let $Z\coloneqq \mathcal{W}_1\left(X\right)$ and for each $i=0,\ldots,n$ still denote by $\varphi_i$ the composition of $\varphi_i:X_{t_i}\hookrightarrow X$ and the canonical embedding $X\hookrightarrow\mathcal{W}_1(X)=Z$. Then, we still have $\dW{p}^Z\left((\varphi_i)_\#\mu_{t_i},\lc\varphi_j\rc_\#\mu_{t_j}\right)=|t_i-t_j|\rho.$ Since $Z$ is compact and geodesic (cf. \Cref{thm:compact-w} and \Cref{thm:W-geodesic}), $\mathcal{W}_{p}(Z)$ is geodesic (cf. \Cref{thm:wp-geo}). Hence, for each $i=0,\ldots,n-1$ there {exists an} $\ell^p$-Wasserstein geodesic $\gamma_i:[0,1]\rightarrow \mathcal{W}_p\left(Z\right)$ such that $\gamma_i\left(0\right)=(\varphi_i)_\#\mu_{t_i}$ and $\gamma_i\left(1\right)=(\varphi_{i+1})_\#\mu_{t_{i+1}}$. Then, $\gamma_i(1)=\gamma_{i+1}(0)$ for $i=0,\ldots,n-1$ and
\begin{align*}
    \dW p^Z(\gamma_0(0),\gamma_{n-1}(1))&=\dW p^Z\left((\varphi_0)_\#\mu_{t_0},(\varphi_n)_\#\mu_{t_n}\right)\\
    &=\sum_{i=0}^{n-1}\dW p^Z((\varphi_{i})_\#\mu_{t_{i}},(\varphi_{i+1})_\#\mu_{t_{i+1}})\\
    &=\sum_{i=0}^{n-1}\dW p^Z(\gamma_i(0),\gamma_{i}(1))).
\end{align*}

Therefore, we can concatenate all the $\gamma_i$s via \Cref{prop:geo-concatenate} to obtain a new $\ell^p$-Wasserstein geodesic $\tilde{\gamma}:[0,1]\rightarrow \mathcal{W}_p\left(Z\right)$ such that $\tilde{\gamma}(t_i)=(\varphi_{i})_\#\mu_{t_{i}}$ for each $i=0,\ldots,n$. Then, by \Cref{lm:wgeo-to-dgwgeo}, $\tilde{\gamma}$ is a Gromov-Wasserstein geodesic and thus $\tilde{\gamma}\in\Gamma_\mathcal{W}^p$. Now, for any $t\in[0,1]$, suppose $t\in[t_i,t_{i+1}]$ for some $i\in\{0,\ldots,n-1\}$. Then, 
\begin{align*}
    \dgws{p}\left(\gamma\left(t\right),\tilde{\gamma}\left(t\right)\right)&\leq\dgws{p}\left(\gamma\left(t\right),{\gamma}\left(t_i\right)\right)+\dgws{p}\left(\gamma\left(t_i\right),\tilde{\gamma}\left(t_i\right)\right)+\dgws{p}\left(\tilde{\gamma}\left(t_i\right),\tilde{\gamma}\left(t\right)\right) \\
    &=|t-t_i|\,\rho+0+|t-t_i|\,\rho\leq 2\cdot \frac{\eps}{2\rho}\cdot\rho =\eps.
\end{align*}
So, $d_{\infty,p}(\gamma,\tilde{\gamma})\leq\eps$. Therefore, we conclude that $\Gamma_\mathcal{W}^p$ is dense in $\Gamma^p$.
\end{proof}

\paragraph{Hausdorff-boundedness.} Now, we introduce the Hausdorff-boundedness condition mentioned in the introduction with the purpose of identifying a certain family of Wasserstein-realizable geodesics. 

{First recall from \Cref{sec:real function} that for any function $f:I\rightarrow J$ where $I$ and $J$ are intervals in $\overline{\mathbb
R}\coloneqq[0,\infty]$ containing 0, we say that $f$ is \emph{proper} if both $f(0)=0$ and $f$ is continuous at $0$ (cf. \Cref{def:proper-function}).}

\begin{definition}[Hausdorff-bounded families]\label{def:hausdorff-bdd}
Given $p\in[1,\infty)$ and a family $\mathcal{F}$ of metric measure spaces, we say that $\mathcal{F}$ is ($\ell^p$-)\emph{Hausdorff-bounded}, if there exists an \emph{increasing and proper} function $f:[0,\infty)\rightarrow[0,\infty)$ such that for any $\mathcal{X}=(X,d_X,\mu_X),\mathcal{Y}=(Y,d_Y,\mu_Y)\in\mathcal{F}$ and isometric embeddings $\varphi_X:X\hookrightarrow Z$ and $\varphi_Y:Y\hookrightarrow Z$,
$$\dH^Z(X,Y)\leq f\left(\dW{p}^Z((\varphi_X)_\#\mu_X,(\varphi_Y)_\#\mu_Y)\right). $$
After specifying such an $f$, we say that $\mathcal{F}$ is \emph{$f$-Hausdorff-bounded}.
\end{definition}
{The remark below provides a first glimpse into the motivation for the definition of Hausdorff-bounded families.}

\begin{remark}\label{rmk:f-h-bdd to gh bound}
Given an $f$-Hausdorff-bounded family $\mathcal{F}$, for any $\mathcal{X},\mathcal{Y}\in \mathcal{F}$, we have that  
$$\dgh(X,Y)\leq f\left(\dgws{p}(\mathcal{X},\mathcal{Y})\right).$$
\end{remark}

{The definition of Hausdorff-boundedness is not superfluous. For example, the whole class $\ws$ is not a Hausdorff-bounded family for any increasing and proper function $f$ (see also \Cref{ex:non-h-b-gw}).}

{The Hausdorff-boundedness property is not easy to  verify directly. We therefore  seek   conditions which imply it. To this end, we introduce the notion of  \emph{$h$-boundedness} for metric measure spaces which will turn out to imply Hausdorff-boundedness (cf. \Cref{prop:ex-f-bdd-family}).}

\noindent{\textbf{Note:} given a metric space $X$ and $\eps\geq 0$, we will henceforth use the symbol $B_\varepsilon^X(x)$ to denote the \emph{closed ball} $B^X_\eps(x)\coloneqq\{x'\in X:\,d_X(x,x')\leq \eps\}$ centered at $x$ with radius $\eps$. We abbreviate $B_\eps^X(x)$ to $B_\eps(x)$ whenever the underlying space $X$ is clear from the context.}

\begin{definition}[$h$-bounded metric measure spaces]\label{def:h-bdd-mms}
Let $h:[0,\infty)\rightarrow[0,1]$ be a \emph{strictly increasing and proper} function. For any given $\mathcal{X}=(X,d_X,\mu_X)\in\ws$, we say $\mathcal{X}$ is \emph{$h$-bounded}, if for any $x\in X$ and $\eps\geq 0$, $\mu_X\lc B^X_\eps(x)\rc\geq h(\eps)$.
\end{definition}

\begin{remark}[Influence of diameter on $h$-boundedness]
Since $\mu_X(B_0(x))\geq0=h(0)$ and $\mu_X(B_{D}(x))=1\geq h(D)$ for any $D\geq\diam(X)$ always hold, in the above definition one can restrict $\eps$ to the interval $(0,\diam(X)]$. 
\end{remark}

{Below we show that the doubling condition essentially implies $h$-boundedness.}

\begin{example}[Examples of $h$-bounded metric measure spaces]\label{ex:f-bdd-mms}
In this example, we present some common types of $h$-bounded metric measure spaces together with explicit constructions of $h$. Fix a metric measure space $\mathcal{X}=(X,d_X,\mu_X)$ with $\diam(X)\leq D$.
\begin{itemize}
    \item  We say $\mathcal{X}$ is \emph{$C$-doubling} for a constant $C>1$ if for any $x\in X$ and $\eps\geq 0$, we have $$\mu_X(B_{2\eps}(x))\leq C\cdot \mu_X(B_\eps(x)).$$ 
    Then, for any $x\in X$ and $0\leq \eps\leq D$,
    $$\mu_X(B_\eps(x))\geq C^{-1}\mu_X(B_{2\eps}(x))\geq\ldots\geq C^{-\log_2\left(\frac{D}{\eps}\right)-1}\mu_X\left(B_{2^{\log_2\left(\frac{D}{\eps}\right)+1}\eps}(x)\right)= C^{-\log_2\left(\frac{D}{\eps}\right)-1}.$$
    {The function $C^{-\log_2\left(\frac{D}{\cdot}\right)-1}:[0,D]\rightarrow[0,C^{-1}]$ is strictly increasing and proper. Since $C>1$, $C^{-\log_2\left(\frac{D}{\cdot}\right)-1}$ can be extended to a strictly increasing and proper function $h_\mathcal{X}:[0,\infty)\rightarrow[0,1]$. Then, $\mathcal{X}$ is $h_\mathcal{X}$-bounded.}
    
    \item Suppose $X$ is a finite set and let $\delta_\mathcal{X}\coloneqq\min\{\mu_X(x):\,x\in X\}>0$. Let $h_\mathcal{X}:[0,\infty)\rightarrow[0,1]$ be any strictly increasing and proper function such that $h_\mathcal{X}(\eps)\leq \delta_\mathcal{X}$ for all $\eps\in[0,D]$. Then, $\mathcal{X}$ is $h_\mathcal{X}$-bounded.
\end{itemize}
\end{example}

Given a compact metric measure space $\mathcal{X}$, if we define $h_\mathcal{X}^\mathrm{inf}:[0,\infty)\rightarrow[0,1]$ by $\eps\mapsto \inf_{x\in X}\mu_X(B_\eps(x))$ for each $\eps\in[0,\infty)$, then obviously we have for any $x\in X$ and $\eps\geq 0$ that $\mu_X(B_\eps(x))\geq h_\mathcal{X}^\mathrm{inf}(\eps)$. We of course cannot conclude directly that $\mathcal{X}$ is $h_\mathcal{X}^\mathrm{inf}$-bounded since $h_\mathcal{X}^\mathrm{inf}$ is not necessarily strictly increasing and proper. However, it turns out that a slightly modification of the construction of $h_\mathcal{X}^\mathrm{inf}$ gives rise to the following result:

\begin{lemma}\label{lm:compact-surjective}
For any $\mathcal{X}\in\ws$, there exists a strictly increasing and proper function $h_\mathcal{X}:[0,\infty)\rightarrow[0,1]$ such that $\mathcal{X}$ is $h_\mathcal{X}$-bounded.
\end{lemma}
\begin{proof}
Without loss of generality, we assume that $\diam(X)=1$. For each $n\in\mathbb{N}$, let $X_n$ be a finite $\frac{1}{2n}$-net of $X$ such that $X_1\subseteq X_2\subseteq\ldots.$ Let $\eta_n\coloneqq\inf_{x_n\in X_n}\mu_X\left(B_\frac{1}{2n}(x_n)\right)$. Since $X_n$ is finite, the infimum is obtained and $\eta_n>0$ for $n= 1,\ldots.$ Obviously, $1\geq\eta_1\geq \eta_2\geq \ldots>0$. Then, we choose any strictly decreasing positive sequence $1\geq\zeta_1>\zeta_2>\ldots$ such that $\lim_{n\rightarrow\infty}\zeta_n=0$ and $\zeta_n\leq \eta_n$. Such a sequence obviously exists.

Define a function $f:\left\{\frac{1}{n}:n=1,2\ldots\right\}\rightarrow\mathbb{R}$ by mapping $\frac{1}{n}$ to $\zeta_{n+1}$ for $n=1,\ldots$. We extend $f$ to a new function $g:[0,1]\rightarrow[0,\zeta_2]$ by linearly interpolating $f$ inside the intervals $\left[\frac{1}{n+1},\frac{1}{n}\right]$ for $n\in\mathbb{N}$ and by letting $g(0)\coloneqq 0$. Now, let $\hat g:[1,\infty)\rightarrow[\zeta_2,1]$ be any strictly increasing function. Then, we extend $g$ to $h_\mathcal{X}:[0,\infty)\rightarrow[0,1]$ as follows:
$$h_\mathcal{X}(\eps) = \begin{cases} g(\eps),&\eps\in[0,1]\\
\hat g(\eps),&\eps\in(1,\infty)
\end{cases}.$$
Then, it is easy to check that $h_\mathcal{X}$ is strictly increasing and proper. Moreover,
$$h_\mathcal{X}\left(\eps\right)\leq \zeta_{n+1}\quad\forall n\in\mathbb{N},\forall \eps\in\left(\frac{1}{n+1},\frac{1}{n}\right].$$

Now, for any $x\in X$, there exists $x_n\in X_n$ such that $d_X(x,x_n)\leq \frac{1}{2n}$. Then, $B_\frac{1}{2n}(x_n)\subseteq B_\frac{1}{n}(x)$. Thus, $\mu_X\left(B_\frac{1}{n}(x)\right)\geq \mu_X\left(B_\frac{1}{2n}(x_n)\right)\geq \zeta_n$. Since for any $\eps\in(0,1]$, there exists $n\in\mathbb{N}$ such that $\frac{1}{n+1}< \eps\leq \frac{1}{n}$, we obtain
$$\mu_X\left(B_\eps(x)\right)\geq \mu_X\left(B_\frac{1}{n+1}(x)\right)\geq \zeta_{n+1}\geq h_\mathcal{X}\left(\eps\right).$$ 

Therefore, $\mathcal{X}$ is $h_\mathcal{X}$-bounded.
\end{proof}

We say that a family $\mathcal{F}\subseteq\ws$ of compact metric measure spaces is \emph{uniformly $h$-bounded} for some strictly increasing and proper $h:[0,\infty)\rightarrow[0,1]$ if every $\mathcal{X}\in\mathcal{F}$ is $h$-bounded. 

Since each $\mathcal{X}\in\ws$ is $h_\mathcal{X}$-bounded for some strictly increasing and proper $h_\mathcal{X}$, the one-element family $\mathcal{F}\coloneqq\{\mathcal{X}\}$ is obviously uniformly $h_\mathcal{X}$-bounded. Moreover, any finite family $\mathcal{F}$ is uniformly $h$-bounded where $h\coloneqq\min_{\mathcal{X}\in \mathcal{F}} h_\mathcal{X}$. However for an infinite family $\mathcal{F}$, it may not be true that one can find a uniform strictly increasing and proper $h$ for $\mathcal{F}$; see the example below:

\begin{example}[An example of non-uniformly $h$-bounded family]
For $n\in\mathbb{N}$, denote by $\Delta_n$ the $n$-point space with interpoint distance 1. Endow $\Delta_n$ with uniform probability measure (denoted by $\mu_n$) and denote the corresponding metric measure space by $\tilde{\Delta}_n=(\Delta_n,d_n,\mu_n)$. Let $\mathcal{F}=\{\tilde{\Delta}_n:\,n\in\mathbb N\}$. Then, there is no strictly increasing and proper function $h$ such that $\mathcal{F}$ is uniformly $h$-bounded. Indeed, we otherwise assume that $\mathcal{F}$ is uniformly $h$-bounded for some strictly increasing and proper function $h$. Then, for each $\tilde{\Delta}_n$, we pick $x_n\in \Delta_n$ and $\eps=\frac{1}{2}$. Since $\tilde{\Delta}_n$ is $h$-bounded, we have that
$$\mu_n\lc B_\frac{1}{2}(x_n)\rc=\mu_n(\{x_n\})=\frac{1}{n}\geq h\lc\frac{1}{2}\rc>0.$$
Then, this implies that $\frac{1}{n}>h\lc\frac{1}{2}\rc>0$ holds for all $n\in\mathbb N$, which is impossible.
\end{example}

The following result reveals a connection between uniform $h$-boundedness and Hausdorff-boundedness.

\begin{proposition}\label{prop:ex-f-bdd-family}
{Fix $p\in[1,\infty)$. Let $h:[0,\infty)\rightarrow[0,1]$ be a strictly increasing and proper function and let $\mathcal{F}$ be a family of $h$-bounded metric measure spaces. Then, $\mathcal{F}$ is $\tilde{h}^{-1}$-Hausdorff-bounded, where $\tilde{h}:[0,\infty)\rightarrow[0,\infty)$ is defined by $t\mapsto \frac{t}{2}\cdot h^\frac{1}{p}\lc\frac{t}{2}\rc$ for each $t\in[0,\infty)$.}
\end{proposition}

\begin{proof}
Let $\mathcal{X},\mathcal{Y}\in \mathcal{F}$. Suppose that $Z\in\ms$ and that there exist isometric embeddings $X\hookrightarrow Z$ and $Y\hookrightarrow Z$. For notational simplicity, we still denote by $\mu_X$ and $\mu_Y$ their respective pushforwards under the respective isometric embeddings. Let $\rho \coloneqq \dW p^Z(\mu_X,\mu_Y)$ and $\eta\coloneqq\dH^{Z}(X,Y)$. Assume that $\eta>0$ since the case $\eta=0$ is trivial. Then, by compactness of $X$ and $Y$, there exists $x_0\in X$ and $y_0\in Y$ such that $d_Z(x_0,y_0)=\dH^Z(X,Y)=\eta.$ Without loss of generality, we assume that 
$$d_Z(x_0,y_0)=d_Z(x_0,Y)\coloneqq\inf\{d_Z(x_0,y):\,y\in Y\}. $$
Then, consider the closed ball $B_{\frac{\eta}{2}}^X(x_0)$ in $X$, and let $x\in B_{\frac{\eta}{2}}^X(x_0)$. By the triangle inequality, for any $y\in Y$ we have that
$$d_Z(x,y)\geq d_Z(x_0,y) - d_Z(x_0, x)\geq d_Z(x_0,y_0)-d_X(x_0,x)\geq \eta-\frac{\eta}{2}=\frac{\eta}{2}. $$
Let $\mu\in\mathcal{C}^\mathrm{opt}_p(\mu_X,\mu_Y)$ be an optimal coupling, then we have that
\begin{align*}
    \delta &= \dW p^Z(\mu_X,\mu_Y)\\
    &=\lc\int_{X\times Y}\lc d_Z(x,y)\rc^pd\mu(x,y)\rc^\frac{1}{p}\\
    &\geq \lc\int_{B_{\frac{\eta}{2}}^X(x_0)\times Y}\lc d_Z(x,y)\rc^pd\mu(x,y)\rc^\frac{1}{p}\\
    &\geq \lc\int_{B_{\frac{\eta}{2}}^X(x_0)\times Y}\lc \frac{\eta}{2}\rc^pd\mu(x,y)\rc^\frac{1}{p}\\
    &=\frac{\eta}{2}\cdot\lc\mu\lc B_{\frac{\eta}{2}}^X(x_0)\times Y\rc\rc^\frac{1}{p}\\
    &=\frac{\eta}{2}\cdot\lc\mu_X\lc B_{\frac{\eta}{2}}^X(x_0)\rc\rc^\frac{1}{p}\\
    &\geq \frac{\eta}{2}\cdot h^\frac{1}{p}\lc\frac{\eta}{2}\rc\\
    &=\tilde{h}(\eta)
\end{align*}

Obviously, the function $\tilde{h}:[0,\infty)\rightarrow[0,\infty)$ defined by $t\mapsto \frac{t}{2}\cdot h^\frac{1}{p}\lc\frac{t}{2}\rc$ is strictly increasing and proper. By {item 4} of \Cref{prop:increasing-prop}, we have that
$$\dH^Z(X,Y)=\eta\leq \tilde{h}^{-1}(\delta)=\tilde{h}^{-1}\lc\dW p^Z(\mu_X,\mu_Y)\rc. $$

By {items 1 and 5} of \Cref{prop:increasing-prop} we have that $\tilde{h}^{-1}$ is increasing and proper. Moreover, it is easy to see that $\lim_{t\rightarrow\infty}\tilde{h}(t)=\infty$. This implies that $\tilde{h}^{-1}$ is a finite function $\tilde{h}^{-1}:[0,\infty)\rightarrow[0,\infty)$ (cf. {item 2} of \Cref{prop:increasing-prop}).  
Thus, $\mathcal{F}$ is $\tilde{h}^{-1}$-Hausdorff-bounded.
\end{proof}

For an $\ell^p$-Gromov-Wasserstein geodesic $\gamma:[0,1]\rightarrow\ws$, we say that $\gamma$ is Hausdorff-bounded if the family $\{\gamma(t)\}_{t\in[0,1]}$ is Hausdorff-bounded. The family of Hausdorff-bounded Gromov-Wasserstein geodesics is rich and we present two examples as follows.

\begin{proposition}[Geodesics consisting of finite spaces]\label{prop:f-bdd-finite}
Fix $p\in[1,\infty)$. Let $\gamma$ be an $\ell^p$-Gromov-Wasserstein geodesic. Assume that there exist a positive integer $N$, a constant $C>0$ and a constant $D>0$ such that for each $t\in[0,1]$, 
\begin{enumerate}
    \item the cardinality of $\gamma(t)$ is bounded above by $N$;
    \item for any $x\in \gamma(t)$, $\mu_t(\{x\})\geq C$;
    \item $\diam(\gamma(t))\leq D$.
\end{enumerate}
Then, $\gamma$ is Hausdorff-bounded.
\end{proposition}

\begin{proof}
As mentioned earlier in \Cref{ex:f-bdd-mms}, there exists a strictly increasing and proper function $h:[0,\infty)\rightarrow[0,1]$ such that  for each $t\in[0,1]$, $\gamma(t)$ is $h$-bounded. Then, by \Cref{prop:ex-f-bdd-family}, $\{\gamma(t)\}_{t\in[0,1]}$ is Hausdorff-bounded.
\end{proof}

For any two given compact metric measure spaces, Sturm constructed in \cite{sturm2020email} an $\ell^p$-Gromov-Wasserstein geodesic connecting them, which we call a \emph{straight-line $\dgws p$ geodesic}: 

\begin{theorem}[Straight-line $\dgws p$ geodesic \cite{sturm2020email}]\label{thm:straight-line-gw-geo}
Fix $p\in[1,\infty)$ and let $\mathcal{X},\mathcal{Y}\in\ws$. Let $Z\in\ms$ be such that there exist isometric embeddings $X\hookrightarrow Z$ and $Y\hookrightarrow Z$ such that
$$\dgws{p}\left(\mathcal{X},\mathcal{Y}\right)= \dW{p}^Z\left(\mu_X,\mu_Y\right), $$
whose existence is guaranteed by \Cref{lm:dgw_w-realizable}.
Let $\mu\in\mathcal{C}_p^\mathrm{opt}(\mu_X,\mu_Y)$ be any optimal coupling between $\mu_X$ and $\mu_Y$ with respect to $\dW{p}^Z$. Then, the curve $\gamma_\mu:[0,1]\rightarrow\ws$ defined as follows is an $\ell^p$-Gromov-Wasserstein geodesic:
$$\gamma_{\mu}(t)\coloneqq\begin{cases}\mathcal{X},& t=0\\
\left(S,d_{t},\mu\right), & t\in(0,1)\\
\mathcal{Y}, &t=1
\end{cases}$$
where $S\coloneqq\supp(\mu)\subseteq X\times Y $ and $d_{t}\left(\left(x,y\right),\left(x',y'\right)\right)\coloneqq\left(1-t\right)\,d_X\left(x,x'\right)+t\,d_Y\left(y,y'\right)$ for any $(x,y),(x',y')\in S$. We call $\gamma_\mu$ a \emph{straight-line $\dgws p$ geodesic}\footnote{Our definition is slightly different from the one  given in \cite{sturm2020email}: under our notation, for each $t\in[0,1]$, \cite{sturm2020email} defines a metric measure space $\gamma_\mu'(t)\coloneqq(Z\times Z,d_t',\mu)$ where $d_t'$ is defined by $d_t'((z_1,z_2),(z_1',z_2'))\coloneqq(1-t)d_Z(z_1,z_1')+t\,d_Z(z_2,z_2')$ for any $z_i,z_i'\in Z$ where $i=1,2$. Note that our geodesic $\gamma_\mu(t)$ in \Cref{thm:straight-line-gw-geo} can then be obtained by simply restricting $\gamma_\mu'(t)$ to the support of $\mu$.}.
\end{theorem}

Now, we use \Cref{prop:ex-f-bdd-family} to show the following:
\begin{proposition}\label{prop:strline-gw-h-bdd}
Given $p\in[1,\infty)$, any straight-line $\dgws p$ geodesic is Hausdorff-bounded.
\end{proposition}

\begin{proof}
Let $\mathcal{X},\mathcal{Y}\in\ws$ and let $Z\in\ms$ be the ambient space required in \Cref{thm:straight-line-gw-geo}. Let $\mu\in\mathcal{C}_p^\mathrm{opt}(\mu_X,\mu_Y)$ be an optimal coupling for $\dW p^Z$. By \Cref{lm:compact-surjective}, there exist $h_\mathcal{X}$ and $h_\mathcal{Y}$ such that $\mathcal{X}$ is $h_\mathcal{X}$-bounded and $\mathcal{Y}$ is $h_\mathcal{Y}$-bounded. Now, consider $\mathcal{S}\coloneqq(S,d_S\coloneqq\max(d_X,d_Y),\mu)$, where $S\coloneqq\supp(\mu)$. $\mathcal{S}$ is a compact metric measure space and $\mu$ is fully supported. By \Cref{lm:compact-surjective} again, there exists $h_\mathcal{S}$ such that $\mathcal{S}$ is $h_\mathcal{S}$-bounded.

Now, pick any $z_0=(x_0,y_0)\in S=\supp(\mu)$. Then, for any $t\in(0,1)$ and for any $\eps>0$, we have that
\begin{align*}
    B_\eps^{\gamma_\mu(t)}(z_0)&\coloneqq\{(x,y)\in S:\,d_t((x,y),(x_0,y_0))\leq \eps\}\\
    &\supseteq \lc B_\eps^X(x_0)\times B_\eps^Y(y_0)\rc\cap S\\
    &=B_\eps^{(S,d_S)}(z_0)
\end{align*}
Therefore,  
$$\mu\lc B_\eps^{\gamma_\mu(t)}(z_0)\rc\geq \mu\lc B_\eps^{(S,d_S)}(z_0)\rc \geq h_\mathcal{S}(\eps).$$
Let $h\coloneqq\min(h_\mathcal{X},h_\mathcal{Y},h_\mathcal{S})$, then we have that $\gamma_\mu(t)$ is $h$-bounded for any $t\in[0,1]$. Then, by \Cref{prop:ex-f-bdd-family}, we have that $\gamma_\mu$ is Hausdorff-bounded.
\end{proof}

\begin{remark}[Deviant and branching GW geodesics]\label{rmk:deviant and branching GW}
Analogously to the case of Gromov-Hausdorff geodesics (cf. \cite[Section 1.1]{chowdhury2018explicit}), for each $p\in[1,\infty)$ there exist \emph{deviant} geodesics in $\Gamma_\mathcal{W}^p$: that is, geodesics which are not straight-line $\dgws p$ geodesics. Furthermore, there also exist \emph{branching} geodesics. We provide such constructions in \Cref{app:d-b-geodesic} where we also show that these exotic geodesics are Hausdorff-bounded. In analogy with the case of $(\ms,\dgh)$, 
\begin{enumerate}
    \item the existence of branching geodesics shows that $\left(\ws,\dgws p\right)$ is not an Alexandrov space with curvature bounded below \cite[Chapter 10]{burago2001course};
    \item the existence of deviant (non-unique) geodesics shows that $\left(\ws,\dgws p\right)$ is also not a $\mathrm{CAT}(\kappa)$ space with curvature bounded above by some $\kappa\in\mathbb R$ \cite[Chapter 2.1]{bridson2013metric}.
\end{enumerate}
\end{remark}

\paragraph{Proof of \Cref{thm:bdd-geo-w-real}.} We now deal with the proof of \Cref{thm:bdd-geo-w-real}. {The proof consists of two parts: first we carefully approximate a given Hausdorff bounded geodesic $\gamma$ by a convergent sequence of Wasserstein realizable geodesics; then, this carefully chosen convergent sequence allows us to utilize the Hausdorff-boundedness condition on $\gamma$ so that we can follow a strategy analogous to the one used for proving \Cref{thm:main-h-realizable} to now show that this convergent sequence has a Wasserstein realizable limit, which must agree with $\gamma$.}

\thmbddgeo*

\begin{proof}
Let $\gamma:[0,1]\rightarrow\ms^w$ be a Hausdorff-bounded $\ell^p$-Gromov-Wasserstein geodesic. For each $t\in[0,1]$, we write $\gamma(t)=(X_t,d_{t},\mu_t)$. Assume that $\rho\coloneqq\dgws{1}\left(\gamma\left(0\right),\gamma\left(1\right)\right)>0$. Let $T^n=\{t_i^n\coloneqq i\cdot 2^{-n}:i=0,\ldots,2^n\}$ for $n=0,1,\ldots$. Then, by \Cref{coro:finite-seq-w-dgw}, there exist $Z^n\in\ms$ and isometric embeddings $\varphi_i^n:X_{t_i^n}\hookrightarrow Z^n$ for $i=1,\ldots,2^n$ such that $\dW{p}^{Z^n}\left(\lc\varphi_i^n\rc_\#\mu_{t_i^n},\lc\varphi_j^n\rc_\#\mu_{t_j^n}\right)=|t_i^n-t_j^n|\rho.$

For $p=1$, by \Cref{lm:W-geodesic}, we know that $\gamma_i^n:[0,1]\rightarrow \mathcal{W}_1(Z^n)$ defined by 
\[t\mapsto (1-t)\lc\varphi_i^n\rc_\#\mu_{t_i^n}+t\lc\varphi_{i+1}^n\rc_\#\mu_{t_{i+1}^n}\]
is an $\ell^1$-Wasserstein geodesic for each $n=0,\ldots$ and each $i=0,\ldots,2^n-1$. Then,
\begin{align*}
    \dW 1^{Z^n}\left(\gamma_0^n(0),\gamma_{2^{n}-1}^n(1)\right)&=\dW 1^{Z^n}\left((\varphi_0^n)_\#\mu_{t_0^n},\left(\varphi_{2^{n}-1}^n\right)_\#\mu_{t_{2^n}^n}\right)\\
    &=\sum_{i=0}^{2^{n}-1}\dW 1^{Z^n}\left((\varphi_{i}^n)_\#\mu_{t_{i}^n},(\varphi_{i+1}^n)_\#\mu_{t_{i+1}^n}\right)\\
    &=\sum_{i=0}^{2^n-1}\dW 1^{Z^n}(\gamma_i^n(0),\gamma_{i}^n(1)).
\end{align*}

Therefore, we can concatenate all the $\gamma_i^n$s for $i=0,\ldots,2^n-1$ via \Cref{prop:geo-concatenate} to obtain an $\ell^1$-Wasserstein geodesic ${\gamma}^n:[0,1]\rightarrow\mathcal{W}_1(Z^n)$ such that $\gamma^n(t_i^n)=(\varphi_i^n)_\#\mu_{t_i^n}$ for $i=0,\ldots,2^n$. Since $\dW 1^{Z^n}(\gamma^n(0),\gamma^n(1))=\rho=\dgws 1(\gamma(0),\gamma(1))$, by \Cref{lm:wgeo-to-dgwgeo} we have that $\gamma^n$ is actually an $\ell^1$-Gromov-Wasserstein geodesic. We follow the notation from \Cref{lm:wgeo-to-dgwgeo} and denote by $\tilde{\gamma}^n:[0,1]\rightarrow\ws$ the $\ell^1$-Gromov-Wasserstein geodesic corresponding to $\gamma^n$. Then, it is easy to check that $d_{\infty,1}\lc\gamma,\tilde{\gamma}^n\rc\leq 2\cdot\frac{1}{2^n}\cdot\rho=2^{1-n}\rho$ via an argument similar to the one used in the proof of \Cref{prop:main-w-real-dense}. For each $n=0,\ldots,i=0,\ldots,2^n$ and $t\in[0,1]$, we have $\supp(\gamma_i^n(t))\subseteq \varphi_i^n\lc X_{t_i^n}\rc\cup \varphi_{i+1}^n\lc X_{t_{i+1}^n}\rc$. Therefore, $\cup_{t\in[0,1]}\supp({\gamma}^n(t))=\cup_{i=0}^{2^n}\varphi_i^n\lc X_{t_i^n}\rc=:Y^n, $ and thus $\gamma^n$ is actually a geodesic in $\mathcal{W}_1(Y^n)\subseteq \mathcal{W}_1(Z^n)$.

For $p>1$, since $W^n\coloneqq\mathcal{W}_1(Y^n)$ is geodesic (cf. \Cref{thm:W-geodesic}), we have that $\mathcal{W}_p(W^n)$ is geodesic (cf. \Cref{thm:wp-geo}). We still denote $\varphi_i^n$ the composition of $\varphi_i^n:X_{t_i^n}\rightarrow Z^n(\text{or }Y^n)$ and the canonical embedding $Y^n\hookrightarrow W^n$. Then, there exists a geodesic $\gamma_i^n$ connecting $(\varphi_i^n)_\#\mu_{t_i^n}$ and $(\varphi_{i+1}^n)_\#\mu_{t_{i+1}^n}$. Similarly to the case when $p=1$, we concatenate $\gamma_i^n$s to obtain an $\ell^p$-Wasserstein geodesic ${\gamma}^n:[0,1]\rightarrow\mathcal{W}_p(W^n)$ and thus an $\ell^p$-Gromov-Wasserstein geodesic $\tilde{\gamma}^n$ such that $d_{\infty,p}(\gamma,\tilde{\gamma}^n)\leq 2\cdot\frac{1}{2^n}\cdot\rho=2^{1-n}\rho$. Therefore, for each $t\in[0,1]$, $\gamma(t)$ is the $\ell^p$-Gromov-Wasserstein limit of $\{\tilde{\gamma}^n(t)\}_{n=0}^\infty$.

\begin{claim}\label{clm:key-in-main-w}
There exists a $\dgh$-convergent subsequence of $\{Y^n\}_{n=0}^\infty$.
\end{claim}

Assume the claim for now and consider first the case when $p=1$. Suppose $X\in\ms$ is such that $\lim_{n\rightarrow\infty}\dgh(Y^n,X)=0$ (after possibly passing to a subsequence), then by \Cref{lm:ghconv=hconv}, there exist a Polish metric space $Z$ and isometric embeddings $\varphi:X\hookrightarrow Z$ and $\varphi^n:Y^n\hookrightarrow Z$ for $n=0,\ldots$ such that $\lim_{n\rightarrow\infty}\dH^Z(\varphi^n(Y^n),\varphi(X))=0$. By \Cref{thm:W-equal}, 
\[\lim_{i\rightarrow\infty}\dH^{\mathcal{W}_1(Z)}\big((\varphi^n)_\#(\mathcal{W}_1(Y^n)),\varphi_\#(\mathcal{W}_1(X))\big)=0.\]
For each $n=0,\ldots$, we have that ${\gamma}^n:[0,1]\rightarrow\mathcal{W}_1(Y^n)$ is $\rho$-Lipschitz. Then, since $(\varphi^n)_\#$ is an isometric embedding, $(\varphi^n)_\#\circ\gamma^n:[0,1]\rightarrow (\varphi^n)_\#(\mathcal{W}_1(Y^n))$ is also $\rho$-Lipschitz. Moreover, by \Cref{thm:complete-w}, $\mathcal{W}_1(Z)$ is complete. Then, by the generalized Arzel\`a-Ascoli theorem (\Cref{thm:general-AA}), the  sequence
\[\big\{(\varphi^n)_\#\circ\gamma^n:[0,1]\rightarrow (\varphi^n)_\#(\mathcal{W}_1(Y^n))\big\}_{n=0}^\infty\]
uniformly converges to a $\rho$-Lipschitz curve $\hat{\gamma}:[0,1]\rightarrow (\varphi)_\#(\mathcal{W}_1(X))$ in the space $\mathcal{W}_1(Z)$.

By uniform convergence, for each $t\in[0,1]$ we have that
$$\lim_{n\rightarrow\infty}\dW{1}^{Z}(\hat{\gamma}(t),(\varphi^n)_\#\circ\gamma^n(t))=0.$$
Let $\tilde{X}_t\coloneqq\supp(\hat{\gamma}(t))$ and let $\tilde{\gamma}(t)\coloneqq\lc\tilde{X}_t,d_Z|_{\tilde{X}_t\times\tilde{X}_t},\hat{\gamma}(t)\rc$. Then, we have that
$$0\leq \lim_{n\rightarrow\infty}\dgws 1\lc\tilde{\gamma}(t),\tilde{\gamma}^n(t)\rc\leq\lim_{n\rightarrow\infty}\dW{1}^{Z}(\hat{\gamma}(t),(\varphi^n)_\#\circ\gamma^n(t))=0.  $$
We know that $\gamma(t)$ is the Gromov-Wasserstein limit of $\{\tilde{\gamma}^n(t)\}_{n=0}^\infty$, thus $\tilde{\gamma}(t)\cong_w\gamma(t)$ for $t\in[0,1]$. Then, since $\tilde{\gamma}$ is $X$-Wasserstein-realizable, we have that $\gamma\in\Gamma_\mathcal{W}^1$.

When $p>1$, by stability of $\mathcal{W}_1$ (cf. \Cref{thm:W-1-lip}),  $\{W^n=\mathcal{W}_1(Y^n)\}_{n=0}^\infty$ also has a convergent subsequence. Then, via an argument similar to the one used in the case of $p=1$, one concludes that $\gamma\in\Gamma_\mathcal{W}^p$.

Now, we finish by proving \Cref{clm:key-in-main-w}:
\begin{proof}[Proof of \Cref{clm:key-in-main-w}]
We assume that $\gamma $ is $f$-Hausdorff-bounded for an increasing and proper function $f:[0,\infty)\rightarrow[0,\infty)$, i.e., $f(0)=0$ and $f$ is continuous at $0$. We prove the claim by suitably applying Gromov's pre-compactness theorem (cf. \Cref{thm:pre-compact}). 

\begin{enumerate}
    \item For any $t\in(0,1]$, by \Cref{rmk:f-h-bdd to gh bound} we have that $\dgh(X_0,X_t)\leq f\left(\dgws{1}(\gamma(0),\gamma(t))\right)\leq f(\rho)$. Since $\dgh(X_0,X_t)\geq \frac{1}{2}|\diam(X_0)-\diam(X_t)|$ (cf. \cite[Theorem 3.4]{memoli2012some}), we have that 
    \[\diam(X_t)\leq 2f(\rho)+\diam(X_0).\]
    Now, fix $n\in\mathbb{N}$ and $x_0\in  \varphi_0^n\lc X_{t_0^n}\rc\subseteq Y^n$. Then, for any $t_i^n\in T^n$, there exist $x_0'\in  \varphi_0^n\lc X_{t_0^n}\rc$ and $x_i'\in \varphi_i^n\lc X_{t_i^n}\rc$ such that $d_{Y^n}(x_0',x_i')=d_{Y^n}\lc\varphi_0^n\lc X_{t_0^n}\rc,\varphi_i^n\lc X_{t_i^n}\rc\rc$ by compactness. Obviously, we have that
    $$d_{Y^n}(x_0',x_i')\leq \dW{1}^{Y^n}\lc(\varphi_0^n)_\#\mu_{t_0^n},(\varphi_i^n)_\#\mu_{t_i^n}\rc=\left|t_0^n-t_i^n\right|\rho\leq \rho.$$ 
    Now, for any point $x_i\in \varphi_i^n\lc X_{t_i^n}\rc$, we have that
    \begin{align*}
        d_{Y^n}(x_0,x_i)&\leq d_{Y^n}(x_0,x_0')+d_{Y^n}(x_0',x_i')+d_{Y^n}(x_i',x_i)\\
        &\leq \diam\lc\varphi_0^n\lc X_{t_0^n}\rc\rc+\rho+\diam\lc\varphi_i^n\lc X_{t_i^n}\rc\rc\\
        &=\diam\lc X_{t_0^n}\rc+\rho+\diam\lc X_{t_i^n}\rc\leq \rho+2f(\rho)+2\diam(X_0).
    \end{align*}
    The inequality holds for any $i=0,\ldots,2^n$. So $\diam(Y^n)\leq 2(\rho+2f(\rho)+2\diam(X_0))$ for all $n\in\mathbb{N}$.
    
    \item For any $\eps>0$, there exists $M>0$ such that $t_{j+1}^M-t_j^M<f^{-1}(\frac{\eps}{2})\cdot\rho^{-1}$. Here $f^{-1}(\frac{\eps}{2})>0$ follows from {item 5} of \Cref{prop:increasing-prop}. Let $S^M_j$ be an $\frac{\eps}{2}$-net of $\gamma\left(t_j^M\right)$ for all $j=0,\ldots,2^M$. For any $n>M$, $\{t_j^M\}_{j=0}^{2^M}$ is a subsequence of $\{t_i^n\}_{i=0}^{2^n}$. Indeed, $t_j^M=t^n_{j\cdot 2^{n-M}}$ for $j=0,\ldots,2^M$. For any $t_i^n$, there exists $j$ such that $t_j^M\leq t_i^n<t^M_{j+1}$. We know by construction of $Y^n$ that
    $$\dW{1}^{Y^n}\left(\lc\varphi_{j\cdot 2^{n-M}}^n\rc_\#\mu_{t_j^M},\lc\varphi_{i}^n\rc_\#\mu_{t_i^n}\right)= |t_i^n-t_j^M|\rho\leq |t^M_{j+1}-t_j^M|\rho< f^{-1}\left(\frac{\eps}{2}\right).$$
    Therefore,
    $$\dH^{Y^n}\left(\varphi_{j\cdot 2^{n-M}}^n\lc X_{t_j^M}\rc,\varphi_i^n\lc X_{t_i^n}\rc\right)\leq f\left(\dW{1}^{Y^n}\left(\lc\varphi_{j\cdot 2^{n-M}}^n\rc_\#\mu_{t_j^M},\lc\varphi_{i}^n\rc_\#\mu_{t_i^n}\right)\right)\leq\frac{\eps}{2},$$ 
    where we used {item 3} of \Cref{prop:increasing-prop} in the second inequality.

    Therefore, for each point $x_i^n\in \varphi_i^n\lc X_{t_i^n}\rc$, there exists a point $x_j^M\in \varphi_{j\cdot 2^{n-M}}^n\lc X_{t_j^M}\rc$ such that $d_{Y^n}\left(x_i^n,x_j^M\right)\leq\frac{\eps}{2}$. Note that $\varphi_{j\cdot 2^{n-M}}^n\lc S_j^M\rc$ is an $\frac{\eps}{2}$-net of $\varphi_{j\cdot 2^{n-M}}^n\lc X_{t_j^M}\rc$.
    Then, there exists a point $\tilde{x}_j^M\in \varphi_{j\cdot 2^{n-M}}^n\lc S_j^M\rc$ such that $d_{Y^n}\left(x^M_j,\tilde{x}_j^M\right)\leq \frac{\eps}{2}$. So, $d_{Y^n}\left(x^n_i,\tilde{x}_j^M\right)\leq \eps$. Therefore, we have that $Y^n\subseteq\cup_{j=0}^{2^M} \left(\varphi_{j\cdot 2^{n-M}}^n\lc S_j^M\rc\right)^\eps$. Let
    $$Q\left(\eps\right)\coloneqq\max\left(\max\{\mathrm{cov}_\eps\left(Y^n\right):\,n=1,\ldots,M-1\},\sum_{j=0}^{2^M}\left|S_j^M\right|\right),$$
    then we have that $\mathrm{cov}_\eps\left(Y^n\right)\leq Q\left(\eps\right)$ for all $n=0,\ldots$.
\end{enumerate}
Therefore, $\{Y^n\}_{n=0}^\infty\subseteq \mathcal{K}\big(Q,2(\rho+2f(\rho)+2\diam(X_0))\big)$ (cf. \Cref{def:CND}). By Gromov's pre-compactness theorem (cf. \Cref{thm:pre-compact}), $\{Y^n\}_{n=0}^\infty$ has a convergent subsequence.
\end{proof}
\end{proof}

There exist examples of non-Hausdorff-bounded Gromov-Wasserstein geodesics.

\begin{example}[An example of non-Hausdorff-bounded Gromov-Wasserstein geodesic]\label{ex:non-h-b-gw}
Let $\mathcal{X}=(X,d_X,\mu_X)$ be the one point metric measure space and $\mathcal{Y}=(Y,d_Y,\mu_Y)$ be a two-point metric measure space with unit distance and uniform probability measure. Assume that $X=\{x\}$ and $Y=\{y_1,y_2\}$. Then, under the map $\varphi_X:X\rightarrow Y$ defined by $x\mapsto y_1$ and the identity map $\varphi_Y\coloneqq \mathrm{Id}:Y\rightarrow Y$, it is easy to check that 
$$\dgws 1(\X,\Y)=\dW 1^Y\lc(\varphi_X)_\#\mu_X,(\varphi_Y)_\#\mu_Y\rc. $$

Define $\gamma:[0,1]\rightarrow\mathcal{W}_1(Y)$ by mapping each $t\in[0,1]$ to $(1-t)\,(\varphi_X)_\#\mu_X+t\,(\varphi_Y)_\#\mu_Y$. It is easy to describe $\gamma(t)$ explicitly as follows: for each $t\in(0,1]$, $\gamma(t)(\{y_1\})=1-\frac{t}{2}$ and $\gamma(t)(\{y_2\})=\frac{t}{2}$. By \Cref{lm:W-geodesic}, $\gamma$ is an $\ell^1$-Wasserstein geodesic. Thus $\gamma$ corresponds to an $\ell^1$-Wasserstein-realizable Gromov-Wasserstein geodesic $\tilde{\gamma}$ connecting $\X$ and $\Y$ (cf. \Cref{lm:wgeo-to-dgwgeo}). 

Note that $\dgh(X,\tilde{\gamma}(t))=\dgh(X,Y)=\frac{1}{2}$ for all $t\in(0,1]$. However, $\lim_{t\rightarrow0}\dgws 1(\X,\tilde{\gamma}(t))=0$ which precludes the existence of any increasing and proper function $f:[0,\infty)\rightarrow[0,\infty)$ such that 
$$\frac{1}{2}=\dgh(X,\tilde{\gamma}(t))\leq f\left(\dgws 1(\X,\tilde{\gamma}(t))\right). $$
Therefore, the geodesic $\tilde{\gamma}$ is not Hausdorff-bounded. 
\end{example}

Note that the example constructed above is highly dependent on the special linear interpolation geodesic corresponding to $\dW 1$ (cf. \Cref{lm:W-geodesic}). We are not aware of any example of an $\ell^p$-Gromov-Wasserstein geodesic for $p>1$ which is not Hausdorff-bounded.

\begin{conjecture}\label{conj:gw}
For every $p\in(1,\infty)$, any $\ell^p$-Gromov-Wasserstein geodesic is Hausdorff-bounded. Also, for every $p\in[1,\infty)$, any $\ell^p$-Gromov-Wasserstein geodesic is Wasserstein-realizable.
\end{conjecture}

\section{Dynamic geodesics}\label{sec:dyn-geo}
In this section, we first carefully study the properties of the Hausdorff displacement interpolation and prove \Cref{thm:main-dyn-hausdorff}, then extend our results to study dynamic Gromov-Hausdorff geodesics and prove \Cref{thm:main-dyn-gh}.

\subsection{Displacement interpolation}

\paragraph{Geodesics in $\mathcal{W}_p\left(X\right)$ for $p\in\left(1,\infty\right)$.} Given a metric space $X$, it is known that if $X$ is a geodesic space, then $\mathcal{W}_p\left(X\right)$ is a geodesic space (cf. \Cref{thm:wp-geo}). Geodesics in $\mathcal{W}_p\left(X\right)$ are also called \emph{displacement interpolation}, due to a refined characterization of geodesics in $\mathcal{W}_p\left(X\right)$ which we now explain. 

Let $C\left([0,1],X\right)$ denote the set of all continuous curves $\gamma:[0,1]\rightarrow X$ with the uniform metric $d_\infty^X\left(\gamma_1,\gamma_2\right)\coloneqq\sup_{t\in[0,1]}d_X\left(\gamma_1\left(t\right),\gamma_2\left(t\right)\right)$. Let $\Gamma([0,1],X)$ denote the subset of $C([0,1],X)$ consisting of all geodesics in $X$.
For $t\in[0,1]$, let $e_t:C\left([0,1],X\right)\rightarrow X$ be the \emph{evaluation map} taking $\gamma\in C\left([0,1],X\right)$ to $\gamma\left(t\right)\in X$.

\begin{definition}[Dynamic (optimal) coupling]\label{def:dyn-coup}
Let $X$ be a metric space and let $\alpha,\beta\in\mathcal{P}(X)$. We call $\Pi\in\mathcal{P}\left(C\left([0,1],X\right)\right)$ (the space of probability measures on $C([0,1],X)$) a \emph{dynamic coupling} between $\alpha$ and $\beta$, if $({e_0})_\#\Pi=\alpha$ and $({e_1})_\#\Pi=\beta$, where $(e_t)_\#$ represents the pushforward map under $e_t$. We call $\Pi$ a $\ell^p$-\emph{dynamic optimal coupling} between $\alpha$ and $\beta$, if $\supp\left(\Pi\right)\subseteq\Gamma\left([0,1],X\right)$ and $\left(e_0,e_1\right)_\#\Pi\in\mathcal{C}^\mathrm{opt}_p(\alpha,\beta)$ is an optimal transference plan with respect to $d_{\mathcal{W},p}$ for a given $p\in\left(1,\infty\right)$.
\end{definition}

Based on the notion of dynamic optimal coupling, a probability measure on the space of geodesics, there is the following characterization of geodesics in $\ell^p$-Wasserstein hyperspaces. For background materials and proofs, interested readers are referred to \cite[Section 3.2]{ambrosio2013user} or \cite[Section 7]{villani2008optimal}.

\begin{theorem}[Displacement interpolation]\label{thm:dis-int}
Let $X$ be a Polish geodesic space and let $p\in\left(1,\infty\right)$. Given $\alpha,\beta\in\mathcal{P}_p\left(X\right)$, and a continuous curve $\gamma:[0,1]\rightarrow\mathcal{W}_p\left(X\right)$, the following properties are equivalent:
\begin{enumerate}
    \item $\gamma$ is a geodesic in $\mathcal{W}_p\left(X\right)$;
    \item there exists an $\ell^p$-dynamic optimal coupling $\Pi$ between $\alpha$ and $\beta$ such that $(e_t)_\#\Pi=\gamma(t)$ for each $t\in[0,1]$.
\end{enumerate}
\end{theorem}

The notion of dynamic optimal coupling is the main inspiration for our definitions of Hausdorff displacement interpolation (cf. \Cref{def:haus-dis-int}) and dynamic optimal correspondences (cf. \Cref{def:dyn-cor}), whereas \Cref{thm:dis-int} serves as the motivation for one of our main results \Cref{thm:main-dyn-hausdorff}.

\subsubsection{Hausdorff displacement interpolation}\label{sec:hausdorff displacement interpolation}
Given a compact metric space $X$, the following is a direct consequence of definition of the uniform metric $d_\infty^X$ on $C([0,1],X)$.

\begin{lemma}
For any $t\in[0,1]$, the evaluation $e_t:C\left([0,1],X\right)\rightarrow X$ taking $\gamma$ to $\gamma\left(t\right)$ is a continuous map.
\end{lemma}

For any closed subsets $A,B\subseteq X$ with $\rho\coloneqq \dH^X\left(A,B\right)>0$, recall the definition of $\mathfrak{L}\left(A,B\right)$:
\begin{align*}
    \mathfrak{L}\left(A,B\right)&\coloneqq\left\{\gamma:[0,1]\rightarrow X:\,\gamma(0)\in A,\,\gamma(1)\in B\text{ and }\forall s,t\in[0,1],\,d_X\left(\gamma(s),\gamma(t)\right)\leq|s-t|\,\rho\right\}\\
    &=\left\{\gamma\in C([0,1],X):\,\gamma(0)\in A,\,\gamma(1)\in B\text{ and }\gamma\text{ is }\rho\text{-Lipschitz}\right\}.
\end{align*}

We have the following basic property of $\mathfrak{L}\left(A,B\right)$.

\begin{proposition}
$\mathfrak{L}\left(A,B\right)$ is a compact subset of $C\left([0,1],X\right)$.
\end{proposition}

\begin{proof}
Let $\{\gamma_i:[0,1]\rightarrow X\}_{i=0}^\infty$ be a sequence in $\mathfrak{L}\left(A,B\right)$. By definition of $\mathfrak{L}\left(A,B\right)$, each $\gamma_i$ is $\rho$-Lipschitz. Since $X$ is also compact, by Arzel\`a-Ascoli theorem (\Cref{thm:AA}), there exists a subsequence of the sequence $\{\gamma_i\}_{i=0}^\infty$ uniformly converging to a $\rho$-Lipschitz curve $\gamma:[0,1]\rightarrow X$. Since $A$ is compact, $\gamma(0)=\lim_{i\rightarrow\infty}\gamma_i(0)\in A$. Similarly, $\gamma(1)\in B$ and as a result, $\gamma\in\mathfrak{L}\left(A,B\right)$. This implies sequential compactness and thus compactness of $\mathfrak{L}\left(A,B\right)$.
\end{proof}

\begin{remark}\label{rmk:et-closed-map}
In particular, $e_t:\mathfrak{L}\left(A,B\right)\rightarrow X$ is a closed map, i.e., for any closed subset $\mathfrak{D}\subseteq\mathfrak{L}\left(A,B\right)$, the image $e_t(\mathfrak{D})$ is closed in $X$. Indeed, $\mathfrak{D}$ is then compact since $\mathfrak{L}\left(A,B\right)$ is compact. Therefore, $e_t(\mathfrak{D})$ is compact and thus closed since $e_t$ is continuous.
\end{remark}

\begin{definition}[Hausdorff displacement interpolation]\label{def:haus-dis-int}
We call a closed subset $\mathfrak{D}\subseteq\mathfrak{L}\left(A,B\right)$ a \emph{Hausdorff displacement interpolation} between $A$ and $B$ if $e_0\left(\mathfrak{D}\right)=A$ and $e_1\left(\mathfrak{D}\right)=B$.
\end{definition}

\paragraph{Proof of \Cref{thm:main-dyn-hausdorff}.}Recall that one of our main results is the following:
\thmdynhaus*

Note that in item 2 of the theorem, that $\gamma(t)\in\mathcal{H}(X)$ follows from \Cref{rmk:et-closed-map}. The proof of the implication $1\Rightarrow 2$ is based on the following observations.

\begin{lemma}\label{lm:geo-char-hausdorff}
Let $X\in\ms$ and let $\gamma:[0,1]\rightarrow \mathcal{H}\left(X\right)$ be a Hausdorff geodesic. Let $\rho\coloneqq\dH^X\left(\gamma\left(0\right),\gamma\left(1\right)\right)$ and assume that $\rho>0$. Then, for any $t_*\in[0,1]$ and any $x_*\in\gamma\left(t_*\right)$, there exists a $\rho$-Lipschitz curve $\zeta:[0,1]\rightarrow X$ such that $\zeta\left(t\right)\in\gamma\left(t\right)$ for $t\in[0,1]$ and $\zeta\left(t_*\right)=x_{*}$.
\end{lemma}
The proof of \Cref{lm:geo-char-hausdorff} is technical and we postpone it to the end of this section. One immediate consequence of the lemma is the following result.

\begin{coro}\label{coro:1to2-thm-3}
Let $X\in\ms$ and let $\gamma:[0,1]\rightarrow \mathcal{H}\left(X\right)$ be a Hausdorff geodesic. Let $\rho\coloneqq\dH^X\left(\gamma\left(0\right),\gamma\left(1\right)\right)$ and assume that $\rho>0$. Let
$$\mathfrak{D}\coloneqq\left\{\zeta\in C([0,1],X):\,\forall t\in[0,1],\,\zeta(t)\in\gamma\left(t\right),\,\text{and }\zeta\text{ is }\rho\text{-Lipschitz}\right\}.$$ 
Then, $\mathfrak{D}$ is a nonempty closed subset of $\mathfrak{L}\left(\gamma\left(0\right),\gamma\left(1\right)\right)$ such that $e_t(\mathfrak{D})=\gamma(t)$ for each $t\in[0,1]$. 
\end{coro}

\begin{proof}
By \Cref{lm:geo-char-hausdorff}, $\mathfrak{D}$ is obviously nonempty and $e_t(\mathfrak{D})=\gamma(t)$ for each $t\in[0,1]$. Now, it remains to prove the  closedness of $\mathfrak{D}$. Let $\{\zeta_i\}_{i=1}^\infty$ be a sequence in $\mathfrak{D}$ converging to some $\zeta\in C([0,1],X)$ with respect to the metric $d_\infty^X$. Then, by \Cref{thm:AA}, $\zeta$ is $\rho$-Lipschitz. Moreover, for each $t\in[0,1]$, since $\zeta_i(t)\in\gamma(t)$ for all $i=1,\ldots$ and $\lim_{i\rightarrow \infty}\zeta_i(t)=\zeta(t)$,  by the  closedness of $\gamma(t)$ we have that $\zeta(t)\in\gamma(t)$. Therefore, $\zeta\in\mathfrak{D}$ and thus $\mathfrak{D}$ is closed.
\end{proof}

This corollary establishes the direction $1\Rightarrow 2$ of \Cref{thm:main-dyn-hausdorff}. We prove the opposite direction as follows:

\begin{proof}[{Proof of $2\Rightarrow 1$ in \Cref{thm:main-dyn-hausdorff}}]
Suppose that there exists a closed subset $\mathfrak{D}\subseteq\mathfrak{L}\left(\gamma\left(0\right),\gamma\left(1\right)\right) $ such that $\gamma\left(t\right)=e_t\left(\mathfrak{D}\right)$ for all $t\in[0,1]$. For any $t,t'\in[0,1]$, choose an arbitrary $x_t\in e_t\left(\mathfrak{D}\right)$. Let $\zeta\in\mathfrak{D}$ be such that $\zeta(t)=x_t$ for each $t\in[0,1]$. Let $x_{t'}\coloneqq \zeta({t'})$. Then, $d_X\left(x_t,x_{t'}\right)\leq |t-t'|\rho$. So $\gamma\left(t\right)\subseteq\lc\gamma\left(t'\right)\rc^{|t-t'|\rho}$. Similarly, $\gamma\left(t'\right)\subseteq\lc\gamma\left(t\right)\rc^{|t-t'|\rho}$ and thus 
$$\dH^X\left(\gamma\left(t\right),\gamma\left(t'\right)\right)\leq |t-t'|\cdot\rho= |t-t'|\cdot\dH^X\left(\gamma\left(0\right),\gamma\left(1\right)\right).$$
Hence, $\gamma$ is a Hausdorff geodesic.
\end{proof}

\paragraph{Interpretations and consequences of \Cref{thm:main-dyn-hausdorff}.}The following remark discusses a difference between the dynamic optimal coupling and the Hausdorff displacement interpolation.
\begin{remark}[A difference between dynamic optimal coupling and Hausdorff displacement interpolation]\label{rmk:couterex-haus-geo}
Given a compact metric space $X$ and $\alpha,\beta\in\mathcal{P}\left(X\right)$, suppose $A=\supp\left(\alpha\right)$ and $B=\supp\left(\beta\right)$. Note in fact that any dynamic optimal coupling $\Pi$ between $\alpha$ and $\beta$ is supported on $\Gamma\left(A,B\right)$, the set of geodesics $\gamma$ with $\gamma\left(0\right)\in A$ and $\gamma\left(1\right)\in B$ (cf. \cite[Corollary 7.22]{villani2008optimal}). It is tempting to ask whether $\mathfrak{L}\left(A,B\right)$ can be replaced by $\mathfrak{L}\left(A,B\right)\cap \Gamma\left(A,B\right)$ in \Cref{thm:main-dyn-hausdorff}. This is not necessarily true. Indeed, consider the following case where $X\subseteq\R^2$ is a large enough compact disk around the origin. Let $A$ be the square $[-3,-1]\times [-1,1]$ and let $B$ be the square $[1,3]\times [-1,1]$. Then, $\dH^X\left(A,B\right)=4$. Now, let $\gamma:[0,1]\rightarrow \mathcal{H}\left(X\right)$ be the{ Hausdorff geodesic constructed in \Cref{thm:haus-geo-cons}}, i.e., $\gamma(t)=A^{(1-t)\dH^X(A,B)}\cap B^{t\dH^X(A,B)}$ for each $t\in[0,1]$. Then, $\gamma\left(\frac{1}{2}\right)=A^2\cap B^2$. It is easy to see that $\left(0,2\right)\in \gamma\left(\frac{1}{2}\right)$ but there is no geodesic passing through $\left(0,2\right)$ starting in $A$ and ending in $B$. See \Cref{fig:haus-geo} for an illustration.
\end{remark}

\begin{figure}[htb]
	\centering		\includegraphics[width=0.6\textwidth]{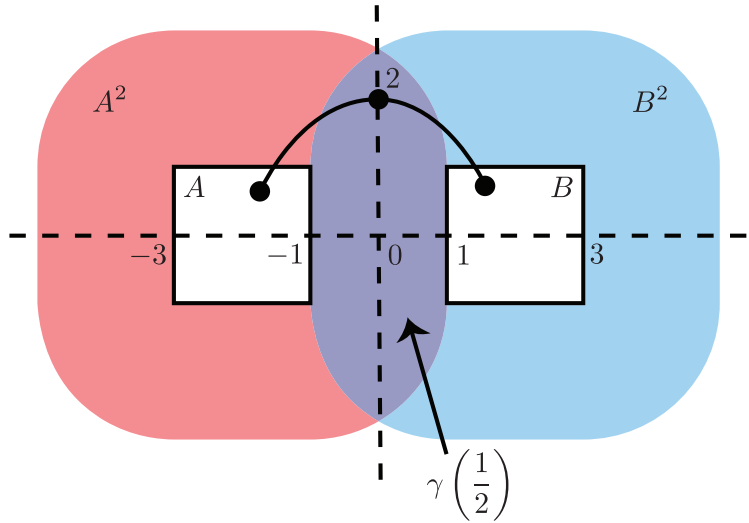}
	\caption{\textbf{Illustration of the example in \Cref{rmk:couterex-haus-geo}.} There is no line segment (geodesic) passing through $\left(0,2\right)$ starting $A$ and ending in $B$. \textbf{(color figure)}} \label{fig:haus-geo}
\end{figure}

\Cref{thm:main-dyn-hausdorff} is a powerful tool which allows us to easily reprove the geodesic property of the Hausdorff hyperspace of a geodesic space (cf. \Cref{thm:hgeo}). In fact, we apply \Cref{thm:main-dyn-hausdorff} to prove a stronger result:

\begin{theorem}\label{thm:hyper-geod-iff}
Suppose $X$ is a compact metric space. Then, $\mathcal{H}\left(X\right)$ is geodesic \emph{if and only if} $X$ is geodesic.
\end{theorem}
\begin{proof}
We first assume that $X$ is geodesic. Let $A,B\in\mathcal{H}(X)$ be two distinct subsets. Then, we have the following observation: 
\begin{claim}\label{clm:nonempty L}
$\mathfrak{L}(A,B)\neq\emptyset$.
\end{claim}
\begin{proof}[Proof of \Cref{clm:nonempty L}]
Let $\rho\coloneqq\dH^X(A,B)$. Since $X$ is compact and $A,B$ are closed, there exists $a\in A$ and $b\in B$ such that $d_X(a,b)=\rho$. Then, there exists a geodesic $\gamma:[0,1]\rightarrow X$ such that $\gamma(0)=a\in A$, $\gamma(1)=b\in B$ and $\gamma$ is $\rho$-Lipschitz. This implies that $\gamma\in\mathfrak{L}(A,B)$ and thus $\mathfrak{L}(A,B)\neq\emptyset$.
\end{proof}
Then by \Cref{thm:main-dyn-hausdorff}, any nonempty closed subset of $\mathfrak{L}(A,B)$ (e.g., $\mathfrak{L}(A,B)$ itself) gives rise to a Hausdorff geodesic connecting $A$ and $B$ in $\mathcal{H}(X)$. Therefore, $\mathcal{H}(X)$ is geodesic.

Now, we assume that $\mathcal{H}(X)$ is geodesic. Then, for any two distinct points $x,y\in X$, there exists a Hausdorff geodesic connecting $\{x\}$ and $\{y\}$. By \Cref{thm:main-dyn-hausdorff}, $\mathfrak{L}(\{x\},\{y\})\neq \emptyset$. Note that $d_X(x,y)=\dH^X(\{x\},\{y\})=:\rho$. Then, any $\rho$-Lipschitz curve in $\mathfrak{L}(\{x\},\{y\})$ is automatically a geodesic connecting $x$ and $y$ which implies that $X$ is geodesic.
\end{proof}

Let $X$ be a compact geodesic metric space. Then, for any two distinct sets $A,B\in\mathcal{H}(X)$, we know by \Cref{clm:nonempty L} that $\mathfrak{L}(A,B)\neq\emptyset$. Then by letting $\mathfrak{D}\coloneqq\mathfrak{L}(A,B)$ and applying \Cref{thm:main-dyn-hausdorff}, we have that the curve $\gamma_{A,B}^X:[0,1]\rightarrow\mathcal{H}(X)$ defined by $t\mapsto e_t(\mathfrak{L}(A,B))$ for all $t\in[0,1]$ is a Hausdorff geodesic. In fact, $\gamma_{A,B}^X$ coincides with the curve constructed in \Cref{thm:haus-geo-cons}, which provides an alternative proof of \Cref{thm:haus-geo-cons}:

\begin{proposition}\label{prop:hausdorff geod from interpolation}
Let $X$ be a compact geodesic metric space. For any two distinct sets $A,B\in\mathcal{H}(X)$, we define $\gamma_{A,B}^X:[0,1]\rightarrow\mathcal{H}(X)$ by $t\mapsto e_t(\mathfrak{L}(A,B))$ for all $t\in[0,1]$. Then, $\gamma_{A,B}^X$ is a Hausdorff geodesic. Moreover, if we let $\rho\coloneqq\dH^X(A,B)$, then for any $t\in[0,1]$, we have that 
\[ e_t(\mathfrak{L}(A,B))=A^{t\rho}\cap B^{(1-t)\rho}.\]
\end{proposition}
\begin{proof}
That $\gamma_{A,B}^X$ is a Hausdorff geodesic has already been proved above.

Since $\gamma_{A,B}^X$ is a Hausdorff geodesic, we have that
\[\dH^X\lc A,\gamma_{A,B}^X(t)\rc=t\rho\,\text{ and }\,\dH^X\lc \gamma_{A,B}^X(t),B\rc=(1-t)\rho.\]
Then, $\gamma_{A,B}^X(t)\subseteq A^{t\rho}$ and $\gamma_{A,B}^X(t)\subseteq B^{(1-t)\rho}$. Therefore,
\begin{equation}\label{eq:et inside cap}
    e_t(\mathfrak{L}(A,B))=\gamma_{A,B}^X(t)\subseteq A^{t\rho}\cap B^{(1-t)\rho}.
\end{equation}

For the other direction, we first observe that \Cref{eq:et inside cap} immediately implies that $A^{t\rho}\cap B^{(1-t)\rho}\neq\emptyset$ and thus $A^{t\rho}\cap B^{(1-t)\rho}\in\mathcal{H}(X)$. Since $\mathcal{H}(X)$ is geodesic, there exist geodesics $\gamma_A,\gamma_B:[0,1]\rightarrow\mathcal{H}(X)$ such that $\gamma_A(0)=A$, $\gamma_A(1)=\gamma_B(0)=A^{t\rho}\cap B^{(1-t)\rho}$ and $\gamma_B(1)=B$. Note that $A^{t\rho}\cap B^{(1-t)\rho}\subseteq A^{t\rho}$ and $A\subseteq \lc \gamma_{A,B}^X(t)\rc^{t\rho}\subseteq\lc A^{t\rho}\cap B^{(1-t)\rho}\rc^{t\rho},$ where $A\subseteq \lc \gamma_{A,B}^X(t)\rc^{t\rho}$ follows from the fact that $\dH^X\lc A, \gamma_{A,B}^X(t)\rc=t\rho$. This implies that $\dH^X\lc A, A^{t\rho}\cap B^{(1-t)\rho}\rc\leq t\rho.$ Similarly, $\dH^X\lc B, A^{t\rho}\cap B^{(1-t)\rho}\rc\leq (1-t)\rho$. By the triangle inequality,
\[\rho=\dH^X(A,B)\leq \dH^X\lc A, A^{t\rho}\cap B^{(1-t)\rho}\rc+\dH^X\lc B, A^{t\rho}\cap B^{(1-t)\rho}\rc\leq t\rho+(1-t)\rho=\rho.\]
Therefore, all equalities must hold.
Then by \Cref{prop:geo-concatenate}, we concatenate $\gamma_A$ and $\gamma_B$ to obtain a geodesic $\gamma:[0,1]\rightarrow\mathcal{H}(X)$ such that $\gamma(0)=A,\gamma(1)=B$ and $\gamma(t)=A^{t\rho}\cap B^{(1-t)\rho}$. By \Cref{thm:main-dyn-hausdorff}, there exists a closed subset $\mathfrak{D}\subseteq\mathfrak{L}(A,B)$ such that ${\gamma}(t)=e_t(\mathfrak{D})$. Then,
\[A^{t\rho}\cap B^{(1-t)\rho}=\gamma(t)=e_t(\mathfrak{D})\subseteq e_t(\mathfrak{L}(A,B)).\]
This concludes the proof.
\end{proof}

The special geodesic $\gamma_{A,B}^X$ constructed above turns out to be instrumental for proving the following interesting property of Gromov-Hausdorff geodesics:

\begin{theorem}[Existence of infinitely many Gromov-Hausdorff geodesics]\label{thm:exist infinite geod}
For any $X,Y\in\ms$ such that $X\not\cong Y$, there exist \emph{infinitely} many distinct Gromov-Hausdorff geodesics connecting $X$ and $Y$.
\end{theorem}
In \cite{chowdhury2018explicit} it is shown that for any $n\in\mathbb{N}$ there exist  infinitely many Gromov-Hausdorff geodesics between the one point space $\Delta_1$ and the $n$-point space $\Delta_n$ . Then, \Cref{thm:exist infinite geod} is a generalization of this result to the case of two arbitrary compact metric spaces. 

\begin{proof}[Proof of \Cref{thm:exist infinite geod}]
By \Cref{lm:dgh_hausdorff-realizable}, there exists $Z_0\in\ms$ and isometric embeddings $\varphi_X^{(0)}:X\hookrightarrow Z$ and $\varphi_Y^{(0)}:Y\hookrightarrow Z$ such that 
\[\dH^{Z_0}\left(\varphi_X^{(0)}(X),\varphi_Y^{(0)}(Y)\right)=\dgh(X,Y).\]
Without loss of generality, we assume that $Z_0$ is geodesic (otherwise we replace $Z_0$ with one  its extension $\mathcal{W}_1\left(Z_0\right)$, which is geodesic by \Cref{thm:W-geodesic}). Consider the following chain of isometric embeddings:
\begin{equation}\label{eq:chain of z embeddings}
    Z_0\hookrightarrow Z_1\hookrightarrow Z_2\hookrightarrow\cdots \hookrightarrow Z_n \hookrightarrow \cdots
\end{equation}
where for each $n\geq 1$, $Z_n\coloneqq\mathcal{W}_1(Z_{n-1})$ and the map $Z_{n-1}\hookrightarrow Z_n=\mathcal{W}_1(Z_{n-1})$ is the canonical embedding sending $z_{n-1}\in Z_{n-1}$ to the Dirac delta measure $\delta_{z_{n-1}}\in \mathcal{P}(Z_{n-1})$. Let $\varphi_X^{(n)}:X\hookrightarrow Z_n$ denote the composition of the map $\varphi_X^{(0)}:X\hookrightarrow Z_0$ with the following composition of canonical isometric embeddings:
\[Z_0\hookrightarrow Z_1\hookrightarrow Z_2\hookrightarrow\cdots \hookrightarrow Z_n.\]
We similarly define $\varphi_Y^{(n)}:Y\hookrightarrow Z_n$. Then, by \Cref{lm:dH under embedding} we have that 

\begin{equation}\label{eq:dh=dgh n}
    \dH^{Z_n}\left(\varphi_X^{(n)}(X),\varphi_Y^{(n)}(Y)\right)=\dH^{Z_0}\left(\varphi_X^{(0)}(X),\varphi_Y^{(0)}(Y)\right)=\dgh(X,Y),\,\,\,\forall n\in\mathbb{N}.
\end{equation}

Following \Cref{eq:chain of z embeddings}, we have the following chain of isometric embeddings:
\begin{equation}\label{eq:chain of l embedding}
    \mathfrak{L}\left(\varphi_X^{(0)}(X),\varphi_Y^{(0)}(Y)\right)\hookrightarrow\mathfrak{L}\left(\varphi_X^{(1)}(X),\varphi_Y^{(1)}(Y)\right)\hookrightarrow\cdots
\end{equation}
By \Cref{prop:hausdorff geod from interpolation} and \Cref{lm:hgeo-to-dghgeo}, for each $n\in\mathbb{N}$, the curve $\gamma_{\varphi_X^{(n)}(X),\varphi_Y^{(n)}(Y)}^{Z_n}:[0,1]\rightarrow \mathcal{H}(Z_n)$ defined by $$\gamma_{\varphi_X^{(n)}(X),\varphi_Y^{(n)}(Y)}^{Z_n}(t)\coloneqq e_t\lc \mathfrak{L}\left(\varphi_X^{(n)}(X),\varphi_Y^{(n)}(Y)\right)\rc$$ for $t\in[0,1]$ is a Gromov-Hausdorff geodesic connecting $\varphi_X^{(n)}(X)\cong X$ and $\varphi_Y^{(n)}(Y)\cong Y$. Now, to conclude the proof, we show that for each $m\neq n\in\mathbb{N}$, $\gamma_{\varphi_X^{(m)}(X),\varphi_Y^{(m)}(Y)}^{Z_m}\neq\gamma_{\varphi_X^{(n)}(X),\varphi_Y^{(n)}(Y)}^{Z_{n}}$.

Note that for any $n\in\mathbb{N}$ and any $t\in(0,1)$, we have the isometric embedding (induced from \Cref{eq:chain of l embedding}):
\[\Psi_t^n:\gamma_{\varphi_X^{(n)}(X),\varphi_Y^{(n)}(Y)}^{Z_n}(t)\hookrightarrow\gamma_{\varphi_X^{(n+1)}(X),\varphi_Y^{(n+1)}(Y)}^{Z_{n+1}}(t).\]
In fact, $\Psi_t^n$ is the restriction of the canonical embedding $Z_n\hookrightarrow Z_{n+1}$ and sends each $z_n\in \gamma_{\varphi_X^{(n)}(X),\varphi_Y^{(n)}(Y)}^{Z_n}(t)$ to $\delta_{z_n}\in \gamma_{\varphi_X^{(n+1)}(X),\varphi_Y^{(n+1)}(Y)}^{Z_{n+1}}(t)$. Let $\rho\coloneqq\dgh(X,Y)$. Then, for any $x\in \varphi_X^{(n)}(X)$ there exists $y\in \varphi_Y^{(n)}(Y)$ such that $d_{Z_n}(x,y)\leq \rho$ (cf. \Cref{eq:dh=dgh n}). By \Cref{lm:W-geodesic}, the curve $s\mapsto (1-s)\delta_x+s\,\delta_y$ is a geodesic in the space $Z_{n+1}=\mathcal{W}_1(Z_n)$ and therefore it is a $\rho$-Lipschitz curve in $\mathfrak{L}\left(\varphi_X^{(n+1)}(X),\varphi_Y^{(n+1)}(Y)\right)$. Then, it is easy to see that for the given $t\in(0,1)$
\[(1-t)\delta_x+t\delta_y\,\,\in \,\,\gamma_{\varphi_X^{(n+1)}(X),\varphi_Y^{(n+1)}(Y)}^{Z_{n+1}}(t)\backslash \Psi_t^n\lc \gamma_{\varphi_X^{(n)}(X),\varphi_Y^{(n)}(Y)}^{Z_n}(t)\rc.\]
Therefore, $\Psi_t^n$ is not surjective. By \cite[Theorem 1.6.14]{burago2001course}, \[\gamma_{\varphi_X^{(n)}(X),\varphi_Y^{(n)}(Y)}^{Z_n}(t)\not\cong\gamma_{\varphi_X^{(n+1)}(X),\varphi_Y^{(n+1)}(Y)}^{Z_{n+1}}(t)\]
since both spaces are compact. Therefore
\[\gamma_{\varphi_X^{(n)}(X),\varphi_Y^{(n)}(Y)}^{Z_n}\neq\gamma_{\varphi_X^{(n+1)}(X),\varphi_Y^{(n+1)}(Y)}^{Z_{n+1}}.\]

Now, for $m<n$ and $t\in(0,1)$, we have an embedding $\Psi_t^{m,n}:\gamma_{\varphi_X^{(m)}(X),\varphi_Y^{(m)}(Y)}^{Z_m}(t)\hookrightarrow\gamma_{\varphi_X^{(n)}(X),\varphi_Y^{(n)}(Y)}^{Z_{n}}(t)$ defined as the following composition of maps:
\[\Psi_t^{m,n}:\gamma_{\varphi_X^{(m)}(X),\varphi_Y^{(m)}(Y)}^{Z_m}(t)\xhookrightarrow[]{\Psi_t^m}\gamma_{\varphi_X^{(m+1)}(X),\varphi_Y^{(m+1)}(Y)}^{Z_{m+1}}(t)\xhookrightarrow[]{\Psi_t^{m+1}}\cdots\xhookrightarrow[]{\Psi_t^{n-1}} \gamma_{\varphi_X^{(n)}(X),\varphi_Y^{(n)}(Y)}^{Z_{n}}(t).\]
Then, it {follows from the above} that $\Psi_t^{m,n}$ is not surjective. Hence 
\[\gamma_{\varphi_X^{(m)}(X),\varphi_Y^{(m)}(Y)}^{Z_m}(t)\neq\gamma_{\varphi_X^{(n)}(X),\varphi_Y^{(n)}(Y)}^{Z_{n}}(t),\]
and thus 
\[\gamma_{\varphi_X^{(m)}(X),\varphi_Y^{(m)}(Y)}^{Z_m}\neq\gamma_{\varphi_X^{(n)}(X),\varphi_Y^{(n)}(Y)}^{Z_{n}} \forall m\neq n\in \mathbb{N}.\]
Therefore, $\left\{ \gamma_{\varphi_X^{(n)}(X),\varphi_Y^{(n)}(Y)}^{Z_{n}}\right\}_{n\in\mathbb{N}}$ is an infinite family of \emph{distinct} Gromov-Hausdorff geodesics connecting $X$ and $Y$.

We summarize the construction of $\left\{ \gamma_{\varphi_X^{(n)}(X),\varphi_Y^{(n)}(Y)}^{Z_{n}}\right\}_{n\in\mathbb{N}}$ above via the following diagram:
$$
\begin{tikzcd}
\centering
     [0,1]\arrow{rr}{=}\arrow[hookrightarrow,"\gamma_{\varphi_X^{(0)}(X),\varphi_Y^{(0)}(Y)}^{Z_{0}}"]{dd} && {[0,1]}  \arrow[rr,"="]\arrow[hookrightarrow,"\gamma_{\varphi_X^{(1)}(X),\varphi_Y^{(1)}(Y)}^{Z_{1}}"]{dd} &&  {[0,1]} \arrow[hookrightarrow,"\gamma_{\varphi_X^{(2)}(X),\varphi_Y^{(2)}(Y)}^{Z_{2}}"]{dd}\arrow{rr}{=}&&  {\cdots}\arrow[hookrightarrow]{dd} \\
      && \,\,\,\,\,\,\,\,\,\, & \\
        \mathcal{H}(Z^0)\arrow[hookrightarrow]{rr} && \mathcal{H}(Z^1)\arrow[hookrightarrow]{rr} && \mathcal{H}(Z^2)\arrow[hookrightarrow]{rr}&&  {\cdots}\\
      && \,\,\,\,\,\,\,\,\,\, & \\
      Z^0\arrow{rr}{\mathcal{W}_1}\arrow[hookrightarrow,"\mathcal{H}"]{uu} && {Z^1}  \arrow[rr,"\mathcal{W}_1"]\arrow[hookrightarrow,"\mathcal{H}"]{uu} &&  {Z^2} \arrow[hookrightarrow,"\mathcal{H}"]{uu}\arrow{rr}{\mathcal{W}_1}&&  {\cdots}\arrow[hookrightarrow,"\mathcal{H}"]{uu}
\end{tikzcd}
$$
\end{proof}

{We end this section by proving \Cref{lm:geo-char-hausdorff}:}

\begin{proof}[Proof of \Cref{lm:geo-char-hausdorff}]
Without loss of generality, we assume that $0<t_*<1$. For $k=0,\ldots$ and $i=0,\ldots, 2^{k+1}$, let
$$t_i^k\coloneqq\begin{cases}\frac{i}{2^k}\cdot t_*,& 0\leq i\leq 2^k\\
t_*+\lc\frac{i}{2^k}-1\rc(1-t_*),& 2^k+1\leq i\leq 2^{k+1}\end{cases}. $$
Let $T^k=\{t_i^k\}_{i=0}^{2^{k+1}}$. Then, $t_*=t_{2^k}^k\in T^k$, $T^k\subseteq T^{k+1}$ and $T\coloneqq\cup_{k=0}^\infty T^k$ is dense in $[0,1]$. Now, for any given $k\in\mathbb{N}$, let $x_{2^k}^k\coloneqq x_{*}$. Then, there exist $x_{{2^k+1}}^k\in\gamma\left(t^k_{2^k+1}\right)$ and $x_{{2^k-1}}^k\in\gamma\left(t^k_{2^k-1}\right)$ such that 
\[d_X\left(x_{{2^k+1}}^k,x_{2^k}^k\right)\leq |{t^k_{2^k+1}}-t_*|\,\rho\text{ and }d_X\left(x_{{2^k-1}}^k,x_{2^k}^k\right)\leq \ls{t^k_{2^k-1}}-t_*\rs\,\rho\] 
since 
\[\dH^X\left(\gamma\left(t^k_{2^k+1}\right),\gamma\left(t_*\right)\right)\leq \ls{t^k_{2^k+1}}-t_*\rs\,\rho\text{ and }\dH^X\left(\gamma\left(t^k_{2^k-1}\right),\gamma\left(t_*\right)\right)\leq \ls{t^k_{2^k-1}}-t_*\rs\,\rho.\] 
Similarly, there exist $x_{{2^k+2}}^k\in\gamma\left(t^k_{2^k+2}\right)$ and $x_{{2^k-2}}^k\in\gamma\left(t^k_{2^k-2}\right)$ such that 
\[d_X\left(x_{{2^k+2}}^k,x_{2^k+1}^k\right)\leq \ls{t^k_{2^k+2}}-t^k_{2^k+1}\rs\,\rho\text{ and }d_X\left(x_{{2^k-2}}^k,x_{2^k-1}^k\right)\leq \ls{t^k_{2^k-2}}-t^k_{2^k-1}\rs\,\rho.\]
Then, in a similar fashion, we inductively construct a sequence of points $\{x_i^k\}_{i=0}^{2^{k+1}}$ such that $x_i^k\in\gamma\left(t_i^k\right)$ and $d_X\left(x_i^k,x_{i+1}^k\right)\leq \ls t_i^k-t^k_{i+1}\rs\,\rho$. In particular, by the triangle inequality, we have $d_X\left(x_i^k,x_j^k\right)\leq \ls t_i^k-t_j^k\rs\,\rho$ for all $i,j=0,\ldots,2^{k+1}$.

Let $Z\coloneqq \mathcal{W}_1\left(X\right)$. Since $X$ is compact, $Z$ is geodesic by \Cref{thm:W-geodesic}. We identify $X$ with its image under the canonical embedding $X\hookrightarrow\mathcal{W}_1(X)=Z$, which then is a closed subset of $Z$. For each $k\in\mathbb{N}$, we interpolate between points $x_i^k$ and $x_{i+1}^k$ by a geodesic in $Z$ for all $i=0,\ldots,2^{k+1}-1$ and concatenate all such geodesics via \Cref{prop:geo-concatenate} to obtain a $\sum_{i=0}^{2^{k+1}-1}d_X\lc x_i^k,x_{i+1}^k\rc$-Lipschitz curve $\zeta_k:[0,1]\rightarrow Z$. In particular, $\zeta_k$ satisfies that $\zeta_k\left(t_i^k\right)=x_i^k$ for all $i=0,\ldots,2^{k+1}$. Moreover,
\[\sum_{i=0}^{2^{k+1}-1}d_X\lc x_i^k,x_{i+1}^k\rc\leq\sum_{i=0}^{2^{k+1}-1}\ls t_i^k-t_{i+1}^k\rs\,\rho=\rho. \]
Then, $\zeta_k$ is $\rho$-Lipschitz. By the Arzel\`a-Ascoli theorem (\Cref{thm:AA}), $\{\zeta_k\}_{k=0}^\infty$ has a subsequence, still denoted by $\{\zeta_k\}_{k=0}^\infty$, uniformly converging to a $\rho$-Lipschitz curve $\zeta:[0,1]\rightarrow Z$. 

Since $\zeta_k(t_*)=x_*$ for all $k$, we have that $\zeta(t_*)=x_*$.  For each $t_i^k$, since it belongs to $T^m$ for all $m\geq k$, we have $\zeta\left(t_i^k\right)=\lim_{m\rightarrow\infty}\zeta_m\left(t_i^k\right)\in X$. Now, it remains to prove that $x_t\in\gamma\left(t\right)$ for each $t\in[0,1]\backslash T$. Let $\{t_{n_k}\}_{k=0}^\infty$ be a subsequence of $T=\cup_k T^k$ converging to $t\in[0,1]\backslash T$. Assume on the contrary that $x_t\not\in\gamma\left(t\right)$ and let $\delta\coloneqq d_X\left(x_t,\gamma\left(t\right)\right)>0$. For $k$ large enough, we have $d_X\left(x_{t_{n_k}},x_t\right)<\frac{\delta}{2}$, which implies that $d_X\left(x_{t_{n_k}},\gamma\left(t\right)\right)>\frac{\delta}{2}$. This contradicts with the fact that $\lim_{k\rightarrow\infty}\dH^X\left(\gamma\left(t_{n_k}\right),\gamma\left(t\right)\right)=0$. This concludes the proof.
\end{proof}

\subsection{Dynamic Gromov-Hausdorff geodesics.}
In this section we extend our results about Hausdorff geodesics in the previous section to the case of Gromov-Hausdorff geodesics.

\begin{definition}[Dynamic correspondence]\label{def:dyn-cor}
For a continuous curve $\gamma:[0,1]\rightarrow\ms$, we call $\mathfrak{R}\subseteq\Pi_{t\in[0,1]}\gamma\left(t\right)$ a \emph{dynamic correspondence} for $\gamma$ if for any $s,t\in[0,1]$, the image of $\mathfrak{R}$ under the evaluation $e_{st}:\Pi_{t\in[0,1]}\gamma\left(t\right)\rightarrow\gamma\left(s\right)\times\gamma\left(t\right)$ taking $(x_t)_{t\in[0,1]}$ to $(x_s,x_t)$ is a correspondence between $\gamma\left(s\right)$ and $\gamma\left(t\right)$. We call $\mathfrak{R}$ a \emph{dynamic optimal correspondence} if each $e_{st}\left(\mathfrak{R}\right)$ is an optimal correspondence between $\gamma\left(s\right)$ and $\gamma\left(t\right)$.
\end{definition}

In either case, we say that $\gamma$ \emph{admits} a dynamic (optimal) correspondence.

\begin{remark}[Comparison between (dynamic) coupling and (dynamic) correspondence]\label{rmk:relation-dyn-coup-corr}
The following observation from optimal transport inspires our definition of dynamic correspondence. Given a compact metric space $X$ and $\alpha,\beta\in\mathcal{P}\left(X\right)$, note that any coupling $\mu\in\mathcal{C}\left(\alpha,\beta\right)$ satisfies $\supp\left(\mu\right)\in\mathcal{R}\left(\supp\left(\alpha\right),\supp\left(\beta\right)\right)$ (cf. \cite[Lemma 2.2]{memoli2011gromov}), i.e., the support of a coupling is a correspondence between supports of measures. Now, for a dynamic coupling $\Pi$ between $\alpha$ and $\beta$, we know $\left(e_s,e_t\right)_\#\Pi$ is a coupling between ${(e_s)}_\#\Pi$ and ${(e_t)}_\#\Pi$ for any $s,t\in[0,1]$. So
$$\supp\left(\left(e_s,e_t\right)_\#\Pi\right)\in\mathcal{R}\big(\supp\left(({e_s})_\#\Pi\right),\supp\left(({e_t})_\#\Pi\right)\big).$$
\end{remark}

\begin{definition}[Dynamic geodesic]\label{def:dyn-geo}
A Gromov-Hausdorff geodesic $\gamma:[0,1]\rightarrow\ms$ is \emph{dynamic}, if it admits a dynamic optimal correspondence $\mathfrak{R}$.
\end{definition}

It is easy to check that straight-line $\dgh$ geodesics are dynamic.
\begin{proposition}
Any straight-line $\dgh$ geodesic is dynamic.
\end{proposition}

\begin{proof}
We adopt the notation from \Cref{thm:str-line-geo}. We first define
$$\Tilde{\mathfrak{R}}:=\left\{\left(\left(x,y\right)\right)_{t\in[0,1]}:\,\left(x,y\right)\in R\right\}\subseteq\Pi_{t\in[0,1]} R.$$

Consider the quotient map $q:\Pi_{[0,1]}R\rightarrow\Pi_{[0,1]}R_t$ taking $\left(\left(x,y\right)\right)_{t\in[0,1]}$ to the tuple $\left(z_t\right)_{t\in[0,1]}$ such that $z_0=x$, $z_1=y$ and $z_t=\left(x,y\right)$ for any $t\in\left(0,1\right)$. Then, $\mathfrak{R}\coloneqq q\left(\Tilde{\mathfrak{R}}\right)$ is a dynamic optimal correspondence for $\gamma_R$.
\end{proof}



Now, we proceed to proving \Cref{thm:main-dyn-gh}. {The proof combines \Cref{thm:main-h-realizable} with \Cref{thm:main-dyn-hausdorff} in a direct way: we transform a Gromov-Hausdorff geodesic into a Hausdorff geodesic, and its Hausdorff displacement interpolation will then generate a dynamic optimal correspondence. }
\thmdyngh*
\begin{proof}
Let $\gamma:[0,1]\rightarrow\ms$ be a Gromov-Hausdorff geodesic and let $\rho\coloneqq\dgh\left(\gamma\left(0\right),\gamma\left(1\right)\right)$. Without loss of generality, assume $\rho>0$. By \Cref{thm:main-h-realizable}, there exist $Z\in\ms$ and isometric embeddings $\varphi_t:\gamma\left(t\right)\hookrightarrow Z$ for $t\in[0,1]$ such that $\dH^Z\big(\varphi_s\left(\gamma\left(s\right)\right),\varphi_t\left(\gamma\left(t\right)\right)\big)=\dgh\left(\gamma\left(s\right),\gamma\left(t\right)\right)$ for $s,t\in[0,1]$. 

Since $t\mapsto\varphi_t(\gamma(t))$ is a Hausdorff geodesic in $Z$, by \Cref{thm:main-dyn-hausdorff} there exists a Hausdorff displacement interpolation $\mathfrak{D}\subseteq\mathfrak{L}\big(\varphi_0(\gamma(0)),\varphi_1(\gamma(1))\big)$ such that $e_t(\mathfrak{D})=\varphi_t(\gamma(t))$ for each $t\in[0,1]$. Now, let 
$$\mathfrak{R}\coloneqq\left\{\left(x_t\right)_{t\in[0,1]}\in\Pi_{t\in[0,1]}\gamma(t):\,\exists\zeta\in\mathfrak{D}\text{ such that }\varphi_t(x_t)=\zeta(t),\,\forall t\in[0,1]\right\}. $$
It is obvious that $\mathfrak{R}\neq\emptyset$ and $e_t(\mathfrak{R})=\gamma(t)$ for any $t\in[0,1]$. Since $e_t=e_{t}\circ e_{st}$ and $e_s=e_{s}\circ e_{st}$ for $s,t\in[0,1]$, we have that $e_{s}(e_{st}(\mathfrak{R}))=\gamma(s)$ and $e_{t}(e_{st}(\mathfrak{R}))=\gamma(t)$ so that $e_{st}(\mathfrak{R})$ is a correspondence between $\gamma(s) $ and $\gamma(t)$. Therefore, $\mathfrak{R}$ is a dynamic correspondence for $\gamma$.

Now, we show that $\mathfrak{R}$ is optimal. In fact, for any $s,t\in[0,1]$,
\begin{align*}
    e_{st}(\mathfrak{R})&=\left\{\left(x_s,x_t\right)\in\gamma\left(s\right)\times \gamma\left(t\right):\,\exists\zeta\in\mathfrak{D}\text{ such that }\varphi_s(x_s)=\zeta(s)\text{ and }\varphi_t(x_t)=\zeta(t)\right\}\\
    &\subseteq\{\left(x_s,x_t\right)\in\gamma\left(s\right)\times \gamma\left(t\right):\,d_Z\left(\varphi_s(x_s),\varphi_t(x_t)\right)\leq |s-t|\,\rho\}=:{R}_{st}.
\end{align*}
For any $\left(x_s,x_t\right),\left(x_s',x_t'\right)\in{R}_{st}$, by identifying $x_t$ with $\varphi_t(x_t)\in Z$, we have
\begin{align*}
    |d_Z\left(x_s,x_s'\right)-d_Z\left(x_t,x_t'\right)|&\leq |d_Z\left(x_s,x_s'\right)-d_Z\left(x_s',x_t\right)|+|d_Z\left(x_s',x_t\right)-d_Z\left(x_t,x_t'\right)|\\
    &\leq d_Z\left(x_s,x_t\right)+d_Z\left(x_s',x_t'\right)\\
    &\leq2|s-t|\,\rho\\
    &=2\dgh\left(\gamma\left(s\right),\gamma\left(t\right)\right).
\end{align*}
Therefore, ${R}_{st}$ is an optimal correspondence between $\gamma\left(s\right)$ and $\gamma\left(t\right)$ and so \emph{must} be $e_{st}(\mathfrak{R})\subseteq{R}_{st}$. Therefore, $\mathfrak{R}$ is an dynamic optimal correspondence for $\gamma$.
\end{proof}

\section{Discussion}\label{sec:discussion}

Gromov-Hausdorff and Gromov-Wasserstein geodesics have been studied in \cite{ivanov2016gromov,chowdhury2018explicit,ivanov2019hausdorff,sturm2006geometry,sturm2012space}. However, those papers did not address the \emph{characterization} of such geodesics. In this paper,  we have proved that not only Gromov-Hausdorff geodesics are actually Hausdorff geodesics but also, in an analogous sense, that a large collection of Gromov-Wasserstein geodesics are Wasserstein geodesics. We further drew structural connections between Hausdorff geodesics and Wasserstein geodesics and studied the dynamic characterization of Gromov-Hausdorff geodesics.

\subsection*{Some open problems}

Besides \Cref{conj:gw} about Wasserstein-realizable Gromov-Wasserstein geodesics, there are other related unsolved problems which we summarize next.

\paragraph{Geodesic hull.} \Cref{lm:union-geo-closed} states that for any $X$-Hausdorff-realizable Gromov-Hausdorff geodesic $\gamma$, the union $\mathcal{G}_X\coloneqq\cup_{t\in[0,1]}\gamma\left(t\right)\subseteq X$ also Hausdorff-realizes $\gamma$. Intuitively speaking, such a $\mathcal{G}_X$ is a ``minimal'' ambient space containing $\gamma$ without any redundant points. It is tempting to call $\mathcal{G}_X$ ``the geodesic hull'' of $\gamma$. In order to make sense of this nomenclature, we need to understand the relation between $\mathcal{G}_X$ and $\mathcal{G}_Y$ whenever $\gamma$ is also $Y$-Hausdorff-realizable for some $Y\in\ms$ which is not isometric to $X$. Is it true that $\mathcal{G}_X\cong\mathcal{G}_Y$ for any such $Y$? Even if this is not the case, it still remains interesting to unravel commonalities between elements of the family $$\mathcal{M}(\gamma):=\{\mathcal{G}_X:\, \gamma \,\mbox{is $X$-Hausdorff-realizable}\}$$ where each $X$ Hausdorff-realizes the given Gromov-Hausdorff geodesic $\gamma$. For example: what is the Gromov-Hausdorff diameter of $\mathcal{M}(\gamma)$?

\paragraph{Beyond geodesics.} 

One of the main insights behind our characterization of Gromov-Hausdorff geodesics is  the observation in \Cref{coro:finite-seq-h-dgh} that any finite collection of compact metric spaces along a given Gromov-Hausdorff geodesic is embeddable into a common ambient space in a way such that pairwise Hausdorff distances agree with the corresponding Gromov-Hausdorff distances.

A natural question  is whether such an ambient space still exists for an arbitrary finite collection of compact metric spaces that dot not necessarily reside along the trace of a Gromov-Hausdorff geodesic. The following conjecture pertains to the simplest unsolved case: the case involving only three spaces.

\begin{conjecture}\label{conj:3-haus}
Given three compact metric spaces $X_1,X_2,X_3\in\ms$, there exist $Z\in\ms$ and isometric embeddings $\varphi_i:X_i\hookrightarrow Z$ such that $\dH^Z\left(\varphi_i\left(X_i\right),\varphi_j\left(X_j\right)\right)=\dgh\left(X_i,X_j\right)$ for all $i,j\in\{1,2,3\}$.
\end{conjecture}

A similar question can be posed in relation to dynamic optimal correspondences:

\begin{conjecture}\label{conj:3-cor}
Given three compact metric spaces $X_1,X_2,X_3\in\ms$, there exists $R\subseteq X_1\times X_2\times X_3$ such that $e_{ij}\left(R\right)$ is an \emph{optimal} correspondence between $X_i$ and $X_j$ for all $i,j\in\{1,2,3\}$.
\end{conjecture}

It seems also interesting to elucidate the relationship between these two conjectures: note that when only dealing with two spaces, in the proof of \Cref{lm:dgh_hausdorff-realizable} we used the existence of an optimal correspondence to establish the existence of an ambient space which realizes the Gromov-Hausdorff distance between the two given spaces; it is then plausible that \Cref{conj:3-cor} could imply \Cref{conj:3-haus}.

\section*{Acknowledgements}
This work was partially supported by the NSF through grants DMS-1723003, CCF-1740761, and CCF-1526513. {We thank Professor Karl-Theodor Sturm for sharing \cite{sturm2020email} with us.} We also thank Qingsong Wang for pointing out the simple counterexample (\Cref{ex:counter}) to \Cref{claim:false claim}.

\paragraph{Declarations of interest:}none

\appendix
\section{Geodesics in Hausdorff hyperspaces}\label{app:geo-hyp}
\subsection{A counterexample}
\begin{example}[Counterexample to \Cref{claim:false claim}]\label{ex:counter}
Consider $X=[0,3]\times [0,1]$ endowed with the usual Euclidean metric $d_X$. We define $\tilde{d}_{X}:X\times X\rightarrow\R$ on $X$ as follows:
$$\tilde{d}_X\left(\left(x,y\right),\left(x',y'\right)\right)\coloneqq \min\left(x+x',d_X\left(\left(x,y\right),\left(x',y'\right)\right),6-(x+x')\right). $$
It is easy to check that $\tilde{d}_X$ is a pseudo-metric on $X$ and that $\tilde{d}_X|_{[1,2]\times[0,1]}=d_{X}|_{[1,2]\times[0,1]}$.

Note that $\tilde{d}_X\left(\left(x,y\right),\left(x',y'\right)\right)=0$ if and only if one of the following three cases holds: $\left(x,y\right)=\left(x',y'\right)$, $x=x'=0$, or $x=x'=3$. This induces an equivalence relation $\sim$ on $X$. In fact, $\tilde{d}_X$ is the quotient metric (\cite[Definition 3.1.12]{burago2001course}) on $X/\sim$ which arises by identifying all points along $\{0\}\times[0,1]$ and all points along $\{3\}\times[0,1]$, respectively. Let $[0]$ denote the equivalence class of all points on $\{0\}\times[0,1]$ and by $[3]$ the equivalence class of all points on $\{3\}\times[0,1]$. Then, it follows from definition of $\tilde{d}_X$ that for any $\left(x,y\right)\in X/\sim$, we have
$$\tilde{d}_X\left(\left(x,y\right),[0]\right)=x\text{ and }\tilde{d}_X\left(\left(x,y\right),[3]\right)=3-x. $$
So, any continuous curve $\gamma:[0,1]\rightarrow X/\sim$ such that $\gamma\left(0\right)=[0]$, $\gamma\left(1\right)=[3]$ and $p_1\left(\gamma\left(t\right)\right)=3t$ satisfies the condition in \Cref{claim:false claim}, where $p_1:\R^2\rightarrow\R$ is the projection onto the first component. For example, the graph of any continuous function defined on $[0,3]$ with value bounded between $0$ and $1$ composed with the quotient $X\rightarrow X/\sim$ generates such a continuous curve $\gamma$. However, for such a curve, the image of $\gamma|_{\left[\frac{1}{3},\frac{2}{3}\right]}$ lies in $[1,2]\times[0,1]\subseteq X/\sim$, where $\tilde{d}_X$ agrees with the Euclidean metric. As long as the image of $\gamma|_{\left[\frac{1}{3},\frac{2}{3}\right]}$ is not straight, $\gamma$ is not a geodesic in $X/\sim$. See \Cref{fig:counterexample} for an illustration. For an explicit counterexample, we let $\gamma$ be the composition of the curve $t\mapsto (3t,\left|\sin({3t}{\pi})\right|)\in X$ for $t\in[0,1]$ and the quotient $X\rightarrow X/\sim$. Then, obviously the image of $\gamma|_{\left[\frac{1}{3},\frac{2}{3}\right]}$ is not straight and thus $\gamma$ is not a geodesic in $X/\sim$.
\end{example}

\begin{figure}[htb]
	\centering		\includegraphics[width=\textwidth]{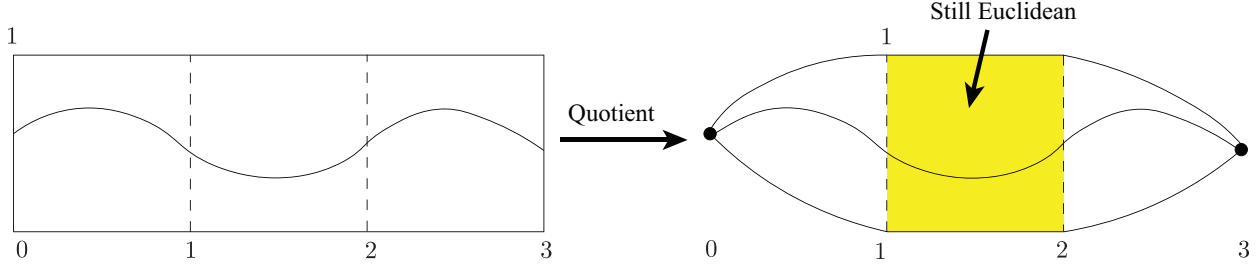}
	\caption{\textbf{Illustration of \Cref{ex:counter}.} Any continuous curve $\gamma:[0,1]\rightarrow X/\sim$ such that $\gamma\left(0\right)=[0]$, $\gamma\left(1\right)=[3]$ and $p_1\left(\gamma\left(t\right)\right)=3t$ satisfies $\tilde{d}_X([0],\gamma(t))=t\,\tilde{d}_X([0],[3])$ and $\tilde{d}_X(\gamma(t),[3])=(1-t)\,\tilde{d}_X([0],[3])$. However, for such an $\gamma$, $\gamma|_{\left[\frac{1}{3},\frac{2}{3}\right]}$ may not be straight (with respect to the Euclidean metric), which will imply that $\gamma$ is not a geodesic in $X/\sim$. \textbf{(color figure)} } \label{fig:counterexample}
\end{figure}

\subsection{Proof of \Cref{thm:haus-geo-cons}}
We now provide a proof for \Cref{thm:haus-geo-cons} which is an amended version of the one given in \cite{serra1998hausdorff}. Recall that for a metric space $X$, a subspace $A$ and $r\geq 0$, $A^r\coloneqq\{x\in X:\,\exists a\in A \text{ such that }d_X\left(a,x\right)\leq r\}$.

\begin{lemma}\label{lm:geo-sum-thick}
Given a \emph{geodesic} metric space $X$ and a closed subset $A\subseteq X$, for any $r_1,r_2\geq0$, we have $(A^{r_1})^{r_2}=A^{r_1+r_2}$.
\end{lemma}

\begin{proof}
The case when ${r_1}=0$ or ${r_2}=0$ is trivial. Now, we assume that ${r_1}>0$ and ${r_2}>0$.

That $(A^{r_1})^{r_2}\subseteq A^{{r_1}+{r_2}}$ holds obviously for any metric space $X$ (not necessarily geodesic). For the other direction, suppose $x\in A^{{r_1}+{r_2}}$ and let $a\in A$ such that $d_X(x,a)\leq {r_1}+{r_2}$. Let $\gamma:[0,1]\rightarrow X$ be a geodesic such that $\gamma(0)=a$ and $\gamma(1)=x$. Let $x_1\coloneqq \gamma\lc\frac{{r_1}}{{r_1}+{r_2}}\rc$. Then, 
$$d_X(a,x_1)=d_X\lc\gamma(0),\gamma\lc\frac{{r_1}}{{r_1}+{r_2}}\rc\rc\leq\frac{{r_1}}{{r_1}+{r_2}}\cdot d_X(\gamma(0),\gamma(1))\leq\frac{{r_1}}{{r_1}+{r_2}}\cdot({r_1}+{r_2})= {r_1}. $$
Similarly, $d_X(x_1,x)\leq {r_2}$. Therefore, $x_1\in A^{r_1}$ and thus $x\in (A^{r_1})^{r_2}$.
\end{proof}

\begin{lemma}\label{lm:geo-int-nonempty}
Given a \emph{geodesic} metric space $X$ and $A,B\in \mathcal{H}(X)$, suppose $\rho\coloneqq\dH^X(A,B)>0$. Then, for any ${r_1},{r_2}\geq 0$ such that ${r_1}+{r_2}=\rho$, we have $A^{r_1}\cap B^{r_2}\neq\emptyset$.
\end{lemma}

\begin{proof}
For any $a\in A$, there exists $b\in B$ such that $d_X(a,b)\leq \rho$. Let $\gamma:[0,1]\rightarrow X$ be a geodesic such that $\gamma(0)=a$ and $\gamma(1)=b$. Let $x_1\coloneqq\gamma\lc\frac{{r_1}}{\rho}\rc$, then it is easy to check as in the previous lemma that $d_X(a,x_1)\leq {r_1}$ and $d_X(x_1,b)\leq {r_2}$. This implies that $x_1\in A^{r_1}\cap B^{r_2}$ and thus $A^{r_1}\cap B^{r_2}\neq\emptyset$.
\end{proof}

\begin{proof}[Proof of \Cref{thm:haus-geo-cons}]
Given $A,B\in \mathcal{H}\left(X\right)$, suppose $\dH^X\left(A,B\right)=\rho.$ Define $\gamma:[0,1]\rightarrow \mathcal{H}\left(X\right)$ by $t\mapsto A^{t\rho}\cap B^{\left(1-t\right)\rho}$. Obviously, $\gamma\left(0\right)=A,\gamma\left(1\right)=B$ and $\gamma\left(t\right)\neq\emptyset$ due to \Cref{lm:geo-int-nonempty}. Now, for any $s<t\in[0,1]$, we need to show that $\dH^X\left(\gamma\left(s\right),\gamma\left(t\right)\right)\leq|s-t|\rho.$

\begin{claim}\label{claim:A1}
{For any $E,F\in \mathcal{H}\left(X\right)$, suppose that for a given $r\geq 0$ we have $E\cap F^r\neq\emptyset$. Then, for any $r_0\in[0,r]$ we have}
$$ E^{r_0}\cap F^{r-r_0}\subseteq\left(E\cap F^r\right)^{r_0}.$$
\end{claim}
\begin{proof}[Proof of \Cref{claim:A1}]
Given $x\in E^{r_0}\cap F^{r-{r_0}}$, there exist $e\in E$ and $f\in F$ such that
$x\in B_s(e)\cap B_{r-{r_0}}(f)\neq\emptyset$. Hence $d_X\left(e,f\right)\leq d_X\left(x,e\right)+d_X\left(x,f\right)\leq {r_0} +r-{r_0}= r$. This implies $e\in E\cap F^r$ and thus $x\in \left(E\cap F^r\right)^{r_0}$.
\end{proof}

Now, let $r_0=\left(t-s\right)\rho,r=\left(1-s\right)\rho$, $E=A^{s\rho}$ and $F=B$, then one has the following from \Cref{lm:geo-sum-thick} and the Claim:
$$A^{t\rho}\cap B^{\left(1-t\right)\rho}=E^{r_0}\cap F^{r-r_0}\subseteq\left(E\cap F^r\right)^{r_0}= \left(A^{s\rho}\cap B^{\left(1-s\right)\rho}\right)^{\left(t-s\right)\rho}. $$ 
Similarly, if we let $r_0=\left(t-s\right)\rho,r=t\rho,E=B^{\left(1-t\right)\rho},F=A$, then we have
$$A^{s\rho}\cap B^{\left(1-s\right)\rho}=F^{r-r_0}\cap E^{r_0} \subseteq\left(F^r\cap E \right)^{r_0}= \left(A^{t\rho}\cap B^{\left(1-t\right)\rho}\right)^{\left(t-s\right)\rho}. $$ 

Therefore $\dH^X\left(\gamma\left(s\right),\gamma\left(t\right)\right)\leq\rho|s-t|.$
\end{proof}

\section{{Additional proofs regarding Wasserstein hyperspaces}}\label{app:ot}

\subsection{Proof of \Cref{thm:W-equal}}
We prove \Cref{thm:W-equal} via the following series of lemmas.

\begin{lemma}\label{lm:proof-lm-1}
For any $X\in\ms$ and a (not necessarily compact) metric space $Y$ and for any $p\in[1,\infty]$, if $\varphi:X\hookrightarrow Y$ is an isometric embedding, then the pushforward map $\varphi_\#:\mathcal{W}_p\left(X\right)\hookrightarrow \mathcal{W}_p\left(Y\right)$ is also an isometric embedding.
\end{lemma}
\begin{proof}
Given $\alpha,\beta\in \mathcal{P}\left(X\right)$, let $\mu\in\mathcal{C}\left(\alpha,\beta\right)$ be any coupling. Consider the pushforward $\tilde{\mu}=\left(\varphi\times\varphi\right)_\#\mu$, where $\varphi\times\varphi:X\times X\rightarrow Y\times Y$ takes $(x,x')$ to $(\varphi(x),\varphi(x'))$. It is easy to check that $\varphi_\#\alpha,\varphi_\#\beta\in\mathcal{P}_p(Y)$ and $\tilde{\mu}\in\mathcal{C}\left(\varphi_\#\alpha,\varphi_\#\beta\right)$. Then, for $p\in[1,\infty)$ we have the following
\begin{align*}
d_{\mathcal{W},p}^Y\left(\varphi_\#\alpha,\varphi_\#\beta\right) & \leq  \left(\int_{Y\times Y}d_Y^p\left(y_1,y_2\right)d\tilde{\mu}\left(y_1,y_2\right) \right) ^\frac{1}{p} \\
 & = \left(\int_{X\times X}d_Y^p\left(\varphi\left(x_1\right),\varphi\left(x_2\right)\right)d\mu\left(x_1,x_2\right)  \right)^\frac{1}{p} \\
 & =  \left(\int_{X\times X}d_X^p\left(x_1,x_2\right)d\mu\left(x_1,x_2\right)  \right)^\frac{1}{p}.
\end{align*}
The last equality follows from the fact that $\varphi$ is an isometric embedding.

For $p=\infty$, since $\varphi$ is an isometric embedding, we have that $(\varphi\times\varphi)\left(\supp\left(\mu\right)\right)=\supp\left(\tilde{\mu}\right).$ Then,
\begin{align*}
d_{\mathcal{W},\infty}^Y\left(\varphi_\#\alpha,\varphi_\#\beta\right) & \leq \sup_{\left(y_1,y_2\right)\in\supp\left(\tilde{\mu}\right)}d_Y\left(y_1,y_2\right)=\sup_{\left(y_1,y_2\right)\in\varphi\times\varphi\left(\supp\left({\mu}\right)\right)}d_Y\left(y_1,y_2\right) \\
 & =\sup_{\left(x_1,x_2\right)\in\supp\left({\mu}\right)}d_Y\left(\varphi\left(x_1\right),\varphi\left(x_2\right)\right)=\sup_{\left(x_1,x_2\right)\in\supp\left({\mu}\right)}d_X\left(x_1,x_2\right).
\end{align*}
By taking infimum over $\mu\in\mathcal{C}(\alpha,\beta)$, we conclude that for $p\in[1,\infty]$, 
$$d_{\mathcal{W},p}^Y\left(\varphi_\#\alpha,\varphi_\#\beta\right)\leq d_{\mathcal{W},p}^X\left(\alpha,\beta\right).$$

Since $\varphi$ is continuous, $\varphi\left(X\right)$ is compact in $Y$ and hence closed. Then, it is obvious to see that $\mathrm{supp}\left(\varphi_\#\alpha\right)\subseteq\varphi\left(X\right)$ for any $\alpha\in \mathcal{P}\left(X\right)$. Hence we have that $$d_{\mathcal{W},p}^Y\left(\varphi_\#\alpha,\varphi_\#\beta\right)= d_{\mathcal{W},p}^{\varphi\left(X\right)}\left(\varphi_\#\alpha,\varphi_\#\beta\right)\leq d_{\mathcal{W},p}^X\left(\alpha,\beta\right).$$
Since $\varphi^{-1}:\varphi\left(X\right)\rightarrow X$ is also an isometric embedding, one has that
$$d_{\mathcal{W},p}^X\left(\alpha,\beta\right)=d_{\mathcal{W},p}^X\left(\varphi^{-1}_\#\circ\varphi_\#\alpha,\varphi^{-1}_\#\circ\varphi_\#\beta\right)\leq d_{\mathcal{W},p}^{\varphi\left(X\right)}\left(\varphi_\#\alpha,\varphi_\#\beta\right).  $$
Therefore, $d_{\mathcal{W},p}^X\left(\alpha,\beta\right)=d_{\mathcal{W},p}^Y\left(\varphi_\#\alpha,\varphi_\#\beta\right)$ and thus $\varphi_\#$ is an isometric embedding.
\end{proof}

\begin{lemma}\label{lm:w<=h}
Given two compact metric spaces $\left(X,d_X\right)$ and $\left(Y,d_Y\right)$, suppose there exist a (not necessarily compact) metric space $\left(Z,d_Z\right)$ and isometric embeddings $\varphi_X:X\hookrightarrow Z$ and $\varphi_Y:Y\hookrightarrow Z$. Then, for any $p\in[1,\infty]$, we have that 
\[\dH^{\mathcal{W}_p\left(Z\right)}\big((\varphi_X)_\#\lc \mathcal{W}_p\left(X\right)\rc,(\varphi_Y)_\#\lc \mathcal{W}_p\left(Y\right)\rc\big)\leq \dH^Z\left(\varphi_X(X),\varphi_Y(Y)\right).\]
\end{lemma}

\begin{proof}
For notational simplicity, we identify $X$ with $\varphi_X\left(X\right)$ and $\mathcal{W}_p\left(X\right)$ with $\left(\varphi_X\right)_\#\left(\mathcal{W}_p\left(X\right)\right)$ (same for $Y$ and $\mathcal{W}_p\left(Y\right)$).

Let $\eta\coloneqq\dH^Z\left(X,Y\right)$, and consider some small $\eps>0$.

\begin{claim}\label{clm:mc}
There exists a Borel measurable map $\xi:X\rightarrow Y$ such that for any $x\in X$, 
$$d_Z\left(x,\xi\left(x\right)\right)\leq\eta+\eps.$$
\end{claim}

\begin{proof}[Proof of \Cref{clm:mc}]
Consider an $\eps$-net $\{x_k\}_{k=1}^n$ of $X$. Then, we construct the Voronoi cells $\{X_k\}_{k=1}^n$ of $X$ with respect to the net, i.e., $X_k\coloneqq\{x\in X:\,d_X\left(x,x_k\right)=\min_{i=1,\ldots,n}d_X\left(x,x_i\right)\}$. We adjust the Voronoi cells to make them disjoint. For example, let $Y_1=X_1$ and $Y_k=X_k\backslash\cup_{i=1}^{k-1}X_i$ for $k\geq 1$. We still let $\{X_k\}_{k=1}^n$ denote the cells after adjustment. Note that these cells are Borel measurable. For each $x_k$, we let $y_k\in Y$ such that $d_Z\left(x_k,y_k\right)\leq\eta$. Then, we define $\xi:X\rightarrow Y$ by mapping $x$ to $y_k$ if $x\in X_k$. The map $\xi$ thus constructed is obviously measurable. Moreover, since $x\in X_k$ and $\{x_k\}_{k=1}^n$ is an $\eps$-net, one has that $d_Z\left(x,x_k\right)=d_X\left(x,x_k\right)\leq\eps$. Therefore
$$d_Z\left(x,\xi\left(x\right)\right)\leq d_Z\left(x,x_k\right)+d_Z\left(x_k,y_k\right)\leq\eps+\eta. $$
\end{proof}

Let $\xi:X\rightarrow Y$ be the map constructed in \Cref{clm:mc}. For any $\alpha\in \mathcal{P}\left(X\right)$, let $\beta\coloneqq\xi_\#\alpha\in \mathcal{P}\left(Y\right)$. For $p\in[1,\infty)$, we have
\[\left(d_{\mathcal{W},p}^Z\left(\alpha,\beta\right)\right)^p\leq\int_X d_Z^p\left(x,\xi\left(x\right)\right)d\alpha\left(x\right)\leq\int_X \left(\eta+\eps\right)^p d\alpha\left(x\right)=\left(\eta+\eps\right)^p.\]

For $p=\infty$, define $\widehat{\xi}:X\rightarrow X\times Y$ by letting $\widehat{\xi}\left(x\right)\coloneqq\left(x,\xi\left(x\right)\right)$ for any $x\in X$. Then, $\mu:=\widehat{\xi}_\#\alpha$ is a coupling between $\alpha$ and $\beta$. Let $C=\overline{\widehat{\xi}\left(X\right)}$, then $\mathrm{supp}\left(\mu\right)\subseteq C$. Indeed, for any $\left(x,y\right)\in (X\times Y)\backslash C$, there exists an open neighborhood $N$ of $\left(x,y\right)$ such that $N\cap C=\emptyset$ and therefore
$$\mu\left(N\right)=\alpha\left(\widehat{\xi}^{-1}N\right)=\alpha\left(\emptyset\right)=0. $$
Now, for any point $\left(x,y\right)\in C$, there exists a sequence $\{\left(x_i,y_i\right)\}_{i=1}^\infty$ in $\widehat{\xi}\left(X\right)$ approaching $\left(x,y\right)$ as $i\rightarrow\infty$. Then, we have that 
$$d_Z\left(x,y\right)=\lim_{i\rightarrow\infty}d_Z\left(x_i,y_i\right)=\lim_{i\rightarrow\infty}d_Z\left(x_i,\xi\left(x_i\right)\right)\leq\eta+\eps. $$
Hence we have that 
$$d_{\mathcal{W},\infty}^Z\left(\alpha,\beta\right)\leq\sup_{\left(x,y\right)\in\mathrm{supp}\left(\mu\right)}d_Z\left(x,y\right)\leq\eta+\eps. $$

Therefore, for any $p\in[1,\infty]$ we have that $\mathcal{W}_p\left(X\right)\subseteq \lc\mathcal{W}_p\left(Y\right)\rc^{\eta+\eps}$, and similarly we have that $\mathcal{W}_p\left(Y\right)\subseteq \lc\mathcal{W}_p\left(X\right)\rc^{\eta+\eps}$. This implies that
$$\dH^{\mathcal{W}_p\left(Z\right)}\left(\mathcal{W}_p\left(X\right),\mathcal{W}_p\left(Y\right)\right)\leq{\eta+\eps}=\dH^Z\left(X,Y\right)+\eps.$$ 
Since $\eps$ is arbitrary, we conclude that 
$$ \dH^{\mathcal{W}_p\left(Z\right)}\left(\mathcal{W}_p\left(X\right),\mathcal{W}_p\left(Y\right)\right)\leq \dH^Z\left(X,Y\right).$$
\end{proof}

\begin{lemma}\label{lm:w>=h}
Under the same assumptions as in \Cref{lm:w<=h}, we have that 
\[\dH^{\mathcal{W}_p\left(Z\right)}\big((\varphi_X)_\#\lc \mathcal{W}_p\left(X\right)\rc,(\varphi_Y)_\#\lc \mathcal{W}_p\left(Y\right)\rc\big)\geq \dH^Z\left(\varphi_X(X),\varphi_Y(Y)\right).\]
\end{lemma}

\begin{proof}
Again for notational simplicity, we will omit the symbols for isometric embedding maps such as $\varphi_X$ or $(\varphi_X)_\#$ in the following proof.

It follows from \Cref{rmk:alternative dH} that $\dH^Z(X,Y)=\max\lc \sup_{x\in X}\inf_{y\in Y}d_Z\left(x,y\right),\sup_{y\in Y}\inf_{x\in X}d_Z\left(x,y\right)\rc$. Without loss of generality, we assume that $\dH^Z(X,Y)=\sup_{y\in Y}\inf_{x\in X}d_Z\left(x,y\right)$. Suppose that the points $x_0\in X,y_0\in Y$ are such that $d_Z\left(x_0,y_0\right)=\sup_{y\in Y}\inf_{x\in X}d_Z\left(x,y\right)$. The existence of such $(x_0,y_0)$ follows from compactness of $X$ and $Y$. Then, we consider the Dirac delta measure $\delta_{y_0}\in \mathcal{P}\left(Y\right)$ and any $\mu_X\in \mathcal{P}\left(X\right)$. We identify $\delta_{y_0}$ with $(\varphi_Y)_\#\delta_{y_0}\in\mathcal{P}(Z)$ and $\mu_X$ with $(\varphi_Y)_\#\mu_X\in\mathcal{P}(Z)$. Then, for any $p\in[1,\infty)$,

$$d_{\mathcal{W},p}^Z\left(\delta_{y_0},\mu_X\right)=\left(\int_Xd_Z\left(x,y_0\right)^pd\mu_X\left(x\right)\right)^{\frac{1}{p}}\geq d_Z\left(x_0,y_0\right). $$

For $p=\infty,$ we have that 
$$d_{\mathcal{W},\infty}^Z\left(\delta_{y_0},\mu_X\right)=\sup_{x\in\supp\left(\mu_X\right)}d_Z\left(x,y_0\right)\geq d_Z\left(x_0,y_0\right). $$

Consequently, we have for all $p\in[1,\infty]$ that
$$\dH^{\mathcal{W}_p\left(Z\right)}\left(\mathcal{W}_p\left(X\right),\mathcal{W}_p\left(Y\right)\right)\geq \inf_{\mu_X\in\mathcal{P}(X)}d_{\mathcal{W},p}^Z\left(\delta_{y_0},\mu_X\right)\geq d_Z(x_0,y_0)= \dH^Z\left(X,Y\right).$$
\end{proof}

Then, \Cref{thm:W-equal} follows from \Cref{lm:proof-lm-1}, \Cref{lm:w<=h} and \Cref{lm:w>=h}.

\subsection{Proof of \Cref{lm:W-geodesic} (linear interpolation)}
\begin{proof}
For any $x_0\in X$ and each $t\in[0,1]$, we have that 
$$\int_Xd_X(x,x_0)\,d\lc\gamma(t)\rc(x)=(1-t)\int_Xd_X(x,x_0)\,d\alpha(x)+t\int_Xd_X(x,x_0)\,d\beta(x)<\infty. $$
Therefore, for each $t\in[0,1]$, $\gamma(t)\in\mathcal{P}_1(X)$.

Since $X$ is Polish, there exists an optimal coupling $\mu$ between $\alpha$ and $\beta$ such that $d_{\mathcal{W},1}\left(\alpha,\beta\right)=\int_{X\times X} d_X\left(x,y\right)\,d\mu(x,y)$ (see for example \cite[Chapter 4]{villani2008optimal}). For $0\leq s< t\leq 1$, let 
$$\mu\left(s,t\right):=\left(1-t\right)\cdot\iota_\#\alpha+s\cdot\iota_\#\beta+\left(t-s\right)\cdot\mu,$$ 
where {$\iota:X\rightarrow X\times X$ taking $x$ to $(x,x)$ is the diagonal map. }

Then, it is easy to show that $\mu\left(s,t\right)\in\mathcal{C}\left(\gamma\left(s\right),\gamma\left(t\right)\right)$. Therefore, we have that 
\begin{align*}
    d_{\mathcal{W},1}\left(\gamma\left(s\right),\gamma\left(t\right)\right)&\leq\int_{X\times X}  d_X\left(x,y\right) d(\mu(s,t))(x,y)\\
    &= \int_{X\times X}  d_X\left(x,y\right) \lc\left(1-t\right)\,d\lc\iota_\#\alpha\rc+s\,d\lc\iota_\#\beta\rc+\left(t-s\right)\,d\mu\rc(x,y)\\
    &=\int_{X} d_X\left(x,x\right)\left(\left(1-t\right)\,d\alpha+s\,d\beta\right)(x)+\left(t-s\right)\,\int_{X\times X}  d_X\left(x,y\right)\,d\mu(x,y)\\
    &=\left(t-s\right)\,d_{\mathcal{W},1}\left(\alpha,\beta\right).
\end{align*}
Hence, by the triangle inequality we obtain that $d_{\mathcal{W},1}\left(\gamma\left(s\right),\gamma\left(t\right)\right)=|t-s|\cdot d_{\mathcal{W},1}\left(\alpha,\beta\right)$ and thus $\gamma$ is a geodesic in $\mathcal{W}_1\left(X\right)$ connecting $\alpha$ and $\beta$. 
\end{proof}
\section{Deviant and branching Gromov-Wasserstein geodesics}\label{app:d-b-geodesic}
For $n\in\mathbb{N}$, denote by $\Delta_n$ the $n$-point space with interpoint distance 1. Now, endow $\Delta_n$ with uniform probability measure (denoted by $\mu_n$) and denote the corresponding metric measure space by $\tilde{\Delta}_n=(\Delta_n,d_n,\mu_n)$. In \cite{chowdhury2018explicit}, the authors constructed an infinite family of deviant Gromov-Hausdorff geodesics between $\Delta_1$ and $\Delta_n$ for $n\geq 2$ and an infinite family of Gromov-Hausdorff geodesics \emph{branching} from the straight-line $\dgh$ geodesic from $\Delta_1$ to $\Delta_n$. In this section, we mimic their constructions and construct deviant and branching $\ell^p$-Gromov-Wasserstein geodesics for $p\in[1,\infty)$.

Since there is only one measure coupling between $\mu_1$ and $\mu_n$, there exists a unique straight-line $\dgws p$ geodesic $\gamma$ that connects $\tilde{\Delta}_1$ and $\tilde{\Delta}_n$ (cf. \Cref{thm:straight-line-gw-geo}). We first write down $\gamma$ explicitly: $\gamma(0)=\tilde{\Delta}_1$ and $\gamma(1)=\tilde{\Delta}_n$; for each $t\in(0,1)$, $\gamma(t)$ has the underlying set $X_n=\{x_1,\ldots,x_n\}$, distance function $d_t$ such that $d_t(x_i,x_j)=t\cdot\delta_{i\neq j}$ for $i,j=1,\ldots,n$ and uniform probability measure $\nu_t$.

Next, we compute explicitly $\dgws{p}\left(\gamma(0),\gamma(1)\right)=\dgws{p}\left(\tilde{\Delta}_1,\tilde{\Delta}_n\right)$ for later use.

\begin{proposition}\label{prop:delta1 deltan}
For each $p\in[1,\infty)$ and any positive integer $n\in\mathbb N$, we have that
$$\dgws{p}\left(\tilde{\Delta}_1,\tilde{\Delta}_n\right)=\frac{1}{2}.$$
\end{proposition}

\begin{proof}
Note that there exists only one coupling $\pi_n$ between $\mu_1$ and $\mu_n$. Let $\{p\}$ and $X_n=\{x_1,\ldots,x_n\}$ be the underlying sets of $\Delta_1$ and $\Delta_n$, respectively. Assume that $d\in\mathcal{D}(d_1,d_n)$, i.e., $d$ is a metric coupling between $d_1$ and $d_n$ (cf. \Cref{rmk:metric coupling}). Then, we have that 
$$\dgws{p}\left(\tilde{\Delta}_1,\tilde{\Delta}_n\right)\leq \lc\frac{1}{n}\sum_{i=1}^n d^p(p,x_i) \rc^\frac{1}{p}.$$
Since $p\geq 1$, by the generalized means inequality (see \cite{bullen2013means}), we have that
\begin{align*}
    \lc\frac{1}{n}\sum_{i=1}^n d^p(p,x_i) \rc^\frac{1}{p}&\geq \frac{1}{n} \sum_{i=1}^nd(p,x_i)=\frac{1}{n}\cdot\frac{1}{n-1}\sum_{i<j}(d(x_i,p)+d(p,x_j)) \\
    &\geq \frac{1}{n(n-1)}\sum_{i<j}d_n(x_i,x_j)=\frac{1}{n(n-1)}\cdot \frac{n(n-1)}{2}\cdot 1=\frac{1}{2}.
\end{align*}
If we let $d(p,x_i)= \frac{1}{2}$ for each $i=1,\ldots,n$, then $d$ is a metric coupling between $d_1$ and $d_n$ satisfying
\[\lc\frac{1}{n}\sum_{i=1}^n d^p(p,x_i) \rc^\frac{1}{p}=\frac{1}{2}.\]
Therefore, we obtain the proposition.
\end{proof}

\subsection{Deviant geodesics} 
In this section, we construct an infinite family of Gromov-Wasserstein geodesics $\{\tilde{\gamma}_\sigma\}_{\sigma\in(0,1]}$ connecting $\tilde{\Delta}_1$ and $\tilde{\Delta}_n$ such that for each $\sigma\in(0,1]$, $\tilde{\gamma}_\sigma$ is different from the straight-line $\dgws p$ geodesic $\gamma$. Since $\gamma$ is the unique straight-line $\dgws p$ geodesic connecting $\tilde{\Delta}_1$ and $\tilde{\Delta}_n$, $\{\tilde{\gamma}_\sigma\}_{\sigma\in(0,1]}$ is an infinite family of deviant geodesics.

For any $\sigma\in(0,1]$ and $t\in[0,1]$, define
$$f(\sigma,t)\coloneqq\begin{cases}t\sigma,&0\leq t\leq \frac{1}{2}\\ \sigma-t\sigma,& \frac{1}{2}<t\leq 1\end{cases}. $$

Let $m$ be an integer such that $1\leq m\leq n$. For each $t\in(0,1)$, let $X_{n+m}^t\coloneqq\{x_1^t,x_2^t,\ldots,x_{n+m}^t\}$ be an $(n+m)$-point set. Choose $\sigma\in(0,1]$ and define for each $t\in(0,1)$ a function $ d_t^\sigma:X_{n+m}^t\times X_{n+m}^t\rightarrow\mathbb{R}_{\geq 0}$ as follows:

$$\forall 1\leq i,j\leq n+m,\quad d_t^\sigma(x_i,x_j)\coloneqq\begin{cases}0,&i=j\\
f(\sigma,t),&|j-i|=n\\
t,&\text{otherwise}\end{cases}. $$
It was proved in \cite[Section 1.1.1]{chowdhury2018explicit} that $(X_{n+m}^t, d_t^\sigma)$ is a metric space when $0<t<1$. Moreover, they proved that
$$\gamma_{\sigma}(t)\coloneqq\begin{cases}\Delta_1, & t=0\\
(X_{n+m}^t,d_t^\sigma), & t\in(0,1)\\
\Delta_n, &t=1
\end{cases}$$
is a Gromov-Hausdorff geodesic connecting $\Delta_1$ and $\Delta_n$. Now, we assign to $\gamma_{\sigma}(t)$ for each $t\in(0,1)$ a probability measure $\nu_t$ as follows: 
$$\nu_t(\{x_i^t\})\coloneqq\begin{cases}\frac{1}{2n},&1\leq i\leq m\text{ or }n+1\leq i\leq n+m\\
\frac{1}{n},&\text{otherwise.}\end{cases} $$

Define a curve $\tilde{\gamma}_{\sigma}(t):[0,1]\rightarrow \ms^w$ of metric measure spaces as follows:
$$\tilde{\gamma}_{\sigma}(t)\coloneqq\begin{cases}\tilde{\Delta}_1 & t=0\\
(X_{n+m}^t,d_t^\sigma,\nu_t) & t\in(0,1)\\
\tilde{\Delta}_n &t=1
\end{cases}.$$

\begin{proposition}
For each $\sigma\in(0,1]$ and each $p\in[1,\infty)$, we have that
\begin{enumerate}
    \item $\tilde{\gamma}_{\sigma}$ is an $\ell^p$-Gromov-Wasserstein geodesic;
    \item for any $t\in(0,1)$, we have that $\tilde{\gamma}_\sigma\lc t\rc\not\cong_w\gamma\lc t\rc$. Given any $\sigma'\in(0,1]$ such that $\sigma\neq\sigma'$, we have that $\tilde{\gamma}_\sigma\lc t\rc\not\cong_w\tilde{\gamma}_{\sigma'}\lc t\rc$;
    \item $\tilde{\gamma}_\sigma$ is Hausdorff-bounded.
\end{enumerate}
\end{proposition}

\begin{proof}
For item 1, we need to prove for any $0\leq s<t\leq 1$ that $\dgws p\lc\tilde{\gamma}_\sigma(s),\tilde{\gamma}_\sigma(t)\rc\leq|t-s|\cdot \dgws{p}(\tilde{\Delta}_1,\tilde{\Delta}_n).$ There are three cases: (i) $s,t\in(0,1)$; (ii) $s=0$ and $t\in(0,1)$; (iii) $s\in(0,1)$ and $t=1$. We only prove case (i) and the rest two cases follow a similar strategy. Let $\pi_{st}\in\mathcal{C}(\nu_s,\nu_t)$ be the identity coupling, i.e., $\pi_{st}=\sum_{i=1}^{n+m}\nu_s(x_i^s)\delta_{(x_i^s,x_i^t)}$.
Define a function $ d_{st}^\sigma:X_{n+m}^s\sqcup X_{n+m}^t\times X_{n+m}^s\sqcup X_{n+m}^t\rightarrow\mathbb R_{\geq 0}$ as follows:
\begin{enumerate}
    \item $ d_{st}^\sigma|_{X_{n+m}^s\times X_{n+m}^s}\coloneqq d_s^\sigma$ and $ d_{st}^\sigma|_{X_{n+m}^t\times X_{n+m}^t}\coloneqq d_t^\sigma$;
    \item for any $x_i^s\in X_{n+m}^s $ and $x_j^t\in X_{n+m}^t $, 
$$d_{st}^\sigma\lc x_i^s,x_j^t\rc\coloneqq\begin{cases}\frac{|t-s|}{2},&i=j\\
\frac{|t-s|}{2}+\min(f(\sigma,s),f(\sigma,t)),&|i-j|=n\\
\frac{s+t}{2},&\text{otherwise.}\end{cases}$$
\item  for any $x_i^s\in X_{n+m}^s $ and $x_j^t\in X_{n+m}^t $, $d_{st}^\sigma\big( x_i^t,x_j^s\big)\coloneqq d_{st}^\sigma\big( x_j^s,x_i^t\big)$.
\end{enumerate}
Then, it is easy to check that $d_{st}^\sigma\in\mathcal{D}( d_s^\sigma, d_t^\sigma)$ (cf. \Cref{rmk:metric coupling}).
Therefore, for any $p\in[1,\infty)$,
\begin{align*}
    \dgws p\lc\tilde{\gamma}_\sigma(s),\tilde{\gamma}_\sigma(t)\rc&\leq\lc\int_{X_{n+m}^s\times X_{n+m}^t}\lc  d_{st}^\sigma\lc x^s,x^t\rc\rc^p\,d\pi_{st}\lc x^s,x^t\rc\rc^\frac{1}{p}\\
    &=\lc \sum_{i=1}^m\lc\frac{|t-s|}{2}\rc^p\cdot\frac{1}{2n}+\sum_{i=m+1}^n\lc\frac{|t-s|}{2}\rc^p\cdot\frac{1}{n}+\sum_{i=n+1}^{n+m}\lc\frac{|t-s|}{2}\rc^p\cdot\frac{1}{2n}\rc^\frac{1}{p}\\
    &=\frac{|t-s|}{2}=|t-s|\cdot \dgws{p}(\tilde{\Delta}_1,\tilde{\Delta}_n),
\end{align*}
where we use \Cref{prop:delta1 deltan} in the last equality. This implies that $\tilde{\gamma}_\sigma$ is an $\ell^p$-Gromov-Wasserstein geodesic connecting $\tilde{\Delta}_1$ and $\tilde{\Delta}_n$.

For item 2, given $t\in(0,1)$ and $\sigma\in(0,1]$, we have that $\tilde{\gamma}_\sigma\lc t\rc$ is a $(n+m)$-point metric measure space (with full support). Since $\gamma\lc\frac{1}{2}\rc$ is an $n$-point space, then $\tilde{\gamma}_\sigma\lc\frac{1}{2}\rc\not\cong_w\gamma\lc\frac{1}{2}\rc$. For $\sigma\neq\sigma'$, the underlying metric spaces of $\tilde{\gamma}_\sigma$ and $\tilde{\gamma}_{\sigma'}$ have different sets of distance values and thus they must not be isometric, let alone isomorphic. Therefore, $\tilde{\gamma}_\sigma\lc t\rc\not\cong_w\tilde{\gamma}_{\sigma'}\lc t\rc$.

For item 3, that $\tilde{\gamma}_\sigma$ is Hausdorff-bounded follows directly from \Cref{prop:f-bdd-finite}.
\end{proof}
{From the above proposition, we conclude that $\gamma$ is different from $\tilde{\gamma}_\sigma$ for all $\sigma\in(0,1]$. Moreover, $\tilde{\gamma}_\sigma\neq\tilde{\gamma}_{\sigma'}$ for all $\sigma,\sigma'\in(0,1]$, $\sigma'\neq \sigma$. Therefore, $\{\tilde{\gamma}_\sigma\}_{\sigma\in(0,1]}$ is an infinite family of deviant Gromov-Wasserstein geodesics connecting $\tilde{\Delta}_1$ and $\tilde{\Delta}_n$.}

\subsection{Branching geodesics} 
We construct an infinite family of $\ell^p$-Gromov-Wasserstein geodesics $\{\gamma^a\}_{a\in(0,1)}$ branching off from the straight-line $\dgws p$ geodesic $\gamma:[0,1]\rightarrow\ws$ from $\tilde{\Delta}_1$ to $\tilde{\Delta}_n$. More precisely, for each $a\in(0,1)$, $\gamma(t)=\gamma^a(t)$ for $t\in[0,a]$ and $\gamma(t)\not\cong_w\gamma^a(t)$ for $t\in(a,1]$. See \Cref{fig:branching} for an illustration.

\begin{figure}[htb]
	\centering		\includegraphics[width=0.4\textwidth]{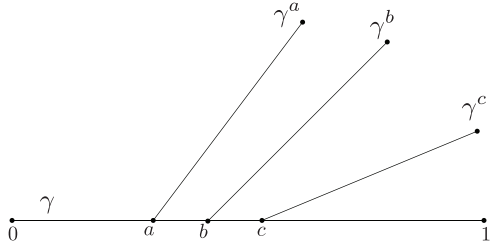}
	\caption{\textbf{Illustration of the family $\{\gamma^a\}_{a\in(0,1)}$}} \label{fig:branching}
\end{figure}

Given $a\in(0,1)$, we define for each $t\in(a,1]$ a metric space as follows: let $X_{n+1}^t\coloneqq\{x_1^t,\ldots,x_n^t,x_{n+1}^t\}$ and let $ d_t^a:X_{n+1}^t\times X_{n+1}^t\rightarrow\mathbb{R}_{\geq 0}$ be such that $ d_t^a\big( x_i^t,x_j^t\big)=t\cdot\delta_{i\neq j}$ for $i,j=1,\ldots,n$, $ d_t^a\big( x_i^t,x_{n+1}^t\big)=t$ for $i=1,\ldots,n-1$ and $ d_t^a\big( x_n^t,x_{n+1}^t\big)=t-a$. That $(X_{n+1}^t, d_t^a)$ is a metric space follows from \cite[Section 1.1.2]{chowdhury2018explicit}. Endow $(X_{n+1}^t, d_t^a)$ with a probability measure $\nu_t$ as follows:
$$\nu_t(\{x_i^t\})\coloneqq\begin{cases}\frac{1}{n},&1\leq i\leq n-1\\
\frac{1}{2n},&i=n\text{ or }i=n+1\end{cases} $$

Now, we define a curve $\gamma^a:[0,1]\rightarrow\ws$ as follows:
$$\gamma^a(t)\coloneqq\begin{cases}
\gamma(t),& t\in[0,a]\\
(X_{n+1}^t, d_t^a,\nu_t),&t\in(a,1]
\end{cases}.$$

\begin{proposition}
For each $a\in(0,1)$ and each $p\in[1,\infty)$, we have that
\begin{enumerate}
    \item $\gamma^a$ is an $\ell^p$-Gromov-Wasserstein geodesic;
    \item $\gamma^a(t)\cong_w\gamma(t)$ for $t\in[0,a]$ whereas $\gamma^a(t)\not\cong_w\gamma(t)$ for $t\in(a,1]$;
    \item $\gamma^a$ is Hausdorff-bounded.
\end{enumerate}

\end{proposition}

\begin{proof}
First of all, using a similar strategy as in proving $\dgws p\big(\tilde{\Delta}_1,\tilde{\Delta}_n\big)=\frac{1}{2}$ (cf. \Cref{prop:delta1 deltan}), one can show that $\dgws p(\gamma^a(0),\gamma^a(1))=\frac{1}{2}$ and we omit details here.

For item 1, we need to prove for any $0\leq s<t\leq 1$ that 
\[\dgws p\lc\gamma^a(s),\gamma^a(t)\rc\leq|t-s|\cdot \dgws{p}(\gamma^a(0),\gamma^a(1)).\]
There are four cases: (i) $s,t\in[0,a]$; (ii) $0\leq s< a<t\leq 1$; (iii) $s=a<t\leq 1$; (iv) $s,t\in(a,1]$. Case (i) follows from the fact that $\gamma$ is a geodesic. Case (ii) follows from case (i) and case (iii). We only prove case (iv) and omit the proof of case (iii) since it follows a  strategy similar to the one used for proving case (iv). 

Fix $a< s<t\leq 1$. Let $\pi_{st}\in\mathcal{C}(\nu_s,\nu_t)$ be the identity coupling, i.e., $\pi_{st}=\sum_{i=1}^{n+1}\nu_s(x_i^s)\delta_{(x_i^s,x_i^t)}$. Define a function $d_{st}^a:X_{n+1}^s\sqcup X_{n+1}^t\times X_{n+1}^s\sqcup X_{n+1}^t\rightarrow\mathbb R_{\geq 0}$ as follows:
\begin{enumerate}
    \item $d_{st}^a|_{X_{n+1}^s\times X_{n+1}^s}\coloneqq d_s^a$ and $d_{st}^a|_{X_{n+1}^t\times X_{n+1}^t}\coloneqq d_t^a$;
    \item for any $x_i^s\in X_{n+1}^s $ and $x_j^t\in X_{n+1}^t $, 
$$d_{st}^a(x_i^s,x_j^t)\coloneqq\begin{cases}\frac{|t-s|}{2},&i=j\\
\frac{t+s}{2}-a,&(i,j)=(n,n+1)\text{ or }(n+1,n)\\
\frac{t+s}{2},&\text{otherwise}\end{cases}.$$
\item  for any $x_i^s\in X_{n+1}^s $ and $x_j^t\in X_{n+1}^t $, $d_{st}^a\lc x_i^t,x_j^s\rc\coloneqq d_{st}^a\lc x_j^s,x_i^t\rc$.
\end{enumerate}
Then, it is easy to check that $d_{st}^a\in\mathcal{D}(d_s^a, d_t^a)$ (cf. \Cref{rmk:metric coupling}). Therefore, for any $p\in[1,\infty)$,
\begin{align*}
    \dgws p\lc{\gamma}^a(s),{\gamma}^a(t)\rc&\leq\lc\int_{X_{n+1}^s\times X_{n+1}^t}\lc d_{st}^a\lc x^s,x^t\rc\rc^p\,d\pi_{st}\lc x^s,x^t\rc\rc^\frac{1}{p}\\
    &=\lc \sum_{i=1}^{n-1}\lc\frac{|t-s|}{2}\rc^p\cdot\frac{1}{n}+\sum_{i=n}^{n+1}\lc\frac{|t-s|}{2}\rc^p\cdot\frac{1}{2n}\rc^\frac{1}{p}\\
    &=\frac{|t-s|}{2}\\
    &=|t-s|\cdot \dgws{p}(\gamma^a(0),\gamma^a(1)).
\end{align*}

For item 2, by definition of $\gamma^a$, $\gamma^a(t)=\gamma(t)$ for $t\in[0,a]$. When $t\in(a,1]$, $\gamma^a(t)$ is an $(n+1)$-point space and as a result, $\gamma^a(t)\not\cong_w\gamma(t)$.

For item 3, that ${\gamma}^a$ is Hausdorff-bounded is a direct consequence of \Cref{prop:f-bdd-finite}.
\end{proof}

Therefore, for each $a\in(0,1)$, ${\gamma}^a$ is an $\ell^p$-Gromov-Wasserstein geodesic branching off from the straight-line $\dgws p$ geodesic $\gamma$ at $t=a$ and thus $\{\gamma^a\}_{a\in(0,1)}$ is an infinite family of Gromov-Wasserstein geodesics branching off from the straight-line $\dgws p$ geodesic $\gamma$.

\bibliography{biblio-wm}

\newcommand{\etalchar}[1]{$^{#1}$}
\begin{thebibliography}{BAMKJ19}

\bibitem[AG13]{ambrosio2013user}
L.~Ambrosio and N.~Gigli.
\newblock A user’s guide to optimal transport.
\newblock In {\em Modelling and optimisation of flows on networks}, pages
  1--155. Springer, 2013.

\bibitem[AMJ18]{alvarez2018gromov}
D.~Alvarez-Melis and T.~Jaakkola.
\newblock {G}romov-{W}asserstein alignment of word embedding spaces.
\newblock In {\em Proceedings of the 2018 Conference on Empirical Methods in
  Natural Language Processing}, pages 1881--1890, 2018.

\bibitem[Ant20]{antonyan2020gromov}
S.A. Antonyan.
\newblock The {G}romov-{H}ausdorff hyperspace of a {E}uclidean space.
\newblock {\em Advances in Mathematics}, 363:106977, 2020.

\bibitem[BALPO18]{bottou2018geometrical}
L.~Bottou, M.~Arjovsky, D.~Lopez-Paz, and M.~Oquab.
\newblock Geometrical insights for implicit generative modeling.
\newblock In {\em Braverman Readings in Machine Learning. Key Ideas from
  Inception to Current State}, pages 229--268. Springer, 2018.

\bibitem[BAMKJ19]{bunne2019learning}
C.~Bunne, D.~Alvarez-Melis, A.~Krause, and S.~Jegelka.
\newblock Learning generative models across incomparable spaces.
\newblock In {\em International Conference on Machine Learning}, pages
  851--861, 2019.

\bibitem[BBI01]{burago2001course}
D.~Burago, Y.~Burago, and S.~Ivanov.
\newblock {\em A course in metric geometry}, volume~33.
\newblock American Mathematical Soc., 2001.

\bibitem[BBK08]{bronstein2008numerical}
A.M. Bronstein, M.M. Bronstein, and R.~Kimmel.
\newblock {\em Numerical geometry of non-rigid shapes}.
\newblock Springer Science \& Business Media, 2008.

\bibitem[BBK{\etalchar{+}}10]{bronstein2010gromov}
A.M. Bronstein, M.M. Bronstein, R.~Kimmel, M.~Mahmoudi, and G.~Sapiro.
\newblock A {G}romov-{H}ausdorff framework with diffusion geometry for
  topologically-robust non-rigid shape matching.
\newblock {\em International Journal of Computer Vision}, 89(2-3):266--286,
  2010.

\bibitem[BGMP14]{blumberg2014robust}
A.J. Blumberg, I.~Gal, M.A. Mandell, and M.~Pancia.
\newblock Robust statistics, hypothesis testing, and confidence intervals for
  persistent homology on metric measure spaces.
\newblock {\em Foundations of Computational Mathematics}, 14(4):745--789, 2014.

\bibitem[BH13]{bridson2013metric}
M.R. Bridson and A.~Haefliger.
\newblock {\em Metric spaces of non-positive curvature}, volume 319.
\newblock Springer Science \& Business Media, 2013.

\bibitem[BL20]{blumberg2020stability}
A.J. Blumberg and M.~Lesnick.
\newblock Stability of 2-parameter persistent homology.
\newblock {\em arXiv preprint arXiv:2010.09628}, 2020.

\bibitem[BMV13]{bullen2013means}
P.S. Bullen, D.S. Mitrinovic, and M.~Vasic.
\newblock {\em Means and their inequalities}, volume~31.
\newblock Springer Science \& Business Media, 2013.

\bibitem[Bry70]{bryant1970convexity}
V.W. Bryant.
\newblock The convexity of the subset space of a metric space.
\newblock {\em Compositio Mathematica}, 22(4):383--385, 1970.

\bibitem[BV92]{berestovskii1992manifolds}
V.N. Berestovskii and A.M. Vershik.
\newblock Manifolds with intrinsic metric, and nonholonomic spaces.
\newblock {\em Representation theory and dynamical systems, Advances in Soviet
  Mathematics}, 9:253--267, 1992.

\bibitem[CCSG{\etalchar{+}}09]{chazal2009gromov}
F.~Chazal, D.~Cohen-Steiner, L.J. Guibas, F.~M{\'e}moli, and S.Y. Oudot.
\newblock Gromov-{H}ausdorff stable signatures for shapes using persistence.
\newblock In {\em Computer Graphics Forum}, volume~28, pages 1393--1403. Wiley
  Online Library, 2009.

\bibitem[CDSO14]{chazal2014persistence}
F.~Chazal, V.~De~Silva, and S.Y. Oudot.
\newblock Persistence stability for geometric complexes.
\newblock {\em Geometriae Dedicata}, 173(1):193--214, 2014.

\bibitem[Cho19]{chowdhury2019metric}
S.~Chowdhury.
\newblock {\em Metric and Topological Approaches to Network Data Analysis}.
\newblock PhD thesis, The Ohio State University, 2019.

\bibitem[CHW20]{cao2020manifold}
K.~Cao, Y.~Hong, and L.~Wan.
\newblock Manifold alignment for heterogeneous single-cell multi-omics data
  integration using pamona.
\newblock {\em bioRxiv}, 2020.

\bibitem[CM18]{chowdhury2018explicit}
S.~Chowdhury and F.~M{\'e}moli.
\newblock Explicit geodesics in {G}romov-{H}ausdorff space.
\newblock {\em Electronic Research Announcements}, 25:48--59, 2018.

\bibitem[CM19]{chowdhury2019gromov}
S.~Chowdhury and F.~M{\'e}moli.
\newblock The {G}romov-{W}asserstein distance between networks and stable
  network invariants.
\newblock {\em Information and Inference: A Journal of the IMA}, 8(4):757--787,
  2019.

\bibitem[CN20a]{chowdhury2020generalize}
S.~Chowdhury and T.~Needham.
\newblock Generalized spectral clustering via {G}romov-{W}asserstein learning.
\newblock {\em arXiv preprint arXiv:2006.04163}, 2020.

\bibitem[CN20b]{chowdhury2020gromov}
S.~Chowdhury and T.~Needham.
\newblock {G}romov-{W}asserstein averaging in a {R}iemannian framework.
\newblock In {\em Proceedings of the IEEE/CVF Conference on Computer Vision and
  Pattern Recognition Workshops}, pages 842--843, 2020.

\bibitem[DSS{\etalchar{+}}20]{demetci2020gromov}
P.~Demetci, R.~Santorella, B.~Sandstede, W.S. Noble, and R.~Singh.
\newblock {G}romov-{W}asserstein optimal transport to align single-cell
  multi-omics data.
\newblock {\em BioRxiv}, 2020.

\bibitem[Edw75]{edwards1975structure}
David~A Edwards.
\newblock The structure of superspace.
\newblock In {\em Studies in topology}, pages 121--133. Elsevier, 1975.

\bibitem[Gro81]{gromov1981groups}
M.~Gromov.
\newblock Groups of polynomial growth and expanding maps (with an appendix by
  {J}acques {T}its).
\newblock {\em Publications Math{\'e}matiques de l'IH{\'E}S}, 53:53--78, 1981.

\bibitem[Hen16]{hendrikson2016using}
R.~Hendrikson.
\newblock {\em Using {G}romov-{W}asserstein distance to explore sets of
  networks}.
\newblock PhD thesis, 2016.

\bibitem[Huh55]{huhunaivsvili1955property}
G.E. Huhunai{\v{s}}vili.
\newblock On a property of {U}rysohn’s universal metric space.
\newblock In {\em Dokl. Akad. Nauk SSSR (NS)}, volume 101, pages 607--610,
  1955.

\bibitem[INT16]{ivanov2016gromov}
A.~Ivanov, N.~Nikolaeva, and A.~Tuzhilin.
\newblock The {G}romov-{H}ausdorff metric on the space of compact metric spaces
  is strictly intrinsic.
\newblock {\em Mathematical Notes}, 100(5):883--885, 2016.

\bibitem[IT17]{ivanov2017local}
A.~Ivanov and A.~Tuzhilin.
\newblock Local structure of {G}romov-{H}ausdorff space around generic finite
  metric spaces.
\newblock {\em Lobachevskii Journal of Mathematics}, 38(6):998--1006, 2017.

\bibitem[IT19a]{ivanov2019hausdorff}
A.~Ivanov and A.~Tuzhilin.
\newblock {H}ausdorff realization of linear geodesics of {G}romov-{H}ausdorff
  space.
\newblock {\em arXiv preprint arXiv:1904.09281}, 2019.

\bibitem[IT19b]{ivanov2019isometry}
A.~Ivanov and A.~Tuzhilin.
\newblock Isometry group of {G}romov-{H}ausdorff space.
\newblock {\em Matematicki Vesnik}, 71(1-2):123--154, 2019.

\bibitem[Kli18]{klibus2018convexity}
D.P. Klibus.
\newblock Convexity of a ball in the {G}romov-{H}ausdorff space.
\newblock {\em Moscow University Mathematics Bulletin}, 73(6):249--253, 2018.

\bibitem[LHY19]{le2019fast}
T.~Le, N.~Ho, and M.~Yamada.
\newblock Fast tree variants of {G}romov-{W}asserstein.
\newblock {\em arXiv preprint arXiv:1910.04462}, 2019.

\bibitem[Lie18]{liebscher2018new}
V.~Liebscher.
\newblock New {G}romov-inspired metrics on phylogenetic tree space.
\newblock {\em Bulletin of mathematical biology}, 80(3):493--518, 2018.

\bibitem[LV09]{lott2009ricci}
J.~Lott and C.~Villani.
\newblock Ricci curvature for metric-measure spaces via optimal transport.
\newblock {\em Annals of Mathematics}, pages 903--991, 2009.

\bibitem[M{\'e}m07]{memoli2007use}
F.~M{\'e}moli.
\newblock {On the use of {G}romov-{H}ausdorff distances for shape comparison}.
\newblock In M.~Botsch, R.~Pajarola, B.~Chen, and M.~Zwicker, editors, {\em
  Eurographics Symposium on Point-Based Graphics}. The Eurographics
  Association, 2007.

\bibitem[M{\'e}m11]{memoli2011gromov}
F.~M{\'e}moli.
\newblock {G}romov-{W}asserstein distances and the metric approach to object
  matching.
\newblock {\em Foundations of computational mathematics}, 11(4):417--487, 2011.

\bibitem[M{\'e}m12]{memoli2012some}
F.~M{\'e}moli.
\newblock Some properties of {G}romov-{H}ausdorff distances.
\newblock {\em Discrete \& Computational Geometry}, 48(2):416--440, 2012.

\bibitem[Mik18]{mikhailov2018hausdorff}
I.A. Mikhailov.
\newblock {H}ausdorff mapping: 1-{L}ipschitz and isometry properties.
\newblock {\em Moscow University Mathematics Bulletin}, 73(6):211--216, 2018.

\bibitem[MO19]{memoli2019quantitative}
F.~M{\'e}moli and O.B. Okutan.
\newblock Quantitative simplification of filtered simplicial complexes.
\newblock {\em Discrete \& Computational Geometry}, pages 1--30, 2019.

\bibitem[MS04]{memoli2004comparing}
F.~M{\'e}moli and G.~Sapiro.
\newblock Comparing point clouds.
\newblock In {\em Proceedings of the 2004 Eurographics/ACM SIGGRAPH symposium
  on Geometry processing}, pages 32--40, 2004.

\bibitem[MS05]{memoli2005theoretical}
F.~M{\'e}moli and G.~Sapiro.
\newblock A theoretical and computational framework for isometry invariant
  recognition of point cloud data.
\newblock {\em Foundations of Computational Mathematics}, 5(3):313--347, 2005.

\bibitem[MSS18]{memoli2018sketching}
F.~M{\'e}moli, A.~Sidiropoulos, and K.~Singhal.
\newblock Sketching and clustering metric measure spaces.
\newblock {\em arXiv preprint arXiv:1801.00551}, 2018.

\bibitem[PAR06]{petersen2006riemannian}
P.~Petersen, S.~Axler, and K.A. Ribet.
\newblock {\em Riemannian geometry}, volume 171.
\newblock Springer, 2006.

\bibitem[PC19]{peyre2019computational}
G.~Peyr{\'e} and M.~Cuturi.
\newblock Computational optimal transport.
\newblock {\em Foundations and Trends{\textregistered} in Machine Learning},
  11(5-6):355--607, 2019.

\bibitem[RS20]{rolle2020stable}
A.~Rolle and L.~Scoccola.
\newblock Stable and consistent density-based clustering.
\newblock {\em arXiv preprint arXiv:2005.09048}, 2020.

\bibitem[Ser98]{serra1998hausdorff}
J.~Serra.
\newblock {H}ausdorff distances and interpolations.
\newblock {\em Computational Imaging and Vision}, 12:107--114, 1998.

\bibitem[Stu06]{sturm2006geometry}
K.T. Sturm.
\newblock On the geometry of metric measure spaces.
\newblock {\em Acta mathematica}, 196(1):65--131, 2006.

\bibitem[Stu12]{sturm2012space}
K.T. Sturm.
\newblock The space of spaces: curvature bounds and gradient flows on the space
  of metric measure spaces.
\newblock {\em arXiv preprint arXiv:1208.0434}, 2012.

\bibitem[Stu20]{sturm2020email}
K.T. Sturm.
\newblock Addendum -- the space of spaces: curvature bounds and gradient flows
  on the space of metric measure spaces.
\newblock Personal communication, 2020.

\bibitem[Ury27]{urysohn1927espace}
P.~Urysohn.
\newblock Sur un espace m{\'e}trique universel.
\newblock {\em Bull. Sci. Math}, 51(2):43--64, 1927.

\bibitem[VCTF19]{titouan2019optimal}
T.~Vayer, N.~Courty, R.~Tavenard, and R.~Flamary.
\newblock Optimal transport for structured data with application on graphs.
\newblock In {\em International Conference on Machine Learning}, pages
  6275--6284, 2019.

\bibitem[VFT{\etalchar{+}}19]{vayer2019sliced}
T.~Vayer, R.~Flamary, R.~Tavenard, L.~Chapel, and N.~Courty.
\newblock Sliced {G}romov-{W}asserstein.
\newblock {\em arXiv preprint arXiv:1905.10124}, 2019.

\bibitem[Vil08]{villani2008optimal}
C.~Villani.
\newblock {\em Optimal transport: old and new}, volume 338.
\newblock Springer Science \& Business Media, 2008.

\bibitem[XLC19]{xu2019scalable}
H.~Xu, D.~Luo, and L.~Carin.
\newblock Scalable {G}romov-{W}asserstein learning for graph partitioning and
  matching.
\newblock In {\em Advances in neural information processing systems}, pages
  3052--3062, 2019.

\end{thebibliography}
\bibliographystyle{alpha}

\end{document}